\documentclass[reqno]{amsart}

\usepackage[utf8]{inputenc}
\usepackage[T1]{fontenc}
\usepackage{fullpage}
\usepackage{color}
\usepackage{anyfontsize}

\usepackage{newtxtext}
\usepackage{mathrsfs}
\usepackage{stmaryrd}
\usepackage{amssymb}
\usepackage{amsfonts}
\usepackage{bbm}
\usepackage{spectralsequences}

\usepackage{mathtools, amsmath, amsthm, amsopn}
\usepackage{thm-restate}
\usepackage{enumitem}
\usepackage[new]{old-arrows}

\usepackage{tikz-cd}
\usetikzlibrary{cd}
\tikzset{
  symbol/.style={
    draw=none,
    every to/.append style={
      edge node={node [sloped, allow upside down, auto=false]{$#1$}}}
  }
}
\usepackage[all]{xy}

\usepackage[
  backend=biber,
  style=alphabetic,
  sorting=nyt,
  giveninits=true,
  maxcitenames=4,
  maxbibnames=4,
  maxnames=4,
  maxalphanames=4
]{biblatex}
\DeclareFieldFormat[misc]{title}{\mkbibemph{#1}}
\DeclareFieldFormat[article]{title}{\mkbibemph{#1}}
\DeclareFieldFormat[article]{journaltitle}{#1}
\renewbibmacro{in:}{}

\usepackage{csquotes}
\usepackage{xcolor}
\definecolor{darkgreen}{rgb}{0,0.8,0}
\definecolor{darkblue}{rgb}{0,0,0.8}
\usepackage{hyperref}
\hypersetup{
  colorlinks=true,
  linktoc=all,
  linkcolor=darkblue,
  citecolor=darkgreen
}
\usepackage{cleveref}
\crefname{introthm}{Theorem}{Theorems}
\Crefname{introthm}{Theorem}{Theorems}
\crefrangelabelformat{introthm}{#3#1#4--#5#2#6}
\usepackage{footmisc}

\usepackage[colorinlistoftodos]{todonotes}

\declaretheorem[name=Theorem,numberwithin=section]{thm}

\theoremstyle{plain}
\newtheorem{lem}[thm]{Lemma}
\newtheorem*{keylem}{Key Lemma}
\newtheorem{prop}[thm]{Proposition}
\newtheorem{introthm}{Theorem}
\newtheorem*{introcor}{Corollary}
\newtheorem{cor}[thm]{Corollary}

\newtheorem*{thm*}{Theorem}

\theoremstyle{definition}

\newtheorem{df}[thm]{Definition}
\newtheorem{exmp}[thm]{Example}

\theoremstyle{remark}
\newtheorem{rem}[thm]{Remark}

\DeclareMathOperator{\spec}{Spec}

\DeclareMathOperator{\pr}{pr}

\DeclareMathOperator{\id}{id}
\DeclareMathOperator{\Frac}{Frac}
\DeclareMathOperator{\ev}{ev}
\DeclareMathOperator{\EV}{Ev}

\DeclareMathOperator{\Hom}{Hom}

\newcommand{\C}{\mathbb{C}}
\newcommand{\Cx}{\C^\times}
\newcommand{\Cxd}{\C^\times_{\mathrm{dil}}}
\newcommand{\Cxh}{\C^\times_\hbar}

\newcommand{\Z}{\mathbb{Z}}

\newcommand{\PP}{\mathbb{P}}

\newcommand{\pt}{\mathrm{pt}}

\newcommand{\Gr}{\mathrm{Gr}}
\newcommand{\oh}{\mathcal{O}}
\newcommand{\ok}{\mathcal{K}}
\newcommand{\lt}{\mathfrak{t}}

\newcommand{\bN}{\mathbf{N}}
\newcommand{\bV}{\mathbf{V}}

\newcommand{\BR}{\mathcal{R}}
\newcommand{\BT}{\mathcal{T}}
\newcommand{\BS}{\mathcal{S}}

\newcommand{\BI}{\mathcal{I}}
\newcommand{\sE}{\mathcal{E}}
\newcommand{\sA}{\mathcal{A}}
\newcommand{\sB}{\mathcal{B}}
\newcommand{\cX}{\mathcal{X}}
\newcommand{\cM}{\mathcal{M}}
\newcommand{\cN}{\mathcal{N}}
\newcommand{\cZ}{\mathcal{Z}}

\newcommand{\hG}{\widehat{G}}

\newcommand{\hT}{\widehat{T}}

\newcommand{\Eff}{\mathrm{Eff}}
\newcommand{\loc}{\mathrm{loc}}
\newcommand{\vir}{\mathrm{vir}}
\newcommand{\bbS}{\mathbb{S}}
\newcommand{\cct}{(\!(t)\!)}
\newcommand{\can}{\mathrm{can}}
\newcommand{\PD}{\mathrm{PD}}
\newcommand{\tw}{\mathrm{tw}}
\newcommand{\Alg}{\mathrm{-Alg}}
\newcommand{\Set}{\mathrm{Set}}

\newcommand{\Iri}{\mathrm{Iri}}

\DeclareFontFamily{U}{mathx}{}
\DeclareFontShape{U}{mathx}{m}{n}{<-> mathx10}{}
\DeclareSymbolFont{mathx}{U}{mathx}{m}{n}
\DeclareMathAccent{\widehat}{0}{mathx}{"70}
\DeclareMathAccent{\widecheck}{0}{mathx}{"71}

\title[Quantum cohomology, shift operators, and Coulomb branches]{Quantum cohomology, shift operators, and Coulomb branches}
\author[K. F. Chan, K. Chan, C. H. E. Lam]{Ki Fung Chan, Kwokwai Chan, Chin Hang Eddie Lam}
\address{The Institute of Mathematical Sciences and Department of Mathematics, The Chinese University of Hong Kong, Shatin, Hong Kong}
\email{kfchan@math.cuhk.edu.hk}
\address{Department of Mathematics, The Chinese University of Hong Kong, Shatin, Hong Kong}
\email{kwchan@math.cuhk.edu.hk}
\address{Department of Mathematics, The Chinese University of Hong Kong, Shatin, Hong Kong}
\email{echlam@math.cuhk.edu.hk}
\begin{document}

\begin{abstract}
Given a complex reductive group $G$ and a $G$-representation $\bN$, there is an associated Coulomb branch algebra $\sA_{G,\bN}^\hbar$ defined in~\cite{Nak,BFN}. In this paper, we prove a new characterization of $\sA_{G,\bN}^\hbar$ as the largest subspace of the equivariant Borel--Moore homology of the affine Grassmannian on which shift operators (and their deformations induced by flavour symmetries) are regular, meaning that they are defined without localizations. 
The proofs involve showing that the defining equations of the Coulomb branch algebras precisely reflect properness of the moduli spaces used to define shift operators.
As a main application, we 
show that if $X$ is a smooth semiprojective variety equipped with a $G$-action, and $f \colon X \to \bN$ is a $G$-equivariant proper holomorphic map, then the equivariant big quantum cohomology $QH^\bullet_G(X)$ defines a family of closed Lagrangians in the Coulomb branch $\spec \sA_{G,\bN}$, yielding a transformation of 3d branes in 3d mirror symmetry.
Regularity of shift operators also gives way to highly efficient computations in equivariant Gromov--Witten theory; in particular, we obtain a very short proof of the Peterson isomorphism.
\end{abstract}
\maketitle


\section*{Introduction}
\subsection*{Background and the main results}

Let $G$ be a connected complex reductive group. For each representation $\bN$ of $G$, one can associate a 3d~$\mathcal{N}=4$ supersymmetric gauge theory with matter $T^*\bN$. The Coulomb branch $\sA^\hbar_{G,\bN}$ of this theory was constructed by \citeauthor{BFN}~\cite{BFN} using representation-theoretic methods. 
This paper concerns the interaction between \emph{quantum cohomology}~\cite{Witten,Kontsevich-Manin,Cox-Katz}, \emph{shift operators}~\cite{Seidel,OP,Iritani}, and \emph{Coulomb branches}~\cite{Nak,BFN,BDGH,2drole} arising from 3d~$\mathcal{N}=4$ supersymmetric gauge theories.

Our main results can be succinctly summarized by the following slogan:
\begin{center}
\emph{Coulomb branches are the maximal domains of regularity for shift operators in quantum cohomology.}
\end{center}
A shift operator is called \emph{regular} if its action on quantum cohomology is defined without localization. The slogan means that the Coulomb branch intrinsically encodes the enumerative information carried by regular shift operators, hence the title of this paper. 
This perspective also suggests that geometric structures on Coulomb branches admit categorical interpretations at the level of quantum cohomology and shift operators. 
We will provide a more precise formulation after recalling the relevant background on quantum cohomology and shift operators.

\subsubsection*{Shift operators}
Let $T\subset G$ be a maximal torus. Throughout this paper, $X$ denotes a smooth semiprojective variety (see \Cref{qh}) equipped with a $G$-action. Shift operators are endomorphisms of equivariant quantum cohomology that play an important role in symplectic topology, mirror symmetry, and representation theory~\cite{Seidel, OP, MO, Iritani}.

Suppose, for the moment, that the $T$-fixed locus $X^T$ is proper. Then for each cocharacter $\lambda: \mathbb{C}^\times \to T$, there exists a shift operator
\[
\bbS_{\lambda} \in \operatorname{End}^\bullet(H^\bullet_{T}(X)[\hbar])_\loc[[q,\tau]];
\]
here $\hbar$ is the parameter of loop rotation, while $q$ and $\tau$ are the Novikov and bulk parameters, respectively, appearing in the theory of quantum connections. We will review this notation in \Cref{qh}. Shift operators are defined using genus zero Gromov--Witten invariants of certain $X$-bundles over $\PP^1$. Since $X$ is not assumed to be compact, these Gromov--Witten invariants are generally defined using equivariant localization. The subscript ``$\loc$'' means that we localize at the multiplicative set of nonzero homogeneous elements of $\C[\lt][\hbar]$, where $\lt$ is the Lie algebra of $T$ and we identify $H_T^\bullet(\pt)=\C[\lt]$.

These shift operators satisfy the relations
\begin{align}
\bbS_{\lambda}\circ\bbS_{\lambda'} &= \bbS_{\lambda+\lambda'}, \label{Shiftrelations1} \\
\bbS_{\lambda}a -a\bbS_{\lambda}&= a(\lambda)\hbar\bbS_{\lambda}, \label{Shiftrelations2}
\end{align}
for any cocharacters $\lambda,\lambda':\Cx\to T$ and $a\in \lt^*$. For each cocharacter $\lambda$ of $T$, we write $t^\lambda$ for the corresponding character of the dual torus $\check T$. Let $\oh^\hbar_{T^*\check T}$ denote the quantization of the ring of regular functions on $T^*\check T=\check T\times \lt$. As a vector space, we have $\oh^\hbar_{T^*\check T}=\C[T^*\check T][\hbar]$, but the product is deformed so that $t^\lambda a -at^\lambda=a(\lambda)\hbar t^\lambda$.
Thus the relations \eqref{Shiftrelations1} and \eqref{Shiftrelations2} can be summarized by saying that there is a graded ring homomorphism
\begin{equation}\label{eq:rational_shift_operators}
    \bbS_{G,X}:(\oh_{T^*\check T}^\hbar)_\loc \to \operatorname{End}^\bullet(H^\bullet_{T}(X)[\hbar])_\loc[[q,\tau]],
\end{equation}
where $\loc$ is understood as above. We may interpret $(\oh_{T^*\check T}^\hbar)_\loc$ as the largest space for which localized shift operators are defined. 

Given a representation $\bN$ of $G$, the Coulomb branch algebra $\sA^\hbar_{G,\bN}$ is a subspace of $(\oh_{T^*\check T}^\hbar)_\loc$.
\begin{introthm}\label{IntroThm2}
Given any equivariant proper morphism $f: X \to \bN$, there exists a graded homomorphism
\begin{equation}\label{Introshiftoperators3}
    \bbS_{G,\bN,X} : \mathcal{A}_{G,\bN}^{\hbar} \to \operatorname{End}^\bullet(H^\bullet_{G}(X)[\hbar])[[q,\tau]],
\end{equation}
that commutes with the quantum connection.
Moreover, if the $T$-fixed locus $X^T$ is proper, then $\bbS_{G,\bN,X}$ agrees with the restriction of $\bbS_{G,X}$ in \eqref{eq:rational_shift_operators}.
\end{introthm}
An endomorphism in $\operatorname{End}^\bullet(H^\bullet_{T}(X)[\hbar])_\loc[[q,\tau]]$ is said to be \emph{regular} if it lies in the nonlocalized subspace $\operatorname{End}^\bullet(H^\bullet_{G}(X)[\hbar])[[q,\tau]]$. Thus, \Cref{IntroThm2} says that every element of $\sA^\hbar_{G,\bN}$ defines a regular shift operator.

Let us spell out a semiclassical version of \Cref{IntroThm2}. We write $QH^\bullet_G(X)$ for the \emph{big} quantum cohomology ring of $X$, whose underlying vector space is $H^\bullet_G(X)[[q,\tau]]$; the relevant definitions will be reviewed in \Cref{qh}.
We set $\sA_{G,\bN} := \sA_{G,\bN}^\hbar|_{\hbar=0}$. This is a commutative algebra, and the affine scheme $\spec \sA_{G,\bN}$ carries a natural Poisson structure whose restriction to the smooth locus is symplectic.
\begin{introthm}\label{IntroThm1}
The shift operator $\bbS_{G,\bN,X}$ induces a graded ring homomorphism, called the \emph{(generalized) Seidel homomorphism}
\begin{equation}\label{Introshiftoperators5}
    \Psi_{G,\bN,X} \colon \sA_{G,\bN} \longrightarrow QH^\bullet_G(X)
\end{equation}
from the Coulomb branch algebra $\sA_{G,\bN}$ to the $G$-equivariant big quantum cohomology of $X$.
Moreover, the image of the induced morphism
\[
\spec QH^\bullet_G(X)\to \spec \sA_{G,\bN}
\]
is a family of closed Lagrangian subvarieties in $\spec \sA_{G,\bN}$ parametrized by the Novikov and bulk parameters.
\end{introthm}
The shift operator $\bbS_{G,\bN,X}$ can therefore be regarded as a \emph{geometric quantization} of the family of Lagrangian subvarieties $\spec QH^\bullet_G(X)\to \spec \sA_{G,\bN}$. 

From the perspective of 3d mirror symmetry, \Cref{IntroThm1} can be interpreted as describing a mirror correspondence of boundary conditions between the 3d theories associated with the Higgs and Coulomb branches. The equivariant morphism $f: X \to \bN$ should be regarded as a boundary condition associated to the Higgs branch $[T^*\bN/\!\!/G]$. Under 3d mirror symmetry, $\spec QH^\bullet_G(X)$ should be regarded as the support of the \emph{mirror} boundary condition on the Coulomb branch $\spec \sA_{G,\bN}$. We remark that regularity of $\bbS_{G,\bN,X}$ is needed in order to prove the Lagrangian property of $\spec QH^\bullet_G(X)$.

On the other hand, the regularity of equivariant Gromov--Witten invariants is highly effective in concrete computations. For example, it allows us to give a very short proof of a slight generalization of the Peterson isomorphism in \Cref{Peterson}. In the sequel \cite{Cotangentofflag}, we apply this regularity result to compute the action of shift operators on the stable envelopes of cotangent bundles of partial flag varieties.

More importantly, Coulomb branch algebras are in fact completely characterized by the regularity of shift operators, as we now explain.

\subsubsection*{Flavour symmetries and deformations}
Suppose that the $G$-action on $\bN$ extends to a group $\hG$ fitting into a short exact sequence
\[
\xi: 1\to G\to \hG\to T_\xi\to 1,
\]
where $T_\xi$ is a torus. We call $\xi$ a \emph{flavour symmetry}, and write $\hT\supset T$ for a maximal torus of $\hG$. We write
\[
\oh_{T^*\check T}^{\hbar,\xi}
=
\C[\hat\lt]\otimes_{\C[\lt]}\oh_{T^*\check T}^{\hbar},
\]
which carries an induced algebra structure from $\oh_{T^*\check T}^{\hbar}$; see \Cref{sectionnewcharacterization}. If, in addition, the morphism $f:X\to\bN$ is $\hT$-equivariant, then one can define the flavour-deformed shift operators
\[
    \bbS^\xi_{G,X}:(\oh_{T^*\check T}^{\hbar,\xi})_\loc
    \to
    \operatorname{End}^\bullet(H^\bullet_{\hT}(X)[\hbar])_\loc[[q,\tau]].
\]
We are now ready to give the promised characterization of the Coulomb branch algebra $\sA^\hbar_{G,\bN}$.
\begin{introthm}\label{Introlargestsubalg}
If $a \in (\oh_{T^*\check T}^\hbar)_\loc\setminus \sA^\hbar_{G,\bN}$, then there exist a flavour symmetry $\xi$ and a smooth semiprojective variety $X$ equipped with a $\hG$-equivariant proper morphism $f: X \to \bN$, such that, for any lift $\tilde a \in (\oh_{T^*\check T}^{\hbar,\xi})_\loc$ of $a$ (see \Cref{df:lift}), the deformed shift operator
\[
\bbS_{G,\bN, X}^\xi(\widetilde a)\in \operatorname{End}^\bullet(H_{\hG\times \Cx}^\bullet(X)[\hbar])_\loc[[q,\tau]]
\]
is not regular.
\end{introthm}
The desired characterization is an immediate consequence of \Cref{IntroThm2,Introlargestsubalg}:
\begin{introcor}\label{IntroCorollary}
The Coulomb branch algebra $\sA^\hbar_{G,\bN}$ is the \emph{largest} linear subspace of $(\oh_{T^*\check T}^\hbar)_\loc$ on which the shift operators and their flavour-deformations are regular. 
\end{introcor}
In this paper, we also consider flavour-deformed Coulomb branches $\sA^{\hbar,\xi}_{G,\bN}$; the corresponding analogues of \Cref{IntroThm2,IntroThm1,Introlargestsubalg} and the Corollary remain valid in this setting, as explained in \Cref{sectionnewcharacterization}.

\subsubsection*{Categorification of properties of the Coulomb branch}
The Corollary gives a characterization of the Coulomb branch in terms of the regularity of shift operators. This suggests that the geometry of Coulomb branches can be studied through shift operators. Physically, our characterization says that boundary conditions on the Coulomb branch determine the Coulomb branch itself. It is therefore natural to ask whether structural properties of Coulomb branches admit categorified counterparts, namely, whether they can be understood in terms of boundary conditions and shift operators.

For instance, the Corollary characterizes the Coulomb branch $\sA_{G,\bN}^\hbar$ only as a linear subspace of $(\oh_{T^*\check T}^\hbar)_\loc$. However, the fact that $\sA_{G,\bN}^\hbar$ is a subalgebra of $(\oh_{T^*\check T}^\hbar)_\loc$ follows from, and should be regarded as the decategorified manifestation of, the fact that the composition of two regular shift operators is again regular.
We expect that most of the properties of Coulomb branches can be understood in this way. Further examples will be given in \Cref{subsection:properties_and_categorifications}. Some properties of Coulomb branches and their categorified counterparts in terms of shift operators are listed in the following table.

\begin{center}
\renewcommand{\arraystretch}{1.35}
\begin{tabular}{p{0.42\textwidth}|p{0.50\textwidth}}
\textbf{Property of Coulomb branches}
&
\textbf{Categorified counterpart for shift operators}
\\
\hline

Algebra structure
&
The composition of regular shift operators is regular.
\\[0.5em]

Coproduct structure
&
The K\"unneth formula for shift operators.
\\[0.5em]

Flavour symmetries deform the Coulomb branch.
&
Flavour symmetries deform the shift operators.
\\[0.5em]

The Weyl group of the flavour symmetry acts on the deformed Coulomb branch.
&
The Weyl group acts on Cartan-equivariant quantum cohomology and commutes with the shift operators.
\\[0.5em]

The torus $T^!=H^2_G(\pt;\C)/H^2_G(\pt;\Z)$ acts on the Coulomb branch.
&
The torus $T^!$ deforms the shift operators.
\end{tabular}
\end{center}

\subsection*{The construction of \texorpdfstring{$\bbS_{G,\bN,X}$}{S\_G,N,X}}
We now explain the construction of the shift operators $\bbS_{G,\bN,X}$ in \Cref{IntroThm2} using \citeauthor{BFN}'s definition of Coulomb branches \cite{BFN}. A comparison with other works will be given at the end of this introduction.

\subsubsection*{Equivariant Novikov variables}
Given $G$ and $X$, we let $\tau\in H_{G\times\Cxh}^\bullet(X)$ be a general point, treated as a formal variable. We define $\mathbb{C}[[q,\tau]]$ to be a formal completion of $\mathbb{C}[[\tau]][H_2^{\mathrm{ord},G}(X;\mathbb{Z})]$ along the cone of effective curve classes (see \Cref{qh}).\footnote{We use $H_\bullet$ to denote Borel--Moore homology and $H_\bullet^{\mathrm{ord}}$ to denote ordinary (i.e., singular) homology.} 
In this paper, the equivariant quantum cohomology ring has underlying vector space given by a completed tensor product
\begin{equation*}
    H^\bullet_G(X)[[q,\tau]] \coloneqq H^\bullet_G(X) {\widehat{\otimes}}\, \mathbb{C}[[q,\tau]].
\end{equation*}
The quantum product is defined using equivariant genus-zero Gromov--Witten invariants.

\subsubsection*{The case $\bN = \mathbf{0}$}

We begin with the construction in the pure gauge case. Note that when $\bN=0$, the variety $X$ is projective. 

The first step involves a convolution-type operation. We define a map (see \Cref{Twistingdef})
\[
\tw := (\pi_R)_* \circ (\pi_L)^*: H_\bullet^{G_\mathcal{O} \rtimes \mathbb{C}^\times_\hbar}(\mathrm{Gr}_G) \otimes_{\C[\hbar]} H_\bullet^{G \times \mathbb{C}^\times_\hbar}(X) \longrightarrow H_\bullet^{G \times \mathbb{C}^\times_\hbar}(G_\mathcal{K} \times_{G_\mathcal{O}} X)
\]
via the correspondence
\begin{equation*}
\begin{tikzcd}
    & G_\mathcal{K} \times X \ar[ld, swap, "\pi_L"] \ar[rd, "\pi_R"] & \\
    \mathrm{Gr}_G \times X && G_\mathcal{K} \times_{G_\mathcal{O}} X
\end{tikzcd}.
\end{equation*}

The map $\tw$ is \emph{twisted-linear} in the following sense. The projection
\begin{equation*}
u: \mathrm{Gr}_G = G_\mathcal{K} / G_\mathcal{O} \longrightarrow [\mathrm{pt} / G_\mathcal{O}]
\end{equation*}
induces a pullback homomorphism
\[
u^*: H^\bullet_{G \times \mathbb{C}^\times_\hbar}(\mathrm{pt}) \longrightarrow H^\bullet_{G_\oh \rtimes \mathbb{C}^\times_\hbar}(\mathrm{Gr}_G).
\]
For any $P \in H^\bullet_{G \times \mathbb{C}^\times_\hbar}(\mathrm{pt})$ and $\Gamma \otimes \alpha \in H_\bullet^{G \times \mathbb{C}^\times_\hbar}(\mathrm{Gr}_G) \otimes H_\bullet^{G \times \mathbb{C}^\times_\hbar}(X)$, we have
\[
\tw(\Gamma \otimes P \alpha) = \tw((u^*P\cap \Gamma )\otimes \alpha) = u^*P\cap \tw(\Gamma\otimes\alpha),
\]
which in the abelian case reflects the twisted-linearity relation~\eqref{Shiftrelations2} (see \Cref{abeliantwlinear}).

The second step involves curve-counting in a certain associated $X$-bundle. By a theorem of \citeauthor{BL}~\cite{BL}, there is a canonical principal $G$-bundle $\mathcal{E}$ over $\mathrm{Gr}_G \times \mathbb{P}^1$ together with a trivialization 
\begin{equation}\label{trivialization}
\varphi: \sE|_{\Gr_G\times\spec(\C[t^{-1}])}\stackrel{\sim}{\longrightarrow}G\times \Gr_G\times\spec(\C[t^{-1}])
\end{equation}
over $\mathrm{Gr}_G \times \operatorname{Spec}(\mathbb{C}[t^{-1}])$; here we use $t$ to denote the coordinate on $\PP^1$. The trivialization $\varphi$ will play a crucial role in the general $\bN$ case below.
Let $\mathcal{E}(X) \coloneqq \mathcal{E} \times_G X$ denote the associated $X$-bundle. The restrictions of $\mathcal{E}(X)$ to $\mathrm{Gr}_G \times \{0\}$ and $\mathrm{Gr}_G \times \{\infty\}$ are isomorphic to $G_\mathcal{K} \times_{G_\mathcal{O}} X$ and $\mathrm{Gr}_G \times X$, respectively. 

Let $\Eff(\mathcal{E}(X))^{\mathrm{sec}} \subset H_2^{\text{ord}}(\mathcal{E}(X); \mathbb{Z})$ denote the subset of effective section classes, i.e., those effective classes whose image in $H_2^{\text{ord}}(\mathrm{Gr}_G \times \mathbb{P}^1; \mathbb{Z})$ is equal to $[\mathrm{pt} \times \mathbb{P}^1]$.
For $\beta\in\Eff(\mathcal{E}(X))^{\mathrm{sec}}$, let $\mathcal{M}(X,\beta)_n$ be the moduli space of genus-zero stable maps to $\mathcal{E}(X)$ with curve class $\beta$ and $n+2$ points $y_0,y_\infty,y_1,\dots,y_n$, such that $y_0$ and $y_\infty$ lie over the $0$ and $\infty$ fibres, respectively. Let $\ev_0, \ev_\infty, \ev_1,...$ be the evaluation maps. We define
\[
\widetilde{\bbS}_{G,X} : H_\bullet^{G_\oh\rtimes \mathbb{C}^\times_\hbar}(G_\mathcal{K} \times_{G_\mathcal{O}} X) \longrightarrow H_\bullet^{G\times \mathbb{C}^\times_\hbar}( X)[[q_G,\tau]]
\]
by the formula
\begin{align*}
\widetilde{\bbS}_{G,X}(\gamma) 
= \sum_{\beta \in \Eff(\mathcal{E}(X))^{\sec}} \sum_{n=0}^\infty
\frac{q^{\overline{\beta}}}{n^!}\, 
\pr_{X*}\ev_{\infty*} \left( 
    \ev_0^*(\gamma) \prod_{\ell=1}^n\ev_\ell^*(\hat\tau)\cap [\mathcal{M}(X, \beta)_n]^{\mathrm{vir}} 
\right),
\end{align*}
where $\pr_X:\Gr_G\times X\to X$ is the projection map, $[\mathcal{M}(X, \beta)_n]^{\mathrm{vir}}$ is the virtual fundamental class of the moduli space, and $\overline{\beta}$ is the image of $\beta$ under the natural map
\[
H_2^{\text{ord}}(\mathcal{E}(X); \mathbb{Z}) \longrightarrow H_2^{\text{ord},G}(X; \mathbb{Z})
\]
induced by $\mathcal{E}(X)=\sE\times_G X \to [X/G]$; and the definition of $\hat\tau \in H_{G \times \Cxh}^\bullet(\sE(X))$ is given in \Cref{tauhat}.

The module map $\bbS_{G,X}$ is then defined to be the composition
\begin{align*}
    H^{G_\oh \rtimes \Cx_\hbar}_\bullet(\Gr_G) \otimes_{\C[\hbar]} H_\bullet^{G \times \Cx_\hbar}(X)[[q_G,\tau]] 
    &\xrightarrow{\mathmakebox[\widthof{$\widetilde\bbS_{G,X}$}][c]{\tw}} 
    H_\bullet^{G \times \Cx_\hbar}(G_\ok \times_{G_\oh} X)[[q_G,\tau]] \\
    &\xrightarrow{\widetilde\bbS_{G,X}} 
    H_\bullet^{G \times \Cx_\hbar}(X)[[q_G,\tau]],
\end{align*}
up to intertwining with Poincar\'e duality. 

Here are two remarks about the above construction.
\begin{enumerate}
\item 
First, to obtain a well-defined notion of virtual fundamental classes, we must approximate the Borel--Moore homology of $\Gr_G$ by the homology of resolutions of its affine Schubert varieties.
\item
Second, the use of equivariant Novikov variables is essential in this construction. When $G = T$ is a torus, one can define a map $H_2^{\text{ord}}(\mathcal{E}(X); \mathbb{Z}) \to H_2^{\text{ord}}(X; \mathbb{Z})$ depending on the choices of suitable section classes (see~\cite{Iritani}). In the cases studied in~\cite{Chow}, the natural map $H_2^{\text{ord}}(X; \mathbb{Z}) \to H_2^{\text{ord},G}(X; \mathbb{Z})$ is an isomorphism. In~\cite{GMP,GMP2}, equivariant Novikov variables appeared implicitly by the consideration of vertical Chern classes (cf. \Cref{degree}).
\end{enumerate}

\subsubsection*{The general case}
The case of a general $\bN$ is more difficult and much richer, revealing a deep and previously unexplored connection between the Coulomb branch and properness of the relevant moduli spaces. We begin by explaining the difficulties.

Since $\sA_{G,\bN}^\hbar$ is a subalgebra of $\sA_{G}^\hbar$, it is tempting to define $\bbS_{G,\bN,X}$ as the restriction of $\bbS_{G,X}$. This idea works for the map $\tw$, but fails for the map $\widetilde{\bbS}_{G,X}$. The main issue is that the $T$-fixed locus $X^T$ may \emph{not} be compact in our general setting. In this case, the evaluation map $\ev_\infty$ (or its restriction to the $T$-fixed locus) may fail to be proper. Hence, the pushforward $\ev_{\infty*}$ is not well-defined in general, even using localization.

Our strategy is to cut down the moduli space $\cM(X,\beta)_n$ to a smaller subspace on which the evaluation maps become proper. A key ingredient in our approach is a reformulation of the Coulomb branch algebra $\sA^\hbar_{G,\bN}$, which we now describe.

Recall that in~\cite{BFN}, the authors considered an infinite-rank vector bundle $\BT = \BT_{\bN}$ over the affine Grassmannian $\Gr_G$ and a fibrewise linear subvariety $\BR \subset \BT$. The Coulomb branch algebra $\sA_{G,\bN}^\hbar$ is defined as the equivariant Borel--Moore homology of $\BR$. Let $\BS \coloneq \BT/\BR$ be the fibrewise quotient. We prove the following theorem in \Cref{Coulomb}.

\begin{introthm}\label{Introthm3}
The (quantized) Coulomb branch algebra $\sA_{G,\bN}^\hbar$ is isomorphic to the following subalgebra of $\sA_G^\hbar$:
\[
 e(\BS) \cap H^{G \times \mathbb{C}^\times_\hbar}_\bullet(\mathrm{Gr}_G) \subset \sA_G^\hbar.
\]
\end{introthm}
The heuristics behind \Cref{Introthm3} are as follows: if one ``resolves'' the affine Grassmannian $\Gr_G$ by the vector bundle $\BT$, then the pullback of $\BS$ admits a canonical section whose zero locus is precisely $\BR$. The fibrewise quotient $\BS$ in \Cref{Introthm3} is a ``stratified'' vector bundle over $\Gr_G$, meaning that $\BS$ restricts to a vector bundle over each affine Schubert cell in $\mathrm{Gr}_G$. See \Cref{Coulomb} on how we make sense of the symbol $e(\BS)\cap$.

Here we uncover a novel relation between Coulomb branches and shift operators. \Cref{Introthm3} establishes that $\sA_{G,\bN}^\hbar$ is the subalgebra of $\sA_G^\hbar$ cut out by the stratified bundle $\BS$. The following lemma
shows that, at the same time, $\BS$ cuts out subspaces of the moduli space $\cM(X,\beta)_n$ that are proper with respect to the evaluation maps.

\begin{keylem}[=\Cref{inTp}+\Cref{propermoduli}]
    The pullback of $\BS$ to $\cM(X,\beta)_n$ admits a canonical section whose zero locus is proper with respect to the evaluation map $\ev_\infty$.
\end{keylem}

We are ready to continue our construction of shift operators for a general $\bN$. For simplicity, we will denote the pullback of $\BS$ to the various spaces by the same symbol. The twisting map $\tw$ now restricts to give a homomorphism
\[
\tw:\sA_{G,\bN} \otimes_{\C[\hbar]} H^{G \times \C^\times_\hbar}_\bullet(X) \longrightarrow e(\BS) \cap H_\bullet^{G \times \mathbb{C}^\times_\hbar}(G_\mathcal{K} \times_{G_\mathcal{O}} X).
\]
The Key Lemma provides a canonical section of $\BS$ over the moduli space $\mathcal{M}(X, \beta)_n$ such that the restriction of $\ev_\infty$ to its zero loci $\mathcal{Z}(X, \beta)_n$ is proper. As an example, the subspace $\mathcal{Z}(\bN,\beta)_0 \subset \mathcal{M}(\bN,\beta)_0 = \Gamma(\Gr_G, \sE(\bN))$ consists of those sections that are constant over $\operatorname{Spec}(\mathbb{C}[t^{-1}])$, with respect to the trivialization of $\varphi$ (\Cref{trivialization}).

We may therefore define
\[
\widetilde\bbS_{G,\bN,X} : e(\BS) \cap H_\bullet^{G \times \mathbb{C}^\times_\hbar}(G_\mathcal{K} \times_{G_\mathcal{O}} X) \longrightarrow H_\bullet^{G \times \mathbb{C}^\times_\hbar}(X)[[q_G,\tau]]
\]
by the formula
\begin{align*}
\widetilde\bbS_{G,\bN,X}(e(\BS) \cap \gamma)
= \sum_{\beta \in \Eff(\mathcal{E}(X))^{\sec}} \sum_{n=0}^\infty
\frac{q^{\overline{\beta}}}{n^!}\, 
\pr_{X*}\ev_{\infty*} \left( 
    \ev_0^*(\gamma) \prod_{\ell=1}^n\ev_\ell^*(\hat\tau)\cap[\mathcal{Z}(X, \beta)_n]^{\mathrm{vir}}\right),
\end{align*}
and set
$\bbS_{G,\bN,X}=\widetilde\bbS_{G,\bN,X}\circ\tw$, up to intertwining with Poincar\'e duality.

\begin{rem}[Assumption on $X$]
    In \Cref{qh}, we assume that $X$ admits an extra commuting $\Cx$-action such that $X^{T\times\Cx}$ is proper. This is required in order to define the quantum cohomology and $\bbS_{G,\bN}$. However, this extra $\Cx$-action is actually not used in the proofs of the main theorems.
\end{rem}

\begin{rem}[Independence of the choice of representation]
A priori, the construction of $\bbS_{G,\bN,X}$ depends on the choice of the representation $\bN$ and the map $f \colon X \to \bN$. However, we will show that this dependence is in fact superfluous. More precisely, suppose $\bN'$ is another representation of $G$, and $g \colon X \to \bN'$ is a proper $G$-equivariant morphism. Then we show in \Cref{CommonDomain} that $\bbS_{G,\bN,X}$ and $\bbS_{G,\bN',X}$ agree on the common domain of definition:
\begin{equation*}
(\sA_{G,\bN}^\hbar \cap \sA_{G,\bN'}^\hbar) \otimes_{\mathbb{C}[\hbar]} QH_{G \times \mathbb{C}^\times_\hbar}^\bullet(X).
\end{equation*}
Therefore, they extend to the same map on the localization
\[(\oh_{T^*\check T}^\hbar)_\loc\otimes_{\C[\hbar]} QH_{G\times\Cxh}^\bullet(X).\]
\end{rem} 

\begin{rem}[Relation with 2d and 3d mirror symmetry]
As explained above, in the context of 3d mirror symmetry, an equivariant fibration $f: X \to \bN$ corresponds to a 3d brane on the Higgs branch $[T^*\bN/\!\!/G]$ and the mirror brane should be given by a Lagrangian in $\spec \sA_{G,\bN}$ produced from \Cref{IntroThm1}. It is expected that this Lagrangian reflects the 2d mirror symmetry of the fibration $f: X \to \bN$ (see~\cite{3dmirror} for further discussion).


In the special case $\bN = 0$, this perspective is useful in the study of quantum cohomology of GIT quotients~\cite{PomerleanoTeleman, IritaniFourier}. The categorical generalization of this correspondence, in terms of wrapped Fukaya categories, was also conjectured in \cite{Lekili-Segal}. If one interprets the quantum cohomology $QH^\bullet_G(X)$ as a closed-string incarnation of the (equivariant) 2d mirror of $X$, then one may expect a corresponding open-string construction of the family of closed Lagrangians
\[
\operatorname{Supp} QH^\bullet_G(X) \subset \operatorname{Spec} \mathcal{A}_{G,\bN}.
\]
Such a construction was described by the first-named author and Leung in the abelian case in~\cite{3dmirror} and then in the non-abelian case in~\cite{nonabelian} (see also~\cite[Section~1.8]{3dmirror}). We conjecture that the Lagrangians constructed in\cite{3dmirror, nonabelian} coincide with those given by \Cref{IntroThm1}.
\end{rem}

\subsection*{Relations with other constructions}

\subsubsection*{The case $G = T$ and $\bN = \mathbf{0}$}
Shift operators were originally introduced by Braverman, Maulik, Okounkov, and Pandharipande in~\cite{OP,BMO,MO} for small quantum cohomology. The version on equivariant symplectic cohomology was constructed in \cite{LJshiftop}. The classical limit as $\hbar \to 0$ corresponds to the so-called Seidel elements or Seidel representations, which had appeared earlier in~\cite{Seidel}. A generalization of shift operators to big quantum cohomology was developed by Iritani~\cite{Iritani} to study toric mirror symmetry. This corresponds to the ``$G = T$ and $\bN = \mathbf{0}$'' case of \Cref{IntroThm2}.

\subsubsection*{The case of general $G$ and $\bN = \mathbf{0}$}
Non-abelian shift operators in the case $\bN = \mathbf{0}$ were suggested in~\cite{Tel}, and constructed in~\cite{GMP} using symplectic geometry in the setting of compact monotone symplectic manifolds. This is (the symplecto-geometric version of) the ``general $G$ and $\bN = \mathbf{0}$'' case of \Cref{IntroThm2} restricted to small quantum cohomology.
We also note that the non-abelian generalization of Seidel elements were studied in~\cite{Savelyev} via the homology of the loop group of the Hamiltonian symplectomorphism group; the ideas therein might have influenced later development. In \cite{GMP2}, a non-abelian version of Seidel representation was also constructed for equivariant symplectic cohomology.

The map $\Psi_{G,X}$ in the special case where $G$ is simple and simply connected, and $X = G/P$ is a partial flag variety was studied in~\cite{Chow}, whose setting is closer to ours. However, their treatment of Novikov variables was specific to that particular case. An essential difference in our approach is that we use equivariant Novikov variables, which allows the construction to extend to more general $X$ and arbitrary $G$. 

The main focus of~\cite{Chow} is to show that the non-abelian shift operators recover the Peterson isomorphism. In \Cref{Peterson}, we show that our generalization likewise induces an analogue of the Peterson isomorphism, valid for more general $G$, including groups that are neither simply connected nor semisimple. In this broader setting, equivariant Novikov variables play a crucial role, as certain Lagrangians in $\spec \sA_{G,\bN}$ appear only after we incorporate such variables.

\subsubsection*{The case of general $G$ and $\bN$} 
Our paper is the first to define the shift operators $\bbS_{G,\bN,X}$ and $\Psi_{G,\bN,X}$ for general choices of $G$, $\bN$, and $X$. In particular, our construction recovers the following special cases:
\begin{enumerate}
    \item In the case $G = T$ with $X^T$ compact, shift operators were defined in~\cite{Iritani} using localization. In fact, under the assumptions of \textit{op.~cit.}, there always exists a $G$-representation $\bN$ and a proper holomorphic map $f: X \to \bN$. Then, the shift operators of \textit{op.~cit.} can be recovered by localizing our construction of $\bbS_{G,\bN,X}$ (see \Cref{CompareIritaniRemark}).
    
    \item We also recover the case $X = \bN$ studied in~\cite{GMP2}, which builds on Teleman’s description of the Coulomb branch $\spec \sA_{G,\bN}$ as a gluing of two copies of the pure gauge Coulomb branch $\spec \sA_G$ (see~\cite{2drole}). Their construction relies on the observation that the gluing maps coincide with certain Seidel elements in the case $X = \bN$. In contrast, our result yields a new, self-contained proof of Teleman’s gluing construction (\Cref{Telemangluing}).
\end{enumerate}
We remark that even in the abelian case, the construction of shift operators when the $T$-fixed locus $X^T$ is noncompact has not previously appeared in the literature.

\subsubsection*{Affine flag manifold version}
In \cite{Cotangentofflag}, we consider the affine flag version of Coulomb branches (called Iwahori--Coulomb branches) and the corresponding shift operators. We prove analogues of \Cref{Introthm3} and the Key Lemma in that setting, and explicitly compute the shift operators on quantum cohomology of cotangent bundles of partial flag manifolds, which yield various interesting applications.

\subsubsection*{Shift operators for quasimaps}
In \cite{Tamagni}, the author provided a different construction of non-abelian shift operators using quasimaps for $G=\operatorname{GL}_n$. Their computation offers a new proof that $\sA^\hbar_{\mathrm{GL}_n}$ is a quotient of the shifted Yangian (see also \cite{BFNslice}).

\subsubsection*{Open-string analogue}
As mentioned above, the first-named author and Leung have attempted to construct Lagrangians in the Coulomb branch $\spec \sA_{G,\bN}$ using equivariant 2d mirror symmetry, motivated again by ideas of Teleman~\cite{Tel}. The abelian case was discussed in~\cite{3dmirror}. The non-abelian case involves studying non-displaceable (real) Lagrangians under a Hamiltonian action. The case of general $G$ with $\bN = 0$ was studied in~\cite{nonabelian}; see~\cite[Section 1.8]{3dmirror} for a discussion of the case with general $G$ and arbitrary $\bN$.

\subsection*{Structure of the paper}
In \Cref{Coulomb}, we review the definition of the Coulomb branches and prove \Cref{Introthm3}. The twisting map $\tw$ is discussed in \Cref{twisting}. \Cref{preparation,sectionshiftoperators,propertyshiftoperators} are devoted to the definitions and properties of the shift operators. The proofs of \Cref{IntroThm1} and \Cref{IntroThm2} are given in \Cref{propertyshiftoperators}. 
Several applications and computational examples of shift operators are presented in \Cref{Applications}, including \Cref{Introlargestsubalg} (which recovers Teleman's gluing formula for Coulomb branches \cite{2drole} as a corollary)
and various generalizations of the Peterson isomorphism.

\subsubsection*{Conventions}
In this paper, $H_\bullet$ denotes the Borel--Moore homology and $H_\bullet^{\mathrm{ord}}$ denotes the ordinary homology. All varieties, schemes and stacks considered in this paper are over $\C$.

\section*{Acknowledgements}
We thank Hiroshi Iritani, Syu Kato, Conan Leung, Yan-Lung Leon Li, Cheuk Yu Mak, Michael McBreen, Daniel Pomerleano, and Ben Webster for valuable discussions at various stages of this project. We are particularly grateful to Hiroshi Iritani for many useful comments and suggestions.

K. F. Chan was substantially supported by grants from the Research Grants Council of the Hong Kong Special Administrative Region, China (Project No. CUHK14301721, CUHK14306322,
and CUHK14305923) and a direct grant from CUHK. He would like to thank the support provided by The Institute of Mathematical Sciences at The Chinese University of Hong Kong. K. Chan and C. H. E. Lam were substantially supported by grants of the Hong Kong Research Grants Council (Project No. CUHK14305322, CUHK14305023 \& CUHK14302524).

\section{Coulomb branches}\label{Coulomb}
In this section, we give a short treatment of the Coulomb branch and set up notations that will be important for later sections. We work over the complex numbers $\mathbb{C}$.

\subsubsection*{Lie-theoretic notations}

We let $G$ denote a connected complex reductive group, $T \subset G$ a Cartan subgroup, and $W = N_G(T)/T$ the corresponding Weyl group. We write $\Phi$ for the set of roots of $G$. We fix a subset $\Phi^+ \subset \Phi$ of positive roots, or equivalently, we fix a Borel subgroup $T \subset B \subset G$. We denote $\lt$ the Lie algebra of $T$.

The coweight lattice of $G$ is denoted by $\Lambda$, and its submonoid of dominant coweights is denoted by $\Lambda^+$.

\subsubsection*{The affine Grassmannian} For a $\mathbb{C}$-algebra $A$, let $G_A$ denote the sheaf\footnote{Sheaves are defined with respect to the fppf topology.} on the category of affine schemes over $\mathbb{C}$ (that is, the opposite category of the category of $\mathbb{C}$-algebras) defined by sending a $\mathbb{C}$-algebra $R$ to $G(R \otimes A)$.

Denote $\oh = \mathbb{C}[[t]]$ (the ring of formal power series) and $\ok = \mathbb{C}\cct$ (the field of Laurent series). In particular, we have
\begin{align*}
    G_\oh(R) &= G(R[[t]]), \\
    G_\ok(R) &= G(R\cct),
\end{align*}
for any $\mathbb{C}$-algebra $R$. The affine Grassmannian of $G$ is defined as the quotient
\[
\Gr_G = G_\ok / G_\oh.
\]
Let $\BI$ be the Iwahori subgroup
\[
\BI = \{g \in G_\oh : g(0) \in B\},
\]
that is, $\BI$ is the preimage of $B$ under the evaluation map $\operatorname{ev}_{t=0} : G_\oh \to G$.

For $\lambda \in \Lambda$, let $t^\lambda$ denote the corresponding point in $\Gr_G$. Define the $\BI$-orbit
\[
C_\lambda = \BI \cdot t^\lambda \subset \Gr_G,
\]
and let $C_{\leq \lambda}$ denote its closure. Each $C_\lambda$ is isomorphic to an affine space, and $C_{\leq \lambda}$ admits the structure of a projective variety. We define a partial order on $\Lambda$ by declaring that $\mu \leq \lambda$ if and only if $C_\mu \subset C_{\leq \lambda}$. In particular,
\[
C_{\leq \lambda} = \bigsqcup_{\mu \leq \lambda} C_\mu.
\]
We remark that $C_\lambda$ (and hence $C_{\leq \lambda}$) is $G_\oh$-invariant if and only if $\lambda \in \Lambda^+$. 

Let $\Cxh$ be a one-dimensional complex torus that scales the parameter $t$. It is called the \emph{group of loop rotations}, and it acts on $G_\ok$, $G_\oh$, and $\Gr_G$. For $z \in \Cxh$ and $g(t) \in G_\ok$, the action is given by
\[
z \cdot g(t) = g(zt),
\]
and the other actions are defined similarly. See \Cref{univbun} and \Cref{Gtorsorappendix} for more discussion.

\subsubsection*{Equivariant Borel--Moore homology}

In this paper, if $X$ is a complex quasiprojective variety, we use $H_\bullet(X) = H^{-\bullet}(X, \omega_X)$ to denote the Borel--Moore homology groups of $X$ with complex coefficients, using the classical topology. Here, $\omega_X$ is the dualizing complex of $X$. Similarly, if $K$ is an algebraic group acting algebraically on $X$, we define the equivariant Borel--Moore homology $H^K_\bullet(X) \coloneqq H_K^{-\bullet}(X, \omega_X)$,
where both $K$ and $X$ are considered with the classical topology. We refer to~\cite{BernsteinLunts} for the basics of cohomology of equivariant sheaves.

\subsubsection*{The pure gauge Coulomb branch} We define
\[
\sA^\hbar_G=H^{G\times\Cxh}_\bullet(\Gr_G) \coloneqq \varinjlim_{\lambda \in \Lambda^+} H^{G\times\Cxh}_\bullet(C_{\leq \lambda}).
\]
We can similarly define $H^{K}_\bullet(\Gr_G)$, for any algebraic group $G$ acting on $\Gr_G$.

There is a convolution product $\ast$ on $\sA^\hbar_G$ turning it into a finitely generated commutative $\mathbb{C}$-algebra~\cite{BFM,BFN}. The resulting algebra $(\sA^\hbar_G,\ast)$ is called the \emph{(pure gauge, quantized) Coulomb branch algebra}. 

We follow~\cite{MV,BFN}. The convolution product on $H^{G_\oh}_\bullet(\Gr_G)$ is defined using the diagram
\begin{equation}\label{convolution1}
\begin{tikzcd}
    \Gr_G \times \Gr_G & \ar[l, "p"'] G_\ok \times \Gr_G \ar[r, "q"] & G_\ok \times_{G_\oh} \Gr_G \ar[r, "m"] & \Gr_G,
\end{tikzcd}
\end{equation}
where $p$ and $q$ are the natural projections, and $m$ is given by $m([g,g']) = [gg']$. Let $G_\oh$ act on $\Gr_G$ and $G_\ok \times_{G_\oh} \Gr_G$ from the left, and let $G_\oh \times G_\oh$ act on $G_\ok \times \Gr_G$ by
\[
(g_1, g_2) \cdot (g, [g']) = (g_1 g g_2^{-1}, [g_2 g']).
\]
Then $p$ is $G_\oh \times G_\oh$-equivariant, $m$ is $G_\oh$-equivariant, and $q$ is equivariant with respect to the first $G_\oh$-action on $G_\ok\times\Gr_G$. The convolution product is given by
\begin{equation}\label{product1}
    m_* \circ (q^*)^{-1} \circ p^* : \sA^\hbar_G \otimes \sA^\hbar_G \to \sA^\hbar_G.
\end{equation}

Since we are dealing with infinite-dimensional spaces and groups, the Borel--Moore homology groups and the homomorphisms among them must be treated carefully. We briefly explain this, and refer to~\cite{BFN} for more details.

For any positive integer $i$, let $G_i = G_{\oh / t^i \oh}$, and $K_i$ denote the kernel of the natural homomorphism $G_\oh \to G_i$. Note that the inclusion $G\subset G_i$ is a deformation retract. For $\lambda \in \Lambda^+$, we write $G_\ok^{\leq \lambda}$ for the preimage of $C_{\leq \lambda} \subset \Gr_G$ in $G_\ok$,.

Let $\lambda_1, \lambda_2 \in \Lambda^+$ and set $\lambda_3 = \lambda_1 + \lambda_2$. Choose a positive integer $i \gg 0$ such that the action of $K_i$ on $C_{\leq\lambda_2}$ is trivial. The diagram~\eqref{convolution1} induces the diagram below (we use the same symbols for the induced maps):
\begin{equation}\label{convolution2}
\begin{tikzcd}
    C_{\leq \lambda_1} \times C_{\leq \lambda_2} & \ar[l, "p"'] (G_\ok^{\leq \lambda_1} / K_i) \times C_{\leq \lambda_2} \ar[r, "q"] & G_\ok^{\leq \lambda_1} \times_{G_\oh} C_{\leq \lambda_2} \ar[r, "m"] & C_{\leq \lambda_3}.
\end{tikzcd}
\end{equation}

The product~\eqref{product1} is understood as the direct limit over $\lambda_1, \lambda_2 \in \Lambda^+$ of
\[
m_* \circ (q^*)^{-1} \circ p^* : H^{G\times\Cxh}_\bullet(C_{\leq \lambda_1}) \otimes_{\C[\hbar]} H^{G\times\Cxh}_\bullet(C_{\leq \lambda_2}) \longrightarrow H^{G\times\Cxh}_\bullet(C_{\leq \lambda_3}),
\]
where $p^*$, $q^*$, and $m^*$ are homomorphisms induced by the diagram~\eqref{convolution2}. Explicitly:
\begin{itemize}
    \item The map
    \[
    p^* : H^{G\times\Cxh}_r(C_{\leq \lambda_1}) \otimes_{\C[\hbar]} H^{G\times\Cxh}_s(C_{\leq \lambda_2}) \longrightarrow H^{G\times G\times\Cxh}_{r+s+2\dim G_i}\left((G_\ok^{\leq \lambda_1} / K_i) \times C_{\leq \lambda_2}\right)
    \]
    is the pullback along $p$.
    
    \item The map
    \[
    q^* : H^{G\times\Cxh}_{r+s}\left(G_\ok^{\leq \lambda_1} \times_{G_\oh} C_{\leq \lambda_2}\right) \xlongrightarrow{\simeq} H^{G\times G\times\Cxh}_{r+s+2\dim G_i}\left((G_\ok^{\leq \lambda_1} / K_i) \times C_{\leq \lambda_2}\right)
    \]
    is the pullback along $q$, with respect to the inclusion $G\cong G\times\{e\} \subset G \times G$.
    
    \item The map
    \[
    m_* : H^{G\times\Cxh}_{r+s}\left(G_\ok^{\leq \lambda_1} \times_{G_\oh} C_{\leq \lambda_2}\right) \longrightarrow H^{G\times\Cxh}_{r+s}(C_{\leq \lambda_3})
    \]
    is the pushforward along $m$.
\end{itemize}

Note that $q^*$ is an isomorphism because the $G_i$-action makes $(G_\ok^{\leq \lambda_1} / K_i) \times C_{\leq \lambda_2}$ into a principal $G_i$-bundle over $G_\ok^{\leq \lambda_1} \times_{G_\oh} C_{\leq \lambda_2}$.

\subsubsection*{The classical pure gauge Coulomb branch}
As shown in~\cite{BFM}, the convolution product of $\sA^\hbar_G$ descends to a commutative product $\sA_G=\sA^\hbar_G/\hbar \sA^\hbar_G$. Moreover, it induces a Poisson bracket on $\sA_G$ which makes $\operatorname{Spec} \sA_G$ smooth symplectic variety.

\subsubsection*{The BFN Coulomb branch}Let $\bN$ be a complex representation of $G$, and we fix a decomposition
\[
\bN = \bigoplus_{j=1}^n \C_{\eta_j}
\] 
into one-dimensional $T$-representations, where each $\eta_j$ is a character of $T$, and $\C_{\eta_j}$ denotes the corresponding one-dimensional $T$-representation. We define the following spaces:
\[
\BT = \BT_{G,\bN} = G_\ok \times_{G_\oh} \bN_\oh,
\]
\[
\BR = \BR_{G,\bN} = \{(g, s) \in \BT : g s \in \bN_\oh \},
\]
\[
\BS = \BS_{G,\bN} = \BT_{G,\bN} / \BR_{G,\bN}.
\]

Suppose $d$ is a positive integer, we write $\BT^d$ for the vector bundle
\[
\BT^d = G_\ok \times_{G_\oh} (\bN_\oh / t^d \bN_\oh).
\]
If $\lambda \in \Lambda$, we write $\BT^d_{\leq \lambda}$ for the restriction of $\BT^d$ to $C_{\leq \lambda}$. If, furthermore, $gt^d \bN_\oh \subset \bN_\oh$ for all $[g] \in C_{\leq \lambda}$, then we write $\BR^d_{\leq \lambda}$ as the image of $\BR_{\leq \lambda}=\BR|_{C_{\leq \lambda}}$ in $\BT^d_{\leq \lambda}$. Note that the fibrewise quotient
\[
\BS_{\leq \lambda} \vcentcolon= \BT^d_{\leq \lambda} / \BR^d_{\leq \lambda}
\]
is independent of the choice of such $d$. Moreover, $\BS_\lambda \vcentcolon= \BS_{\leq \lambda}|_{C_\lambda}$ is a vector bundle, whose rank is denoted by $d_\lambda$. For $p \in \Gr_G$, we write $\BS_p$ for the fibre of $\BS_{\leq \lambda}$ at $p$, for $\lambda \in \Lambda$ satisfying $p \in C_{\leq \lambda}$. This definition is independent of the choice of $\lambda$.

For later use, we record the decomposition
\begin{equation}\label{Nweights}
    \BS_{t^\lambda} \cong \bigoplus_j\bigoplus_{c=0}^{-\langle\eta_j, \lambda \rangle -1} \C_{\eta_j+c\hbar},
\end{equation}
as $T\times\Cxh$-representations. In particular, 
\begin{equation}\label{SlambdaRank}
    d_\lambda = -\sum_{j:\langle\eta_j,\lambda\rangle < 0} \langle\eta_j,\lambda\rangle.
\end{equation}
For each $\lambda \in \Lambda^+$, we choose an integer $d$ so that $\BR^d_{\leq \lambda}$ is defined, and we use the same symbol $z^*_\lambda$ to denote the corresponding Gysin pullbacks
\[H^K_{\bullet}(\BT^d_{\leq\lambda})\longrightarrow H^K_\bullet(C_{\leq\lambda}),\]
where $K$ stands for one of $G$, $G\times\Cxh$, $T$, or $T\times\Cxh$. The following is a reformulation of the definition of the Coulomb branch in~\cite{BFN}, which follows from Lemma 5.11 in \textit{loc.~cit}.

\begin{df}[\cite{BFN}]
The \emph{classical Coulomb branch algebra} $\sA_{G,\bN}$ is defined as the sum over $\lambda \in \Lambda^+$ of the images of $H^{G\times\Cxh}_\bullet(\BR^d_{\leq \lambda})$\footnote{In our convention, the loop rotation does not act on $\bN$.} under the composition
\[
H^{G}_\bullet(\BR^d_{\leq \lambda})
\longrightarrow
H^{G}_\bullet(\BT^d_{\leq \lambda})
\stackrel{z_\lambda^*}{\longrightarrow}
H^{G}_\bullet(C_{\leq \lambda})
\longrightarrow
H^{G}_\bullet(\Gr_G),
\]
where the first and last maps are pushforwards along inclusions.

Similarly, the \emph{(quantized) Coulomb branch algebra} $\sA^\hbar_{G,\bN}$ is the sum over $\lambda \in \Lambda^+$ of the images of $H^{G\times\Cxh}_\bullet(\BR^d_{\leq \lambda})$ under the composition
\[
H^{G\times\Cxh}_\bullet(\BR^d_{\leq \lambda})
\longrightarrow
H^{G\times\Cxh}_\bullet(\BT^d_{\leq \lambda})
\stackrel{z_\lambda^*}{\longrightarrow}
H^{G\times\Cxh}_\bullet(C_{\leq \lambda})
\longrightarrow
H^{G\times\Cxh}_\bullet(\Gr_G).
\]

\end{df}

\subsubsection*{An alternative description}

We now give a new and alternative description of the Coulomb branch algebras. 

\begin{lem}\label{existresol}
There exists an $\BI\rtimes\Cxh$-equivariant resolution of singularities\footnote{That is, $\widetilde{C}_{\leq\lambda}$ is non-singular and $\rho_\lambda$ is a proper birational map.} 
\[
\rho_\lambda : \widetilde{C}_{\leq\lambda} \longrightarrow C_{\leq\lambda}
\]
such that $\rho_\lambda^{-1}(\BS|_{C_\lambda})$ extends to a (necessarily unique) $\BI\rtimes\Cxh$-equivariant quotient vector bundle $\widetilde{\BS}_{\leq\lambda}$ of $\rho_\lambda^{-1}(\BT^d_{\leq\lambda})$. Moreover, we may assume that $\rho_\lambda$ and $\widetilde{\BS}_{\leq\lambda}$ are $G\times\Cxh$-equivariant if $\lambda \in \Lambda^+$.
\end{lem}

\begin{proof}
Let $d > 0$ be a sufficiently large integer such that $\BS_{\leq\lambda}$ is well-defined. Consider the Grassmannian bundle $\Gr(d_\lambda, \BT^d_{\leq\lambda})$, which parameterizes rank $d_\lambda$ quotients of $\BT^d_{\leq\lambda}$. The vector bundle $\BS_\lambda$ defines a section of $\Gr(d_\lambda, \BT^d_{\leq\lambda})$ over $C_\lambda$. Let $\overline{C}_\lambda$ denote the closure of the image of this section. By construction, the section $C_\lambda \to \Gr(d_\lambda, \BT^d_{\leq\lambda})$ extends uniquely to a $\BI\rtimes\Cxh$-equivairant (or $G\times\Cxh$-equivairant when $\lambda\in\Lambda^+$) map
\[
\overline{C}_\lambda \longrightarrow \Gr(d_\lambda, \BT^d_{\leq\lambda}).
\]
As a result, $\BS_\lambda$ extends to a quotient bundle of the pullback of $\BT^d_{\leq\lambda}$ to $\overline{C}_\lambda$. We then choose $\widetilde{C}_{\leq\lambda}$ to be an $\BI\rtimes\Cxh$-equivariant (or $G\times\Cxh$-equivariant) resolution of $\overline{C}_\lambda$, which always exists (see~\cite[Theorem 3.27]{Kollar}).
\end{proof}

\begin{rem}\label{Stildesurjects}
    Let $\widetilde\BR^d_{\leq\lambda}\coloneq \ker(\rho_\lambda^{-1}(\BT^d_{\leq\lambda})\to \widetilde\BS_{\leq\lambda})$. Since $\widetilde\BR^d_{\leq\lambda}$ is a vector bundle, it is equal to the closure of $\rho_\lambda^{-1}(\BR^d_\lambda)$ in $\rho_\lambda^{-1}(\BT^d_{\leq\lambda})$. Thus, we have $\widetilde\BR^d_{\leq\lambda}\subset \rho_\lambda^{-1}(\BR^d_{\leq\lambda})$ and hence there is a natural surjection
    \[\widetilde\BS_{\leq\lambda}\twoheadrightarrow \rho_\lambda^{-1}(\BS_{\leq\lambda}).\]
\end{rem}
We now fix the $\BI\rtimes\Cxh$-equivariant resolution $\rho_\lambda : \widetilde{C}_{\leq\lambda} \to C_{\leq\lambda}$ and $\widetilde\BS_{\leq\lambda}$ for each $\lambda \in \Lambda$, and $\rho_\lambda$ is assumed to be $G\times\Cxh$-equivariant when $\lambda \in \Lambda^+$.

\begin{prop}\label{independentResol}
Let $\lambda \in \Lambda$, and let $d$ be a sufficiently large positive integer such that $\BR^d_{\leq\lambda}$ is well-defined. Then the following subspaces of $H^{T}_\bullet(\Gr_G)$ are equal:
\begin{enumerate}[label=\textup{(\arabic*)}]
    \item The sum of the images $e(\widetilde{\BS}_{\leq\mu}) \cap H^{T}_\bullet(\widetilde{C}_{\leq\mu})$ in $H^{T}_\bullet(\Gr_G)$ under the pushforward
    \[
    H^{T}_\bullet(\widetilde{C}_{\leq\mu}) \xrightarrow{\rho_{\mu*}} H^{T}_\bullet(C_{\leq\mu}) \subset H^{T}_\bullet(\Gr_G),
    \]
    taken over all $\mu \leq \lambda$;
    
    \item The image of $H^{T}_\bullet(\BR^d_{\leq\lambda})$ under the composition
    \[
    H^{T}_\bullet(\BR^d_{\leq\lambda})\longrightarrow H^{T}_\bullet(\BT^d_{\leq\lambda}) \stackrel{z_\lambda^*}{\longrightarrow} H^{T}_\bullet(C_{\leq\lambda}) \subset H^{T}_\bullet(\Gr_G);
    \]
    
    \item The direct sum
\[
\bigoplus_{\mu \leq \lambda} H_T^\bullet(\pt) \cdot p_{\mu*}\left( e(\widetilde{\BS}_{\leq\mu}) \cap [\widetilde{C}_{\leq\mu}] \right),
\]
i.e., the free $H_T^\bullet(\pt)$-submodule of $H^{T}_\bullet(\Gr_G)$ with basis $\{p_{\mu*}\left( e(\widetilde{\BS}_{\leq\mu}) \cap [\widetilde{C}_{\leq\mu}] \right)\}_{\mu \leq \lambda}$.
\end{enumerate}

Moreover, each element $e(\widetilde{\BS}_{\leq\mu}) \cap [\widetilde{C}_{\leq\mu}]$ is independent of the choice of resolution. The same equivalences hold when $T$ is replaced by $T \rtimes \Cxh$. Furthermore, \textup{(1)} and \textup{(2)} remain equivalent when $T$ is replaced by $G$ or $G \times \Cxh$, and only dominant coweights are considered.
\end{prop}

\begin{proof}
We only prove the proposition for $T$; the case for $T\times \Cxh$ proceeds similarly. The cases for $G$ and $G\times \Cxh$ are obtained by taking Weyl invariants.

We first show that (1) $\subset$ (2). Consider the fibre diagram
\[
\begin{tikzcd}
\rho_\mu^{-1}\BT^d_{\leq \mu} \ar[r, "\rho'_\mu"] \ar[d, "\widetilde{\pr}"'] 
& \BT^d_{\leq \mu} \ar[d, "\pr"]\\
\widetilde{C}_{\leq\mu} \ar[r, "\rho_\mu"]
& C_{\leq\mu} 
\end{tikzcd}.
\]

Let $\widetilde{z_\lambda}^*$ be the Gysin map for the vector bundle $\rho_\mu^{-1}\BT^d_{\leq \mu}$ over $\widetilde{C}_{\leq\mu}$, then for $\gamma\in H^{T}_\bullet(\widetilde{C}_{\leq\mu})$, we have
\begin{equation} \label{pullpushcal}
\begin{aligned}
\rho_{\mu*}\left(e(\widetilde{\BS}_{\leq\mu}) \cap \gamma\right)
&= \rho_{\mu*}\left(e(\widetilde{\BS}_{\leq\mu}) \cap \widetilde{z_\lambda}^*\widetilde{\pr}^*\gamma\right) \\
&= \rho_{\mu*}\, \widetilde{z_\lambda}^*\left( \widetilde{\pr}^*e(\widetilde{\BS}_{\leq\mu}) \cap \widetilde{\pr}^*\gamma \right) \\
&= z_\lambda^* \,\rho'_{\mu*}\left(\widetilde\pr^*e(\widetilde{\BS}_{\leq\mu}) \cap \widetilde{\pr}^*\gamma\right).
\end{aligned}
\end{equation}

Let $\widetilde{\BR}^d_{\leq \mu}$ denote the kernel of the projection $\rho_\mu^{-1}\BT^d_{\leq \mu} \to \widetilde{\BS}_{\leq\mu}$, and $\iota:\widetilde{\BR}^d_{\leq \mu}\to \rho_\mu^{-1}\BT^d_{\leq \mu}$ be the inclusion. Since $\widetilde{\BS}_{\leq \mu}=\rho_\mu^{-1}\BT^d_{\leq \mu}/\widetilde{\BR}^d_{\leq\mu}$, we have:
\[
\widetilde\pr^*e(\widetilde{\BS}_{\leq\mu}) \cap H^{T}_\bullet(\rho_\mu^{-1}\BT^d_{\leq \mu}) \subset \iota_*H^{T}_\bullet(\widetilde{\BR}^d_{\leq \mu}).
\]
As a result:
\[
\rho_{\mu*}\big(e(\widetilde{\BS}_{\leq\mu}) \cap \gamma\big)
\in z_\lambda^* \,\rho'_{\mu*} \iota_*H^{T}_\bullet(\widetilde{\BR}^d_{\leq \mu})
\subset z_\lambda^*\,\iota_* H^{T}_\bullet(\BR^d_{\leq \mu}),
\]
where the last inclusion follows from $\rho'_\mu(\widetilde{\BR}^d_{\leq \mu}) \subset \BR^d_{\leq \mu}$. This proves that (1) $\subset$ (2).

To prove (2) $\subset$ (3), note that there is an affine stratification
\begin{equation}\label{Rdecom}
\BR^d_{\leq \lambda} = \bigsqcup_{\mu \leq \lambda} \BR^d_\mu,
\end{equation}
where $\BR^d_{\mu}=\BR^d_{\leq \lambda}|_{C_{\mu}}$ and hence
\[
H^{T}_\bullet(\BR^d_{\leq \lambda}) = \bigoplus_{\mu \leq \lambda} H^\bullet_T(\pt)\cdot [\overline\BR^d_{\mu}],
\]
where $\overline \BR_\mu^d$ is the closure of $\BR^d_\mu$ in $\BR^d_{\leq\lambda}$.
Putting $\gamma=[\widetilde{C}_{\leq\mu}]$ in \eqref{pullpushcal}, and using $\widetilde{\BS}_{\leq \mu}=\rho_\mu^{-1}\BT^d_{\leq \mu}/\widetilde{\BR}^d_{\leq\mu}$ again, we obtain :
\[
\rho_{\mu*}\big(e(\widetilde{\BS}_{\leq\mu}) \cap [\widetilde{C}_{\leq\mu}]\big)
= z_\lambda^* \,\rho'_{\mu*}[\widetilde{\BR}^d_{\leq \mu}] = z_\lambda^*[\overline\BR^d_{\mu}],
\]
where the last equality follows from the fact that $\widetilde{\BR}^d_{\leq \mu}$ is mapped birationally to $\overline\BR^d_{\mu}$ under $\rho_\mu'$. This proves (2) $\subset$ (3) and also establishes that the class $\rho_{\mu*}\big(e(\widetilde{\BS}_{\leq\mu}) \cap [\widetilde C_{\leq\mu}]\big)$ is independent of the choice of resolutions.

To see that the elements $\{\rho_{\mu*}\big(e(\widetilde{\BS}_{\leq\mu}) \cap [\widetilde{C}_{\leq\mu}]\}_{\mu \leq \lambda}$ are linearly independent over $H^\bullet_T(\pt)$, it suffices to show that both the Gysin map $H^{T}_\bullet(\BT^d_{\leq \lambda}) \to H^{T}_\bullet(C_{\leq \lambda})$ and the pushforward $H^{T}_\bullet(\BR^d_{\leq \lambda}) \to H^{T}_\bullet(\BT^d_{\leq \lambda})$ are injective.

The former is clear because $\BT^d_{\leq \lambda}$ is a vector bundle over $C_{\leq \lambda}$. For the latter, note that there is an affine stratification:
\[
\BT^d_{\leq \lambda} = \bigsqcup_{\mu \leq \lambda} \BT^d_\mu,
\]
which is compatible with \eqref{Rdecom}. Therefore, it suffices to show that each pushforward $H^{T}_\bullet(\BR^d_\mu) \to H^{T}_\bullet(\BT^d_\mu)$ is injective. This is equivalent to verifying that the element $e(\BS_\mu) = e(\BT^d_\mu / \BR^d_\mu) \in H^\bullet_T(C_\mu) \cong H^\bullet_T(\pt)$ is nonzero. However, by \Cref{Nweights}, the fibre of $\BS_\mu$ at $t^\mu$ is a $T$-representation with no zero weights, hence $e(\BS_\mu) \neq 0$.

Finally, it is clear that the subspace defined in (3) is contained in the subspace defined in (1).
\end{proof}

We introduce the notation
\begin{equation*}
e(\BS)\cap H^{G}_\bullet(\Gr_G) \coloneqq \sum_{\lambda\in \Lambda^+} \operatorname{Im}\left( e(\widetilde{\BS}_{\leq\lambda})\cap H^{G}_\bullet(\widetilde{C}_{\leq\lambda}) \to H^{G}_\bullet(\Gr_G) \right),
\end{equation*}
and similarly for the versions of $G\times \Cxh$, $T$, or $T\times \Cxh$-equivariant Borel--Moore homology.

Now, the following theorem follows immediately.

\begin{thm}[=\Cref{Introthm3}]
We have
\begin{equation*}
\sA_{G,\bN} = e(\BS_{G,\bN}) \cap H^{G}_\bullet(\Gr_G) = \left( e(\BS_{G,\bN}) \cap H^{T}_\bullet(\Gr_G) \right)^W,
\end{equation*}
and similarly
\begin{equation*}
\sA_{G,\bN}^\hbar = e(\BS_{G,\bN}) \cap H^{G\times\Cxh}_\bullet(\Gr_G) = \left( e(\BS_{G,\bN}) \cap H^{T\times\Cxh}_\bullet(\Gr_G) \right)^W.
\end{equation*}
\end{thm}

The following proposition follows from \Cref{independentResol} and \cite{BFN}. An independent proof is given in \Cref{AppendixStability}.

\begin{prop}\label{stableprod}
    The subspaces $e(\BS) \cap H^{G}_\bullet(\Gr_G)$ and $e(\BS) \cap H^{G\times\Cxh}_\bullet(\Gr_G)$ are stable under the convolution product.
\end{prop}

We recall that $\sA_T$ can be regarded as the ring $\oh_{T^*\check T}$ of regular functions on $T^*\check T$ (see, e.g., \cite{BFN}), and that $\sA_T^\hbar=\oh_{T^*\check T}^\hbar$ is its deformation quantization. We use the subscript ``$\loc$'' to denote localization at the multiplicative set of nonzero homogeneous elements of $H^\bullet_K(\pt)$, where $K$ is the equivariant group under consideration.

We remark that localization
\begin{equation}\label{localizeindepofN}
\sA^\hbar_{G,\bN,\loc}=(\oh_{T^*\check T}^\hbar)_\loc
\end{equation}
is independent of $\bN$ (cf. \cite[Lemma 5.9]{BFN}).

\section{Twisting maps and twisted linearity}\label{twisting}

\subsection{Twisting maps}

Let $X$ be a smooth quasiprojective variety with a $G$-action. Consider the space $G_\ok\times_{G_\oh}X$, where $G_\oh$ acts on $X$ via the homomorphism $\ev_{t=0}:G_\oh\to G$. It is equipped with a left $G_\oh$-action via $h\cdot [g,x]=[hg,x]$. There is a $G_\oh$-equivariant morphism $G_\ok\times_{G_\oh}X\to \Gr_G$ by sending $[g,x]\mapsto [g]\in \Gr_G$. 

Let $K\subset G$ be a subgroup, $C$ be a projective variety equipped with a $K\times \Cxh$-action, and $\vartheta:C\to \Gr_G$ be a $K\times\Cxh$-equivariant morphism. We write 
\[
G_\ok^{C}=C\times_{\Gr_G}G_\ok.
\]
Note that $G_\ok^C$ is a principal $G_\oh$-bundle over $C$.

\subsubsection*{Convolution}
Consider the convolution diagram,
\[
\begin{tikzcd}
C\times X& (G_\ok^C/K_1)\times X \arrow[l, "p_{C,X}"'] \arrow[r, "q_{C,X}"] & G^C_\ok\times_{G_\oh}X,
\end{tikzcd}
\]
where $K_1=\ker(G_\oh\to G)$.
\begin{df}\label{Twistingdef}
The homomorphism
\begin{equation*}
\tw_G=\tw_{G,C,X}=(q_{C,X}^*)^{-1}\circ p_{C,X}^*: H^{K\times\Cxh}_\bullet(C) \otimes_{\C[\hbar]} H^{G\times\Cxh}_\bullet(X)
\longrightarrow H^{K\times\Cxh}_\bullet(G_\ok^C \times_{G_\oh} X).
\end{equation*}
is called the \emph{twisting map}, where
    \begin{align*}
    p_{C,X}^* : H^{K\times\Cxh}_r(C) \otimes_{\C[\hbar]} H^{G\times\Cxh}_s(X)  \longrightarrow H^{K\times G\times\Cxh}_{r+s+2\dim G}\left((G_\ok^{C} / K_1) \times X\right).
    \end{align*}
    is the pullback along $p_{C,X}$, and
    \begin{align*}
    q_{C,X}^* :H^{G\times\Cxh}_{r+s}\left(G_\ok^{C} \times_{G_\oh} X\right)  \xlongrightarrow{\simeq} H^{K\times G\times\Cxh}_{r+s+2\dim G}\left((G_\ok^{\leq \lambda} / K_1) \times X\right).
    \end{align*}
    is the pullback along $q_{C,X}$, with respect to the inclusion $K\cong K\times\{e\} \subset K \times G$.
\end{df}
From now on, we consider the case where $C$ is smooth. Then $G_\ok^{C} \times_{G_\oh} X$ is also smooth. By Poincar\'e duality, there is an induced homomorphism
\begin{equation*}
\tw_G: H^{K\times\Cxh}_\bullet(C) \otimes_{\C[\hbar]} H_{G\times\Cxh}^\bullet(X)
\longrightarrow H_{K\times\Cxh}^\bullet(G_\ok^C \times_{G_\oh} X),
\end{equation*}
which we also denote by the same symbol $\tw_G$
\begin{lem}\label{twistingfunctorial}
Suppose $C'\to C$ is a $K\times \Cxh$-equivariant morphism. Then the following diagram commutes:
\begin{equation*}
    \begin{tikzcd}
H^{K\times\Cxh}_\bullet(C') \otimes_{\C[\hbar]} H_{G\times\Cxh}^\bullet(X) \arrow[r, "\tw_G"]\arrow[d] & H_{T\times\Cxh}^\bullet\big(G_\ok^{C'} \times_{G_\oh} X\big) \arrow[d]\\
H^{K\times\Cxh}_\bullet(C) \otimes_{\C[\hbar]} H_{G\times\Cxh}^\bullet(X) \arrow[r, "\tw_G"] & H_{T\times\Cxh}^\bullet\big(G_\ok^{C} \times_{G_\oh} X\big)
    \end{tikzcd}
\end{equation*}
Here, the vertical maps are pushforwards.
\end{lem}
\begin{lem}\label{twistingeuler}
Suppose $e\in H_{T\times\Cxh}^\bullet(C)$, $\Gamma \in H^{T\times\Cxh}_\bullet(C)$ and $\alpha \in H^{G\times\Cxh}_\bullet(X)$, we have
\begin{equation*}
\tw_G\left( (e \cap \Gamma) \otimes \alpha\right)
=
e \cup \tw_G(\Gamma \otimes \alpha),
\end{equation*}
where $e$ on the right-hand side is understood as the pullback along the projection $\widetilde{G}_\ok^{C} \times_{G_\oh} X \longrightarrow \widetilde{C}$.
\end{lem}
Suppose that $\vartheta:C\to \Gr_G$ is the inclusion of the $T$-fixed point $t^\lambda\in \Gr_G$. Then we have
\[
G_\ok^C \times_{G_\oh} X=T_\ok^C \times_{T_\oh} X\cong X^\lambda,
\]
where $X^\lambda$ is $X$ with the twisted $T\times\Cxh$-action given by $(z,z_\hbar)\cdot x=z\cdot(\lambda(z_\hbar)\cdot x)$.
\begin{prop}\label{twistingcompat}
    Suppose $K=T$, and that the image of $\vartheta:C\to \Gr_G$ is contained in $\Gr_G$, then $\tw_G$ is equal to the composition
    \begin{equation}\label{twistingcompatdiag}
    H^{T\times\Cxh}_\bullet(C) \otimes_{\C[\hbar]} H^{G\times\Cxh}_\bullet(X)\subset H^{T\times\Cxh}_\bullet(C) \otimes_{\C[\hbar]} H^{T\times\Cxh}_\bullet(X)
\xlongrightarrow{\tw_T} H^{K\times\Cxh}_\bullet(X^\lambda).
    \end{equation}
\end{prop}
\subsection{Twisted linearity}
It is clear that $\tw_G(-\otimes-)$ is $H_{K\times\Cxh}^\bullet(\pt)$-linear in the first argument. However, linearity in the second argument is more subtle, which we will explain now. Let
\begin{equation*}\label{secondmodule}
u^*:H_{G}^\bullet(\pt)\longrightarrow H_{K\times\Cxh}^\bullet(C)
\end{equation*}
be the map which sends $c_k(V)\in H_G^\bullet(\pt)$ to the cohomology class
\[u^*(c_k(V))=c_k(G_\ok^C\times_{G_\oh}V)\in H_{G_\oh\times\Cxh}^\bullet(C).\]
for any $G$-representation $V$, and $G_\ok\times_{G_\oh} V$ is the associated vector bundle on $\Gr_G$. 
\begin{prop}\label{twistinglinear}
Let $P \in H_{G \times \Cxh}^\bullet(\pt)$. Then
\begin{equation*}
\tw_G(\Gamma \otimes (P \cap \alpha)) = \tw_G((u^* P \cap \Gamma) \otimes \alpha) = u^* P \cup \tw_G(\Gamma \otimes \alpha),
\end{equation*}
where $u^* P$ on the right-hand side is understood as the pullback along the projection $\pr:\widetilde{G}_\ok^{C} \times_{G_\oh} X \longrightarrow \widetilde{C}$.
\end{prop}
\begin{proof}
We may assume that $P=c_k(V)$ for some $G$-representation $V$. Then $X\times V$ is a $G\times \Cxh$-equivariant vector bundle over $X$. The pullback $p_{C,X}^*(X\times V)$ is the $K\times G\times\Cxh$-equivariant vector bundle $G_\ok^C/K_1\times X\times V$, where $K$ acts trivially on $V$. This descends to the $K\times\Cxh$-equivariant vector bundle
\begin{equation}\label{eq:twistinglinear}
    (G_\ok^C/K_1)\times_G (X\times V)
\end{equation}
over $\widetilde{G}_\ok^{C} \times_{G_\oh} X$. On the other hand, there is an isomorphism
\[
G_\ok^C/K_1\times_{C}(G_\ok^C\times_{G_\oh}X)\cong G_\ok^C/K_1\times X
\]
of $G$-bundles over $G_\ok^C\times_{G_\oh}X$. As a result, \eqref{eq:twistinglinear} is isomorphic to
\[
\left(G_\ok^C/K_1\times_{C}(G_\ok^C\times_{G_\oh}X)\right)\times_G V,
\]
which is the pullback of $G_\ok^C\times_{G_\oh}V$ along the morphism $\widetilde{G}_\ok^{C} \times_{G_\oh} X \longrightarrow \widetilde{C}$.
\end{proof}
\subsubsection*{Twisting by a cocharacter}
Let $G=T$ be a torus. We now give another description of the twisting map. For $\lambda \in \Lambda$,  let $X^\lambda$ denote the fibre of $T_\ok \times_{T_\oh} X$ over $[t^\lambda]\in \Gr_T$. As a $T$-variety, $X^\lambda$ is isomorphic to $X$, but the loop rotation acts by
\[
z_\hbar \cdot x = \lambda(z_\hbar)\cdot x,
\qquad z_\hbar \in \Cxh,\ x\in X^\lambda.
\]

Consider the diagram
\[
\begin{tikzcd}[column sep=huge]
  & T \times X \ar[ld, "a"'] \ar[rd, "b"] & \\
  X && X^\lambda
\end{tikzcd}
\]
where $a$ is the projection onto the second factor and $b$ is induced by the original $T$-action on $X$. We let $T\times T\times \Cxh$ act on $T\times X$ by
\[
    (z_1,z_2,z_\hbar)\cdot (z,x)
    =
    \bigl(\lambda(z_\hbar)z_1zz_2^{-1},z_2x\bigr).
\]
Then $a$ and $b$ are equivariant for the homomorphisms $T\times T\times \Cxh\to T\times \Cxh$ given by
\[
\begin{aligned}
    (z_1,z_2,z_\hbar)&\mapsto (z_2,z_\hbar),\\
    (z_1,z_2,z_\hbar)&\mapsto (z_1,z_\hbar),
\end{aligned}
\]
respectively. Hence the twisting map
\[
\tw_T([t^\lambda] \otimes -) \colon
H_{T \times \Cxh}^\bullet(X)
\longrightarrow
H_{T \times \Cxh}^\bullet(X^\lambda)
\]
is given by
\[
H_{T \times \Cxh}^\bullet(X)
\xrightarrow{a^*}
H_{T \times T \times \Cxh}^\bullet(T \times X)
\xrightarrow{(b^*)^{-1}}
H_{T \times \Cxh}^\bullet(X^\lambda).
\]
The inverse of $b^*$ is $c^*$, where $c:X^\lambda\to T\times X$ is given by $x\mapsto (1,x)$ and is equivariant for
\[
T\times \Cxh\to T\times T\times \Cxh,
\qquad
(z,z_\hbar)\mapsto \bigl(z,\lambda(z_\hbar)z,z_\hbar\bigr).
\]
Moreover, $a\circ c=\operatorname{id}_X$, viewed as a map $X^\lambda\to X$, is equivariant for the automorphism
\[
T\times \Cxh\to T\times \Cxh,
\qquad
(z,z_\hbar)\mapsto \bigl(\lambda(z_\hbar)z,z_\hbar\bigr).
\]
\begin{prop}\label{abeliantwlinear}
Let $\Phi_\lambda \colon H_{T \times \Cxh}^\bullet(X) \to H_{T \times \Cxh}^\bullet(X^\lambda)$
be the homomorphism induced by the identity map $\operatorname{id} \colon X \to X^\lambda$, which is equivariant with respect to to the automorphism $(z,z_\hbar)\to (z\lambda(z_\hbar),z_\hbar)$ of $T\times \Cxh$. Then $\Phi_\lambda$ agrees with the twisting map $\tw_T([t^\lambda] \otimes -)$. In particular, for $P(a,\hbar) \in H_{T \times \Cx_\hbar}^\bullet(\pt)$ and $\alpha \in H_{T \times \Cx_\hbar}^\bullet(X)$, we have
\begin{equation*}
\tw_T([t^\lambda] \otimes P(a, \hbar) \alpha) = P(a + \lambda(\hbar), \hbar)\, \tw_T([t^\lambda] \otimes \alpha).
\end{equation*}
\end{prop}

\begin{rem}
    $\Phi_\lambda$ is the same as the twisted homorphism defined in \cite[Section 3.1]{Iritani}. 
\end{rem}

\section{Shift operators I: preparation}\label{preparation}
\subsection{Quantum cohomology}\label{qh}
\subsubsection*{Assumptions on the space}
Let $X$ be a smooth, semiprojective complex variety with an algebraic $G$-action. By \textit{semiprojective}, we mean that $X$ is quasiprojective and that the affinization map
\begin{equation*}
\mathrm{aff}: X \longrightarrow X^{\mathrm{aff}} \vcentcolon= \spec H^0(X, \mathcal{O}_X)
\end{equation*}
is projective. We assume that $X$ carries an algebraic $G$-action commuting with an auxiliary $\C^\times_{\mathrm{attr}}$-action satisfying:
\begin{enumerate}
  \item all $\C^\times_{\mathrm{attr}}$-weights on $H^0(X, \mathcal{O}_X)$ are non-positive; and
  \item the $\C^\times_{\mathrm{attr}}$-invariants satisfy $H^0(X, \mathcal{O}_X)^{\C^\times_{\mathrm{attr}}} = \C$.
\end{enumerate}

These conditions imply in particular that $X^{\C^\times_{\mathrm{attr}}}$ is compact. We refer to this $\C^\times_{\mathrm{attr}}$-action as the \emph{conical action}, not to be confused with the group of \emph{loop rotations}, $\C^\times_\hbar$, which acts trivially on $X$.

These assumptions ensure that $X$ has nice cohomological properties, including \Cref{equivformal} below. Note that in \cite{Iritani}, it is assumed that the conical group $\Cx_{\mathrm{attr}}$ is a subgroup of $G$, which is required for defining shift operators in their setting. However, we do not need this assumption in the Coulomb branch setting, as will be explained in \Cref{sectionshiftoperators}.

We remark that the $\Cx_{\mathrm{attr}}$-action is auxilliary. It is assumed in the paper in order to define the equivariant quantum cohomology, as well as the shift operators, specifically in \Cref{lemmacurveclass2}. However, the action is not used in the proofs of the main theorems in \Cref{propertyshiftoperators}.

\begin{prop}\label{equivformal}
  The $G$-action on $X$ is equivariantly formal.
\end{prop}

\begin{proof}
  See \cite[Proposition 2.1]{Iritani}. In \textit{loc.~cit.}, only abelian $G$ were considered, but the same proof applies to any reductive group $G$.
\end{proof}

\subsubsection*{Equivariant quantum cohomology}
We fix a graded $\C$-basis $\{a_i\}_{i\in I}$ of $H_G^\bullet(\pt)$ and a graded $H_G^\bullet(\pt)$-basis $\{\phi_j\}_{j=0}^N$ of $H_G^\bullet(X)$. We denote $\phi_{i,j}\coloneq a_i\phi_j$. Let $\{\tau^{i,j}\}$ be coordinates on $H_G^\bullet(X)$ dual to the basis $\{\phi_{i,j}\}$, and set $\tau = \sum_{i,j} \tau^{i,j} \phi_{i,j}$. We declare the degrees of the coordinates $\tau^{i,j}$ to be
\[
\deg \tau^{i,j} = 2 - \deg (\phi_{i,j}).
\]
We write $\C[[\tau]] = \C[[\tau^{i,j}]]$, where the odd variables (that is, those $\tau^{i,j}$ for which $\deg \phi_{i,j}$ is odd) anti-commute. In other words, this is the tensor product of the formal power series ring in the even variables with the exterior algebra in the odd variables. Let
\begin{equation*}
\C[q]=\C[q_X] = \C[H_2^{\text{ord},G}(X;\Z)]
\end{equation*}
be the group algebra of the abelian group $H_2^{\text{ord},G}(X; \Z)$. In other words, it is the $\C$-algebra generated by the symbols $q^\beta$ for $\beta\in H_2^{\text{ord},G}(X;\Z)$ subject to the relations $q^0=1$ and $q^{\beta_1}q^{\beta_2}=q^{\beta_1+\beta_2}$. The ring $\C[q]$ is equipped with a grading determined by
\begin{equation*}
\deg q^\beta = 2\langle c_1^G(X), \beta \rangle.
\end{equation*}

Let $\Eff(X) \subset H_2^{\text{ord}}(X; \Z)$ be the subset of effective curve classes. We denote $\iota_*: H_2^{\text{ord}}(X; \Z) \to H_2^{\text{ord},G}(X; \Z)$ the natural map. 

For a graded vector space $A$, we define $A[[q,\tau]]$ to be the graded completion of $A[[\tau]][q]$ along the direction of $\Eff(X)$. More explcitly, $A[[q,\tau]]$ is a graded vector space whose degree $n$ graded piece consists of formal sums
\begin{equation*}
\sum_{\beta, m} a_{\beta, m}q^\beta  \,  \prod_{\substack{i\in I \\ 0\leq j\leq N}} (\tau^{i,j})^{m_{i,j}}
\end{equation*}
satisfying the following conditions:
\begin{enumerate}
    \item $m=(m_{i,j})$ is a collection of nonnegative integers with finitely many of them being nonzero.
    \item $a_{\beta, m}\in A$ is a homogeneous element of degree $n-\sum m_{i,j} \deg \tau^{i,j} - \deg q^\beta$.
    \item There exists a finite subset $S \subset H_2^{\text{ord},K}(X; \Z)$ such that $a_{\beta, m} = 0$ unless $\beta \in S + \iota_*(\Eff(X))$.
\end{enumerate}
It can be checked that $A$ is a graded-algebra, then $A[[q,\tau]]$ is also a graded-algebra. 

\begin{df}
Let $K\subset G$ be a subgroup, then the \emph{$K$-equivariant quantum cohomology} of $X$ with $G$-equivariant Novikov variables is the ring
\[
  QH^\bullet_K(X)\colon = (H^\bullet_K(X)[[q,\tau]],\star_{\tau}).
\]
Here, the \emph{big quantum product} $\star_{\tau}$ is defined by
\begin{equation*}
  \gamma \star_{\tau} \gamma'
  = \sum_{\beta \in \Eff(X)} \sum_{n=0}^\infty
    \frac{q^{\beta}}{n!} \,
    \mathrm{PD} \circ \ev_{3*} \Bigl( \ev_1^* (\gamma) \, \ev_2^* (\gamma') \prod_{\ell=4}^{n+3} \ev_\ell^* (\tau) \cap [\overline{M}_{0,n+3}(X,\beta)]^\vir \Bigr)
\end{equation*}
for $\gamma, \gamma' \in H^\bullet_K(X)$. Here, $\overline{M}_{0,n+3}(X, \beta)$ is the moduli stack of genus-zero stable maps to $X$ with $n+3$ marked points and curve class $\beta$, and $[\overline{M}_{0,n+3}(X, \beta)]^\vir$ denotes its virtual fundamental class. Note that $\ev_3$ is proper, thanks to $X$ being semiprojective. The product $\star_{\tau}$ on a general element of $H^\bullet_K(X)[[q, \tau]]$ is then defined termwise on its power series expansion.
\end{df}
One can similarly consider $K'$-equivariant Novikov variables for any subgroup $K'\subset G$. Unless otherwise specified, however, we always use the $G$-equivariant Novikov variables.

\subsubsection*{Quantum connection}
We consider the trivial $\C^\times_\hbar$-action on $X$. We follow the notations in \cite[Section 2.1]{IritaniFourier}.
\begin{df}[Quantum connection]\label{def:quantum_connection}
The \emph{equivariant quantum connections} are the operators 
$$\nabla_{\tau^{i,j}}, \nabla_{\hbar\partial_\hbar}, \nabla_{D}:H_{G \times \Cxh}^\bullet(X)[[q_K, \tau]] \to H_{G \times \Cxh}^\bullet(X)[[q_K, \tau]] [\hbar^{-1}],$$
defined as
\begin{align}\label{qconnformula}
\nabla_{\tau^{i,j}} &= \partial_{\tau^{i,j}} - \hbar^{-1}(\phi_{i,j} \star_{\tau}), \nonumber\\
\nabla_{\hbar\partial_\hbar} & = \hbar\partial_\hbar + \hbar^{-1}(E\star_\tau) +\mu_X , \\
\nabla_{Dq\partial_q}  &= Dq\partial_q - \hbar^{-1}(D\star_\tau), \nonumber
\end{align}
where $D\in H^2_G(X)$ and $Dq\partial_q$ is the derivation on $\C[[q_K]]$ with $Dq\partial_q q^\beta=\langle D,\beta\rangle q^\beta$, while $E_X$ is the \emph{Euler vector field} and $\mu_X$ is the \emph{grading operator}, defined respectively by the formulas
\[E_X=c_1^G(X)+\sum_{i,j}\frac{\deg(\tau^{i,j})}{2} \tau^{i,j}\phi_{i,j}\]
and
\begin{equation}\label{gradingop}
\mu_X(\phi_{i,j})=\frac 12\left(\deg(\phi_{i,j})-\dim X\right) \phi_{i,j}.
\end{equation}
\end{df}

\begin{df}[Fundamental solution]\label{def:fundamental_solution}
The \emph{fundamental solution} to the quantum differential equation is the operator
\begin{equation*}
\mathbb{M}_X : H_{G \times \Cxh}^\bullet(X)[[q_G, \tau]][[\hbar^{-1}]] \to H_{G \times \Cxh}^\bullet(X)[[q_G, \tau]][[\hbar^{-1}]]
\end{equation*}
defined by
\begin{equation*}
\mathbb{M}_X(\gamma) 
= \gamma+\sum_{0\neq \beta \in \Eff(X)}
    \sum_{n=0}^\infty \frac{q^\beta}{n!} \,
    \mathrm{PD} \circ \ev_{2*} \left( \frac{\ev_1^* (\gamma)}{\hbar - \psi_1}
    \prod_{\ell=3}^{n+2} \ev_\ell^* (\tau) \cap [\overline{M}_{0,n+2}(X, \beta)]^{\vir} \right),
\end{equation*}
where $\psi_1$ is the equivariant first Chern class of the universal cotangent line bundle at the first marked point.
\end{df}

The following proposition is well-known (see \cite{GiventalequivGW,Pandharipande1998,CrepantToricCompleteInt,IritaniFourier}).
\begin{prop}
The operator $\mathbb{M}_X$ is invertible and satisfies the identities 
\begin{align}\label{qconnintertwine}
\nabla_{\tau^{i,j}} \circ \mathbb{M}_X &= \mathbb{M}_X \circ \partial_{\tau^{i,j}}, \nonumber\\
\nabla_{\hbar\partial_\hbar} \circ \mathbb{M}_X &= \mathbb{M}_X\circ \left(\hbar\partial_\hbar + \hbar^{-1}(c_1^G(X)\cup) + \mu_X\right),\\
\nabla_{Dq\partial_q}\circ \mathbb{M}_X &= \mathbb{M}_X\circ \left(Dq\partial_q-\hbar^{-1}(D\cup)\right). \nonumber
\end{align}
\end{prop}

\subsection{Universal \texorpdfstring{$G$}{G}-torsor}\label{univbun}
We fix the standard affine chart $\C = \spec \C[t] \subset \PP^1$ and take the base point $0 \in \C \subset \PP^1$. By the Beauville--Laszlo theorem \cite{BL} (see also \cite{Zhu}), the affine Grassmannian $\Gr_G$ represents the functor on the category of $\C$-scheme.
\begin{equation*}
Z \longmapsto 
\left\{
(\mathcal{P}, \varphi) \,\middle|\,
\begin{array}{l}
\mathcal{P} \text{ is a } G\text{-torsor over } \PP^1_Z, \\
\varphi_\infty : \mathcal{P}|_{(\PP^1 \setminus \{0\})_Z} \xrightarrow{\sim} (\PP^1 \setminus \{0\})_Z \times G \text{ is a trivialization}
\end{array}
\right\} \big/ \cong.
\end{equation*}
In particular, there exists a universal $G$-torsor $\sE \to \Gr_G \times \PP^1$. If $p\in \Gr_G$, we denote $\sE_p$ the restriction of $\sE$ to $p\times\PP^1$. 

We now make the correspondence between a map $f : Z \to \Gr_G$ and the pair $(\mathcal{P}, \varphi_\infty)$ more explicit. Suppose first that $f$ factors through a map $\widetilde{f} : Z \to G_\ok$. Then $\mathcal{P}$ is the unique (up to isomorphism) $G$-torsor over $\PP^1_R$ with local trivializations 
\begin{equation*}
\varphi_\infty : \mathcal{P}|_{Z \times (\PP^1 \setminus \{0\})} \xrightarrow{\sim} (Z \times (\PP^1 \setminus \{0\})) \times G,
\end{equation*}
and
\begin{equation*}
\varphi_0 : \mathcal{P}|_{Z \times \spec \oh} \xrightarrow{\sim} (Z \times \spec \oh) \times G,
\end{equation*}
such that the induced automorphism
\begin{equation*}
\varphi_\infty|_{Z \times \spec \ok} \circ (\varphi_0|_{Z \times \spec \ok})^{-1} : (Z \times \spec \ok) \times G \to (Z \times \spec \ok) \times G
\end{equation*}
is given by left multiplication by the map
\begin{equation*}
(Z \times \spec \ok) \times G \to Z \times \spec \ok \xrightarrow{\widetilde{f}} G.
\end{equation*}
The existence of such a torsor $\mathcal{P}$ is the content of the Beauville--Laszlo theorem. In general, there exists an fppf cover $\bigsqcup Z_i \to Z$ such that the pullback of $f$ to each $Z_i$ factors through $G_\ok$. One can then construct $\mathcal{P}$ and $\varphi_\infty$ by gluing the corresponding $G$-torsors and local trivializations over the cover, using fppf descent.

Under the above correspondence, the subgroup $G_{\C[t^{-1}]} \subset G_\ok$ acts on $\Gr_G$ by modifying the section $\varphi$ over $\PP^1 \setminus \{0\}$, while loop rotation acts by scaling the coordinate on $\PP^1$. More explicitly, for $z \in \C^\times_\hbar$, let $m_z : \PP^1 \to \PP^1$ be the morphism given by $t \mapsto z^{-1}t$. There is an action of $G_{\C[t^{-1}]} \rtimes \C^\times_\hbar$ on $\Gr_G$ given by
\begin{equation*}
g=(g_1,z) \cdot (\mathcal{P}, \varphi) = \left( m_{z}^* \mathcal{P},\ g \cdot m_{z}^* \varphi \right),
\end{equation*}
where $g_1$ is regarded as a morphism $\spec \C[t^{-1}] \to G$.

The action of $G_{\C[t^{-1}]} \rtimes \C^\times_\hbar$ lifts to the universal torsor $\sE$, in the sense of the next lemma, which follows from the functorial description of $\Gr_G$.

\begin{lem}
Suppose $f : Z \to \Gr_G$ is a morphism from a scheme $Z$ corresponding to the pair $(\mathcal{P}, \varphi_\infty)$, and let $g = (g_1, z) \in G_{\C[t^{-1}]} \rtimes \C^\times_\hbar$. If $(\mathcal{P}', \varphi'_\infty)$ corresponds to the map $g \cdot f : Z \to \Gr_G$, then there exists a $G$-equivariant isomorphism
\begin{equation*}
g_f : \mathcal{P} \xrightarrow{\sim} m_z^*\mathcal{P}'
\end{equation*}
such that $m_z^*\varphi'_\infty \circ g_f = g_1\cdot \varphi_\infty$.
\end{lem}

In particular, if $H \subset G_{\C[t^{-1}]} \rtimes \C^\times_\hbar$ and $f : Z \to \Gr_G$ is an $H$-equivariant morphism corresponding to the pair $(\mathcal{P}, \varphi_\infty)$, then there is an induced $H$-action on $\mathcal{P}$ such that the projection $\mathcal{P} \to Z \times \PP^1$ is $H$-equivariant (where $\C^\times_\hbar$ acts on $\PP^1$ by scalar multiplication).

Strictly speaking, we have defined the action on $\sE$ only at the level of $\C$-points of $G_{\C[t^{-1}]} \rtimes \C^\times_\hbar$. In \Cref{Gtorsorappendix}, we will describe the action of $G_{\C[t^{-1}]} \rtimes \C^\times_\hbar$ as a group scheme by considering its $R$-points, and in particular show that the actions defined above are algebraic.

\begin{exmp}\label{CompareIritani1}
Suppose $G = T = \C^\times$ and $\lambda = 1$. Then
\begin{equation*}
\sE_t \coloneq \sE|_{t} \cong \oh_{\PP^1}(-1)^\times \cong \C^2 \setminus \{0\}.
\end{equation*}
Let $u, v$ be coordinates on $\C^2$, so that the projection $\C^2 \setminus \{0\} \to \PP^1$ is given by $t = v/u$. Indeed, there are trivializations
\begin{equation*}
\varphi_0\colon \oh_{\PP^1}(-1)^\times|_{\spec \C[t]} \xrightarrow{\sim} \spec \C[t] \times \C^\times,
\end{equation*}
\begin{equation*}
\varphi_\infty\colon \oh_{\PP^1}(-1)^\times|_{\spec \C[t^{-1}]} \xrightarrow{\sim} \spec \C[t^{-1}] \times \C^\times,
\end{equation*}
given by
\begin{equation*}
\varphi_0(u, v) = \left( \frac{v}{u}, u \right), \qquad
\varphi_\infty(u, v) = \left( \frac{u}{v}, v \right).
\end{equation*}
It is clear that the transition function $\varphi_\infty \circ \varphi_0^{-1}$ corresponds to left multiplication by $t = v/u$. The subgroup\footnote{$\C^\times$ and $\C^\times_{\C[t^{-1}]}$ have the same underlying $\C$-points} $\C^\times \times \Cxh \subset \C^\times_{\C[t^{-1}]} \rtimes \Cxh$ acts on $\oh_{\PP^1}(-1)^\times$ by $(g, z_\hbar) \cdot (u, v) = (z_\hbar gu, g v)$.

More generally, if $G = T$ is abelian and $\lambda \in \Lambda$ is arbitrary, then
\begin{equation*}
\sE_{t^\lambda} \cong \oh_{\PP^1}(-1)^\times \times_{\C^\times} T,
\end{equation*}
where $\C^\times$ acts on $T$ via $\lambda$, and $T \times \Cxh$ acts on $\sE_{t^\lambda}$ via $(g, z_\hbar) \cdot (u, v, g') = (z_\hbar u, v, gg')$.
\qed
\end{exmp}

\subsection{Seidel spaces}
For each $\lambda \in \Lambda$, we fix an $\BI$-equivariant resolution of singularities $\rho_\lambda \colon \widetilde{C}_{\leq \lambda} \to C_{\leq \lambda}$ satisfying the conditions in \Cref{existresol}. Let $\sE_{\leq \lambda}$ be the principal $G$-bundle on $\widetilde{C}_{\leq \lambda} \times \PP^1$ induced by the morphism
\begin{equation*}
\widetilde{C}_{\leq \lambda} \xrightarrow{\rho_\lambda} C_{\leq \lambda} \subset \Gr_G.
\end{equation*}
If $p \in \widetilde{C}_{\leq \lambda}$, we write $\sE_p$ for the restriction of $\sE_{\leq \lambda}$ to $p \times \PP^1$. In particular, there is a canonical isomorphism $\sE_p \cong \sE_{\rho_\lambda(p)}$.

If $X$ is a $G$-variety, we call
\begin{equation*}
\sE_{\leq \lambda}(X) := \sE_{\leq \lambda} \times_G X
\end{equation*}
the \emph{Seidel space associated to $X$ and $\widetilde{C}_{\leq \lambda}$}.

By construction, there is a $T \times \C^\times_\hbar$-action on $\sE_{\leq \lambda}(X)$ such that the projection
\begin{equation*}
\pi : \sE_{\leq \lambda}(X) \to \widetilde{C}_{\leq \lambda} \times \PP^1
\end{equation*}
is equivariant, where $\C^\times_\hbar$ acts on $\PP^1$ by scalar multiplication. Moreover, if $\lambda \in \Lambda^+$, this action can be upgraded to a $G \times \C^\times_\hbar$-action.
\begin{exmp}\label{CompareIritani2}
The bundle $\sE_{t^\lambda}(X)$ is isomorphic to $\oh_{\PP^1}(-1)^\times \times_{\Cx} X$, where $\C^\times$ acts on $X$ via the cocharacter $\lambda\colon \C^\times \to T$. The action of $T \times \Cxh$ on $\sE_{t^\lambda}(X)$ is given by (cf.~\Cref{CompareIritani1})
\begin{equation*}
(g, z_\hbar) \cdot (u, v, x) = (z_\hbar u, v, gx).
\end{equation*}
This space is $T \times \Cxh$-equivariantly isomorphic to the Seidel space $E_\lambda$ defined in Section~3.2 of \cite{Iritani}, but with the roles of $0$ and $\infty$ swapped.\qed
\end{exmp}

\subsubsection*{Section classes}
An effective curve class $\beta \in H_2^{\text{ord}}(\sE_{\leq \lambda}(X); \Z)$ is called a \emph{section class} if $\pi_* \beta = [\pt \times \PP^1]$. Denote by $\Eff(\sE_{\leq \lambda}(X))^{\mathrm{sec}}$ the subset of section classes.

The projection $\sE_{\leq \lambda}\times X \to X$ is $G$-equivariant, and hence induces a pushforward map in equivariant homology:
\begin{equation*}
H_2^{\text{ord}}(\sE_{\leq \lambda}(X); \Z)\,\longrightarrow H_2^{\text{ord},G}(X; \Z).
\end{equation*}
For any $\beta \in H_2^{\text{ord}}(\sE_{\leq \lambda}(X); \Z)$, we denote by $\overline{\beta}\in H^{\text{ord},G}_2(X,\Z)$ its image under this map.

Let $T^{\mathrm{vert}} \sE_{\leq \lambda}(X) \coloneqq \ker(d\pi) \cong \sE_{\leq \lambda}(TX)$ be the relative tangent bundle of $\pi$. The following lemma is immediate.

\begin{lem}\label{degree}
$\deg(q^{\overline{\beta}}) = 2\, c_1(T^{\mathrm{vert}} \sE_{\leq \lambda}(X)) \cdot \beta$.
\end{lem}

\subsubsection*{Fibres at $0$ and $\infty$}
We denote
\begin{equation*}
\iota_0:\mathcal X_{\leq \lambda, 0} \coloneqq \pi^{-1}(\widetilde{C}_{\leq \lambda} \times \{0\}) \hookrightarrow \sE_{\leq \lambda}(X),
\end{equation*}
\begin{equation*}
\iota_\infty:\mathcal{X}_{\leq \lambda, \infty} \coloneqq \pi^{-1}(\widetilde{C}_{\leq \lambda} \times \{\infty\}) \hookrightarrow \sE_{\leq \lambda}(X).
\end{equation*}
We let $\widetilde G_\ok^{\leq\lambda}$ denote $\widetilde G_\ok^{\widetilde{C}_{\leq\lambda}}=G_\ok\times_{\Gr_G} \widetilde{C}_{\leq\lambda}$. The following lemma is immediate from the construction in \Cref{univbun}.
\begin{lem}\label{zerofibre}
There are $T \times \C^\times_{\hbar}$-equivariant isomorphisms
\begin{align*}
    \mathcal{X}_{\leq \lambda, 0} &\cong \widetilde G_\ok^{\leq\lambda} \times_{G_\oh} X, \\
    \mathcal{X}_{\leq \lambda, \infty} &\cong \widetilde{C}_{\leq \lambda} \times X.
\end{align*}
Moreover, these isomorphisms are $G \times \C^\times_{\hbar}$-equivariant if $\lambda \in \Lambda^+$.
\end{lem}
In particular, $\rho_{X,\lambda} \colon \mathcal{X}_{\leq \lambda, 0} \to G_\ok^{\leq\lambda}\times_{G_\oh}X$ is a resolution of singularities.

\begin{df}\label{tauhat}
We define a map 
\begin{equation}\label{tauhatmap}
H_{G \times \Cxh}^\bullet(X) \to H_{T \times \Cxh}^\bullet(\sE_{\leq \lambda}(X)), \qquad \tau \mapsto \hat{\tau},
\end{equation}
as the composition of the following maps:
\begin{equation*}
H_{G\times\Cxh}^\bullet(X) 
\xrightarrow{\pr_X^*} H_{T \times G\times \Cxh}^\bullet(\sE_{\leq \lambda} \times X) \xrightarrow{\sim} H_{T\times\Cxh}^\bullet(\sE_{\leq \lambda}(X)),
\end{equation*}
where $G$ acts on $\sE_{\leq \lambda} \times X$ by $g\cdot (s,x)=(sg^{-1},gx)$.
\end{df}
The following lemma is clear.
\begin{lem}\label{tauhat=twisting}
The restrictions of $\hat\tau$ to $\cX_{\leq \lambda,0}$ and $\cX_{\leq \lambda,\infty}$ are respectively the images of $1\otimes \tau$ under the twisting map
\[
\tw_G\colon
H^\bullet_{T \times \Cxh}(\widetilde C_{\leq\lambda})
\otimes_{\C[\hbar]}
H^\bullet_{G \times \Cxh}(X)
\to
H^\bullet_{T \times \Cxh}(G_\ok^{\leq\lambda}\times_{G_\oh} X)
\]
and of $\tau$ under the pullback
\[
H_{G\times \Cxh}^\bullet(X)
\to
H_{T\times \Cxh}^\bullet(\widetilde C_{\leq \lambda} \times X).
\]
\end{lem}
\begin{rem}
    It follows from \Cref{tauhat=twisting} that when $G = T$ is abelian, the class $\hat{\tau}$ agrees with~\cite[Notation 3.8]{Iritani}, except that the roles of the zero and infinity fibres are reversed (cf.~\Cref{CompareIritani2}).
\end{rem}

\subsubsection*{Moduli spaces and virtual fundamental classes}
Let $\beta \in \Eff(\sE_{\leq \lambda}(X))^{\sec}$. Recall that $\overline{M}_{0,n+2}(\sE_{\leq \lambda}(X), \beta)$ denotes the moduli stack of genus-zero stable maps to $\sE_{\leq \lambda}(X)$ with $n+2$ marked points and curve class $\beta$. A typical object in this moduli stack is a stable map
\begin{equation*}
\sigma \colon (\Sigma, y_0, y_\infty,y_1,\dots,y_n) \to \sE_{\leq \lambda}(X),
\end{equation*}
where $\Sigma$ is a genus-zero nodal curve and $y_0, y_\infty,y_1\dots,y_n \in \Sigma$ are the marked points.

Since $\sE_{\leq \lambda}(X)$ is smooth, the moduli stack admits a virtual fundamental class (\cite{Intrinsic}),
\begin{equation*}
[\,\overline{M}_{0,n+2}(\sE_{\leq \lambda}(X), \beta)]^{\vir}\in H^{T\times\Cx_\hbar}_{\bullet}(\overline{M}_{0,n+2}(\sE_{\leq \lambda}(X), \beta)).
\end{equation*}

\begin{df}
Let $\cM_{\leq \lambda}(X, \beta)_n \subset \overline{M}_{0,n+2}(\sE_{\leq \lambda}(X), \beta)$ denote the substack consisting of those stable maps $\sigma$ such that $\sigma(y_0)$ lies over $0 \in \PP^1$ and $\sigma(y_\infty)$ lies over $\infty \in \PP^1$. In other words, we have the fibre diagram:
\begin{equation}\label{PBM}
\begin{tikzcd}
\cM_{\leq \lambda}(X, \beta)_n \ar[r, "j_\cM"] \ar[d] &
\overline{M}_{0,n+2}(\sE_{\leq \lambda}(X), \beta) \ar[d, "{(\pr_{\PP^1} \circ \EV_0 , \pr_{\PP^1} \circ \EV_\infty)}"] \\
\{(0, \infty)\} \ar[r,"j"] & \PP^1 \times \PP^1
\end{tikzcd}
\end{equation}
Here, $\EV_0$ and $\EV_\infty$ denote the \emph{evaluation maps at the marked points $y_0$ and $y_\infty$}, respectively. We also denote
\[
\cM_{\leq\lambda}(X)_n:=\bigsqcup_{\beta\in\Eff(\sE_{\leq\lambda}(X))^{\sec}} \cM_{\leq\lambda}(X,\beta)_n.
\]
\end{df}
In particular, the evaluation maps restrict to
\[
\ev_0: \cM_{\leq\lambda}(X)_n \longrightarrow \mathcal X_{\leq\lambda,0},\qquad
\ev_\infty: \cM_{\leq\lambda}(X)_n \longrightarrow \mathcal X_{\leq\lambda,\infty}.
\]
In order to distinguish between the evaluation maps on $\mathcal{M}$ and $\overline M$, we use the different symbols $\ev$ and $\EV$, respectively.

We define the virtual fundamental class of $\cM_{\leq \lambda}(X, \beta)_n$ to be
\begin{equation*}
[\cM_{\leq \lambda}(X, \beta)_n]^{\vir} := j^!\left[\, \overline{M}_{0,n+2}(\sE_{\leq \lambda}(X), \beta) \right]^{\vir},
\end{equation*}
where $j^!$ is the refined Gysin pullback
\begin{equation*}
j^! \colon H^{T\times\Cx_\hbar}_\bullet\big(\overline{M}_{0,n+2}(\sE_{\leq \lambda}(X), \beta)\big) \to H^{T\times\Cx_\hbar}_{\bullet-4}(\cM_{\leq \lambda}(X, \beta)_n)
\end{equation*}
defined via the fibre diagram \eqref{PBM} (see [8.3.21] in \cite{ChrissGinzburg} or part 3(ii)(b) of \cite{BFN}). The degree of the virtual fundamental class equals twice the virtual dimension of $\cM_{\leq \lambda}(X, \beta)_n$, which is computed as follows (cf.\ \cite{Fulton-Pandharipande,Cox-Katz}):
\begin{align*}
\mathrm{vdim}\big(\cM_{\leq \lambda}(X, \beta)_n\big)
&= \dim \sE_{\leq \lambda}(X) + c_1\big(T \sE_{\leq \lambda}(X)\big) \cdot \beta + n - 3 \nonumber \\
&= \dim X + \dim C_\lambda + c_1\big(T^{\mathrm{vert}} \sE_{\leq \lambda}(X)\big) \cdot \beta + n \nonumber \\
&= \dim X + \dim C_\lambda + c_1^G(X) \cdot \overline\beta + n.
\end{align*}
If $p$ is a point in $\Gr_G$ or in $\widetilde{C}_{\leq \lambda}$, the substack $\cM_p(X, \beta)_n \subset \overline{M}_{0,n+2}(\sE_p(X), \beta)$ is defined similarly. We write $\pr_{\widetilde{C}_{\leq\lambda}}:\cM_{\leq\lambda}(X,\beta)_n\to {\widetilde{C}_{\leq\lambda}}$ for the projection.

\section{Shift operators II: definition}\label{sectionshiftoperators}
For graded vector spaces $A=\oplus A^i$, $B=\oplus B^i$, we write
\[
\Hom^\bullet(A,B)=\bigoplus_k\Hom^k(A,B),\qquad \Hom^k(A,B)=\bigoplus_i\Hom(A^i,B^{i+k}).
\]
for the graded vector space whose homogeneous degree-$k$ elements are linear maps from $A$ to $B$ of degree $k$. When $A=B$, we also write
\[
\operatorname{End}^\bullet(A)\coloneqq \Hom^\bullet(A,A),
\]
which is a graded algebra under composition. 

We assume that $X$ satisfies the conditions in \Cref{qh}, that $\bN$ is a representation of $G$, and that there exists an equivariant proper morphism
\begin{equation*}
f \colon X \to \bN.
\end{equation*}
Our goal is to construct a homomorphism
\begin{equation*}
\bbS_{G,\bN,X} \colon \sA^\hbar_{G,\bN}\to \operatorname{End}^\bullet(H_T^\bullet(X)[\hbar])[[q,\tau]]
\end{equation*}
that endows $H^\bullet_{G \times \C^\times_\hbar}(X)[[q_G,\tau]]$ with the structure of a graded $\sA^\hbar_{G,\bN}$-module. Here we use the convention that classes in $H_k^{G_\oh\rtimes \C^\times}(\Gr_G)$ have degree $-k$.

\subsection{The tautological section and its zero locus}
Throughout this section, $\beta$ will always denote a section class.
\subsubsection*{A lemma on $\sE_p(\bN)$} 
Recall in \Cref{univbun} that we have a trivialization of $\sE$ over $\Gr_G\times \{\infty\}$. Let $p \in \Gr_G$, there is thus an induced local trivialization of the vector bundle $\sE_p(\bN)$ over $\PP^1$
\begin{equation*}
    \varphi_\infty(\bN): \sE_p(\bN)|_{\spec \C[t^{-1}]} \xrightarrow{\sim} \spec \C[t^{-1}] \times \bN
\end{equation*}
over $\spec \C[t^{-1}]$. We write $\varphi_\infty(\bN)$ simply as $\varphi_\infty$ when no confusion arises. Suppose $s$ is a section of the vector bundle $\sE_p(\bN)$, which corresponds to a morphism $s_\infty: \spec \C[t^{-1}] \to \bN$ under $\varphi_\infty$. We regard $s_\infty$ as an element of $\bN_{\C[t^{-1}]}$. Recall the bundle $\BT=G_\ok\times_{G_\oh}\bN_\oh$.
\begin{lem}\label{inTp}
    $s_\infty \in \BT_p$.
\end{lem}
\begin{proof}
    Let $\widetilde{p}$ be a lift of $p$ in $G_\ok$, and recall that $\BT_p = \widetilde{p} \,\bN_\oh\subset\bN_{\ok}$. There is a trivialization of $\sE_p(\bN)$ over $\spec \oh$:
    \begin{equation*}
        \varphi_0 : \sE_p(\bN)|_{\spec \oh} \xrightarrow{\sim} \spec \oh \times \bN
    \end{equation*}
    such that the transition function $\varphi_\infty \circ \varphi_0^{-1} : \bN_\ok \to \bN_\ok$ is given by the action of $\widetilde{p} \in G_\ok$. Let $s_0: \spec \oh \to \bN$ be the morphism induced by $s$ under $\varphi_0$, and regard it as an element of $\bN_\oh$. Then we have $s_\infty = \widetilde{p} s_0 \in \BT_p$, as claimed.
\end{proof}

\subsubsection*{The tautological section}
Let $\lambda \in \Lambda$, and let $\rho_\lambda \colon \widetilde{C}_{\leq \lambda} \to C_{\leq \lambda}$ be a resolution as in \Cref{existresol}. We will construct a section of $\widetilde{\BS}_{\leq \lambda}$ over $\cM_{\leq \lambda}(X)_n$\footnote{That is, a morphism $\cM_{\leq \lambda}(X)_n \to \widetilde{\BS}_{\leq \lambda}$ of spaces over $\widetilde{C}_{\leq \lambda}$.}.

Let $\sigma : (\Sigma, y_0, y_\infty,y_1,\dots,y_n) \to \sE_{\leq \lambda}(X)$ be a stable map representing a point in $\cM_{\leq \lambda}(X)_n$. By the definition of the section class, there exists a point $p_\sigma \in \widetilde{C}_{\leq \lambda}$ and a unique irreducible component $\Sigma_0$ of $\Sigma$ such that the composition
\begin{equation*}
    \Sigma \xrightarrow{\sigma} \sE_p(X) \to \widetilde{C}_{\leq \lambda} \times \PP^1
\end{equation*}
restricts to an isomorphism $r : \Sigma_0 \xrightarrow{\sim} p_\sigma \times \PP^1$. Composing $\sigma \circ r^{-1}$ with the projection $\cM_{\leq \lambda}(X)_n \to \cM_{\leq \lambda}(\bN)_n$ gives a section $\overline{\sigma}$ of the vector bundle $\sE_p(\bN) \coloneqq \sE_{\rho_\lambda(p)}(\bN)$. 

By \Cref{inTp}, the image $\overline{\sigma}_\infty$ lies in $\BT_{p_\sigma} \coloneqq \BT_{\rho_\lambda(p_\sigma)}$.

\begin{df}
    Let $\can_{\leq\lambda} (X)$ be the section of $\widetilde{\BS}_{\leq \lambda}$ over $\cM_{\leq \lambda}(X)_n$ given by the assignment
    \begin{equation*}
        \sigma \longmapsto [\overline{\sigma}_\infty],
    \end{equation*}
    where $\sigma \in \cM_{\leq \lambda}(X)_n$, and $[\overline{\sigma}_\infty]$ denotes the image of $\overline{\sigma}_\infty$ under the projection $\BT_{p_\sigma} \to (\widetilde{\BS}_{\leq \lambda})_{p_\sigma}$. 
    
    Moreover, if $\beta$ is a section class of $\sE_{\leq \lambda}(X)$, we denote by $\can_{\leq\lambda}(X, \beta)$ the restriction of $\can_{\leq\lambda}(X)$ to $\cM_{\leq \lambda}(X, \beta)_n$.
\end{df}
If $p \in \widetilde{C}_{\leq \lambda}$, then the maps $\can_p(X) \colon \cM_p(X)_n \to \widetilde{\BS}_{\leq \lambda}|_p$ are defined analogously.

\subsubsection*{The zero locus}
\begin{df}
    We denote by $\cZ_{\leq\lambda}(X,\beta)_n$ and $\cZ_{\leq\lambda}(X)_n$ the zero loci of $\can_{\leq\lambda}(X,\beta)$ and $\can_{\leq\lambda}(X)$ respectively. We also write $\cZ_{\leq\lambda}^\bN(X,\beta)_n$ and $\cZ_{\leq\lambda}^\bN(X)_n$ if we want to emphasize the representation $\bN$.
\end{df}

In contexts involving multiple representations, we include a superscript to indicate the representation, denoting, for example, $\cZ^{\bN}_{\leq\lambda}(X)_n$ and $\can^{\bN}_{\leq\lambda}(X)_n$.

If $p$ is a point in $\Gr_G$ or in $\widetilde{C}_{\leq \lambda}$, the substack $\cZ_p(X, \beta)_n \subset \cM_p(X, \beta)_n$ is defined similarly. We regard $\ev_\infty$ as a morphism $\cM_{\leq\lambda}(X)_n\to \widetilde C_{\leq\lambda}\times X$.

\begin{lem}\label{propermoduli}
    The restriction of the evaluation map $\ev_\infty: \cM_{\leq \lambda}(X)_n \to \widetilde C_{\leq\lambda}\times X$ to $\mathcal{Z}_{\leq \lambda}(X)_n$ is proper.
\end{lem}

\begin{proof}
Since the morphism $\cM_{\leq \lambda}(X)_{n+1} \to \cM_{\leq \lambda}(X)_n$ that forgets the last marked point (and stabilizes) is proper, it suffices to consider the case $n = 0$. For simplicity, we write $\cM_{\leq \lambda}(X)_0 = \cM_{\leq \lambda}(X)$.

We first consider the case $X = \bN$. In this case, the section class $\beta$ is unique. Let $s\in \cZ_{\leq\lambda}(\bN)$ be a point with image $\pr_{\widetilde C_{\leq\lambda}}(s)=p\in \widetilde C_{\leq\lambda}$. By definition, $s$ is a global section of the vector bundle $\sE_p(\bN)\to p\times\PP^1$, such that $\can_{\leq\lambda}(\bN)(s)=[s_\infty]=0\in (\widetilde\BS_{\leq\lambda})_p$. By \Cref{Stildesurjects}, the projection $\BT_p\to \BS_p$ factors as
\[\BT_p\to (\widetilde \BS_{\leq\lambda})_p\to \BS_p.\]
Hence, $s_\infty$ lies in $\BR_p$, which implies 
\[s_\infty \in \bN_\oh\cap \bN_{\C[t^{-1}]}=\bN.\]

Consider the diagram
\begin{equation*}
\begin{tikzcd}
    \cZ_{\leq\lambda}(\bN)\ar[r]\ar[d,"\ev_\infty"] &
    \cM_{\leq \lambda}(\bN)\ar[d] \\
    \widetilde C_{\leq \lambda}\times \bN\ar[rd] &
    \widetilde G^{\leq \lambda}_\ok\times_\BI \bN_\oh\ar[d] \\
    & \widetilde C_{\leq \lambda}\times \bN_\ok.
\end{tikzcd}
\end{equation*}
All arrows except the left vertical one are closed immersions. Hence, the evaluation morphism $\ev_\infty: \cZ_{\leq\lambda}(\bN)\to \widetilde C_{\leq\lambda}\times\bN$ is also a closed immersion.

    For the general case, we consider the commutative diagram:
    \begin{center}
    \begin{tikzcd}
    \mathcal Z_{\leq \lambda}(X,\beta) \arrow[r, "j_X", hook] \arrow[d, "\cZ(f)"] 
        & \cM_{\leq \lambda}(X,\beta) \arrow[r, "\ev_\infty"] \arrow[d, "\cM(f)"] 
        & \sE_{\leq \lambda}(X) \arrow[d, "\sE(f)"] \\
    \mathcal Z_{\leq \lambda}(\bN) \arrow[r, "j_{\bN}", hook] 
        & \cM_{\leq \lambda}(\bN) \arrow[r, "\ev_\infty"] 
        & \sE_{\leq \lambda}(\bN)
    \end{tikzcd}
    \end{center}
    It suffices to show that $\cM(f)$ is proper, since this implies the properness of $\cZ(f)$, and hence of the composition $\sE(f) \circ \ev_\infty \circ j_X = \ev_\infty \circ j_{\bN} \circ \cZ(f)$. In particular, this shows that $\ev_\infty \circ j_X$ is proper.

    Properness of $\cM(f)$ follows from the semiprojectivity of $X$. Indeed, $f: X \to \bN$ factors equivariantly as
    \begin{equation*}
        X \stackrel{h}{\longhookrightarrow} \PP^n \times \bN \longrightarrow \bN,
    \end{equation*}
    where $h$ is a closed embedding. Then there is a corresponding factorization of $\cM(f)$:
    \begin{equation*}
        \cM_{\leq \lambda}(X,\beta) \xrightarrow{\cM(h)} \cM_{\leq \lambda}(\PP^n \times \bN, h_* \beta) \longrightarrow \cM_{\leq \lambda}(\bN).
    \end{equation*}
    The natural map induces an isomorphism
    \begin{equation*}
        \cM_{\leq \lambda}(\PP^n \times \bN, h_* \beta) \cong \cM_{\leq \lambda}(\PP^n, (\pr_{\PP^n})_* h_* \beta) \times \cM_{\leq \lambda}(\bN).
    \end{equation*}
    Since $\cM(h)$ is a closed embedding (cf. \cite[Section\ 5.1]{Fulton-Pandharipande}) and $\cM_{\leq \lambda}(\PP^n, (\pr_{\PP^n})_* h_* \beta)$ is proper, the map $\cM(f)$ is also proper, as claimed.
\end{proof}

\subsection{Section-counting map}
\subsubsection*{Virtual fundamental classes}

Consider the fibre diagram
\begin{equation}\label{PBZ}
\begin{tikzcd}
\cZ_{\leq \lambda}(X, \beta)_n \arrow[r] \arrow[d] & \cM_{\leq \lambda}(X, \beta)_n \arrow[d, "\can_{\leq \lambda}(X{,}\beta)"] \\
\widetilde{C}_{\leq \lambda} \arrow[r,"\mathfrak{z}"] & \widetilde{\BS}_{\leq \lambda}
\end{tikzcd} ,
\end{equation}
where $\mathfrak{z}$ is the inclusion of the zero section. We define the virtual fundamental class
\begin{equation}\label{virZ}
[\cZ_{\leq \lambda}(X, \beta)_n]^\vir \coloneqq \mathfrak{z}^!\big[\cM_{\leq \lambda}(X, \beta)_n\big]^\vir,
\end{equation}
where $\mathfrak{z}^!$ denotes the Gysin pullback
\begin{equation*}
\mathfrak{z}^! \colon H^{T\times\Cx_\hbar}_\bullet\big(\cM_{\leq \lambda}(X, \beta)_n\big) 
\longrightarrow H^{T\times\Cx_\hbar}_{\bullet - 2d_\lambda}\big(\cZ_{\leq \lambda}(X, \beta)_n\big),
\end{equation*}
taken relative to the diagram \eqref{PBZ}. Here, $d_\lambda$ denotes the rank of the vector bundle $\widetilde{\BS}_{\leq \lambda}$ (see \eqref{SlambdaRank}). We record the following.
\begin{equation}\label{vdim2}
\mathrm{vdim}\big(\cZ_{\leq \lambda}(X, \beta)_n\big)= \dim X + \dim C_\lambda + c_1^G(X) \cdot \overline\beta-d_\lambda+n.
\end{equation}

We write $\ev_\infty$ in place of $\ev_\infty \circ \mathfrak{z}$ for simplicity. 

\begin{df}\label{Slambda}
We define the \emph{section-counting map}
\begin{equation*}
\widetilde\bbS_{\bN,\leq \lambda} \colon e(\widetilde{\BS}_{\leq\lambda}) \cup H_{T \times \C^\times_\hbar}^\bullet(\mathcal{X}_{\leq \lambda, 0}) \longrightarrow H^\bullet_{T \times \C^\times_\hbar}(X)[[q_G,\tau]]
\end{equation*}
by setting
\begin{equation}\label{sectioncountingmap}
\widetilde\bbS_{\bN,\leq \lambda}(e(\widetilde{\BS}_{\leq\lambda}) \cup \gamma) \vcentcolon= 
\sum_{\beta \in \Eff(\sE_{\leq \lambda}(X))^{\mathrm{sec}}} 
\sum_{n=0}^\infty \frac{q^{\overline{\beta}}}{n!} \,
\PD\circ \pr_{X*} \ev_{\infty*} \left(
    \ev_0^*(\gamma) \prod_{\ell=1}^n \ev_\ell^*(\hat\tau) \cap 
    [\cZ_{\leq \lambda}(X, \beta)_n]^\vir
\right)
\end{equation}
for $\gamma \in H_{T \times \Cxh}^\bullet(\cX_{\leq \lambda, 0})$. Here, $\pr_X \colon \cX_{\leq \lambda, \infty} \cong \widetilde C_{\leq \lambda} \times X \to X$ is the projection map, and $\mathrm{PD} \colon H_\bullet^{T \times \Cx_\hbar}(X) \to H_{T \times \Cxh}^{2 \dim X - \bullet}(X)$ is the Poincar\'e duality map. 
\end{df}

The above definition makes sense due to the following two lemmas.

\begin{lem}\label{lemmacurveclass2}
There exists a finite subset $S \subset H_2^{\mathrm{ord}, G}(X;\Z)$ such that
\begin{equation*}
\left\{\overline{\beta} \in H_2^{\mathrm{ord}, G}(X;\Z) \,\middle|\, \beta \in \Eff(\sE_{\leq \lambda}(X))^{\mathrm{sec}} \right\} \subset S + i_*\left(\Eff(X)\right),
\end{equation*}
where $i \colon X \hookrightarrow \sE_{\leq \lambda}(X)$ is the inclusion of a fibre.
\end{lem}

\begin{lem}\label{graded}
Let $\gamma\in H_{T\times\Cx_\hbar}^\bullet(\cX_{\leq\lambda,0})$ be a homogeneous element. Then, each summand in the power series expansion of the right-hand side of \eqref{sectioncountingmap} is of degree $\deg e((\widetilde{\BS}_{\leq\lambda}) \cup \gamma)-2\dim C_{\lambda}=2d_\lambda-2\dim C_{\lambda}+\deg\gamma$, which is independent of $\beta$ and $n$.
\end{lem}

\begin{df}\label{bbSN}
\emph{The shift operator with matter} $\bN$ is the $H^{T}_\bullet(\pt)[\hbar]$-linear map
\begin{equation*}
\bbS_{G,\bN,X} \colon e(\BS)\cap H^{\BI\rtimes\Cx_\hbar}_\bullet(\Gr_G)\to \operatorname{End}^\bullet(H_T^\bullet(X)[\hbar])[[q,\tau]]
\end{equation*}
uniquely determined by
\begin{equation}\label{bbSNformula}
\bbS(e(\BS) \cap [C_{\leq \lambda}])(\alpha)=\widetilde\bbS_{\bN,\leq \lambda} \left(e(\widetilde{\BS}_{\leq\lambda})\cup \tw_{\leq \lambda}\left([\widetilde{C}_{\leq \lambda}] \otimes \alpha\right)\right).
\end{equation}
The map $\bbS_{G,\bN,X}$ is then extended $H_{T \times \C^\times_\hbar}^\bullet(\pt)$-linearly in the first argument. 
\end{df}
We will show in the next section that $\bbS_{G,\bN,X}$ is
$W$-equivariant. We denote the induced homomorphism $\sA_{G,\bN}\to \operatorname{End}^\bullet(H_G^\bullet(X)[h])[[q,\tau]]$ also by $\bbS_{G,\bN,X}$. 
\begin{rem}
    By \Cref{twistingeuler}, the expression \eqref{bbSNformula} can also be written as the composition $\bbS_{\bN,\leq\lambda}\circ \tw_G(e(\widetilde{\BS}_{\leq\lambda})\cap [\widetilde{C}_{\leq\lambda}]\otimes\alpha)$.
\end{rem}
We will simply write $\bbS_{G,\bN}$ when the space $X$ is clear from the context, and we write $\bbS_{G}$ if the representation $\bN=\bf 0$. We will show in \Cref{Wequiv} that $\bbS_{G,\bN,X}$ is $W$-equivariant and in particular its restriction to the subspace of $W$-invariants gives the map \eqref{Introshiftoperators3}.

\begin{rem}\label{CompareIritaniRemark}
    In the case where $X$ is $T$-compact, Iritani \cite{Iritani} defined shift operators for the $T$-action on $X$. Note that for semiprojective $X$, there always exists a $T$-representaion $\bN$ with an equivariant proper map $f:X\to \bN$. As we will see in \Cref{subsection:moduleproperty} (see also \Cref{IritaniRemark}), the shift operators $\bbS_{T,\bN}$ defined in this paper recovers Iritani's definition, up to adjoint and Novikov variables.
\end{rem}
\begin{rem}\label{remSemi-ve}
Suppose $G = T$ is abelian, and let $\lambda \in \Lambda$ be semi-negative with respect to $X$, in the sense that $\eta(\lambda) \leq 0$ for every weight $\eta$ of $H^0(X, \oh_X)$ (cf.~\cite[Definition~3.3]{Iritani}).

It is easy to show that there exists a representation $\bN$ with $\lambda$-non-positive weights and a proper $T$-equivariant morphism $f \colon X \to \bN$. By \Cref{Nweights}, we have $\BS_{t^\lambda} = 0$, and hence $\cZ_{\leq \lambda}(X)_n = \cM_{\leq \lambda}(X)_n$. In this case, the properness statement in \Cref{propermoduli} follows from~\cite[Lemma~3.5]{Iritani}.

Remark~3.10 of~\cite{Iritani} also notes that the semi-negativity condition implies that the shift operator is defined without localization. One may regard \Cref{propermoduli} and \Cref{bbSN} as generalizations of this observation. In \Cref{propertyshiftoperators} below, we will explain in detail how our construction of shift operators compares with that in~\cite{Iritani}.
\end{rem}

\begin{proof}[Proof of \Cref{lemmacurveclass2}]
We first reduce to the case where $X$ is compact. Indeed, if $C$ is a curve on $\sE_{\leq \lambda}(X)$ representing a section class, we may use the $\C^\times_{\mathrm{attr}}$-action to move $\sigma$ so that it lies in $\sE_{\leq \lambda}(X^{\Cx_{\mathrm{attr}}})$. Hence, we may replace $X$ by $X^{\Cx_{\mathrm{attr}}}$, which is compact.

Next, we use the $T$-action to move the curve $C$ to a $T$-invariant curve. In particular, its projection $p = \pr_{\widetilde{C}_{\leq \lambda}}(C)$ must be a $T$-fixed point in $\widetilde{C}_{\leq \lambda}$. We have
\begin{equation*}
\sE_p(X) = \sE_{\rho_\lambda(p)}^T(X),
\end{equation*}
where $\sE^T$ denotes the universal $T$-torsor over $\Gr_T \times \PP^1$. Therefore, we may reduce to the case where $G = T$ is abelian.

Now let $C \subset \sE^T_p(X)$ be a $T$-invariant curve representing a section class with irreducible components $C_0,\dots, C_m$. Then one of them, say $C_0$, must be the constant section
\begin{equation*}
\mathrm{Const}_x = \sE_p^T \times_T x
\end{equation*}
for some $x \in X^T$, and all other components lie in the fibres of $\sE_p^T(X) \to p \times \PP^1$. Let $\{F_k\}_{k\in I}$ be the connected components of $X^T$, and let $\beta_k \in \Eff(\sE_{\leq \lambda}(X))^{\sec}$ be the class represented by $\mathrm{Const}_x$ for some (any) $x \in F_k$. Then the above reasoning gives
\begin{equation*}
[C] \in \sum_{k\in I} \Z_{\geq 0} \beta_k + \Eff(X).
\end{equation*}
Therefore, we may conclude the lemma by taking $S = \{\overline{\beta_k}\}_{k\in I}$.
\end{proof}

\begin{proof}[Proof of \Cref{graded}]
Using \eqref{vdim2} and \Cref{degree}, we see that for any $i_\ell\in I$ and $j_\ell \in \{0, 1, \dots, N\}$, the degree of
\begin{equation*}
\frac{q^{\overline{\beta}}}{n!} \,
\mathrm{PD} \circ \pr_{X*} \ev_{\infty*}
    \left( \ev_0^* (\gamma) \prod_{\ell=1}^n \ev_\ell^* (\tau^{i_\ell,j_\ell} \hat\phi_{i_\ell,j_\ell}) \cap 
    [\cZ_{\leq \lambda}(X, \beta)_n]^\vir \right)
\end{equation*}
is equal to
\begin{align*}
& \quad 2c_1^G(X) \cdot \overline{\beta} + 2 \dim X 
-  2 \mathrm{vdim}(\cZ_{\leq\lambda}(X, \beta)_n) + \deg \gamma 
 + \sum_{\ell=1}^n \left( \deg \tau^{i_\ell,j_\ell} + \deg \phi_{i_\ell,j_\ell} \right) \\
&= 2d_\lambda - 2 \dim C_{\lambda} + \deg \gamma. \qedhere
\end{align*}
\end{proof}

We adopt the cohomological convention for degrees in Borel--Moore homology. In particular, an element of $H_k^{T \times \Cxh}(\Gr_G)$ has degree $-k$.

Since $\tw_G$ increases the degree by $2\dim C_\lambda$ (due to Poincar\'e duality), we obtain the following corollary of \Cref{graded}.

\begin{prop}
Let \( P \in H_{G \times \Cxh}^\bullet(\pt) \). Then
\begin{equation*}
\bbS_{G,\bN}(\Gamma)( (P \cup \alpha)) = \bbS_{G,\bN}(u^* P \cap \Gamma )(\alpha), 
\end{equation*}
where $u^*:H_G^\bullet(\pt)\to  H_{\BI\rtimes\Cx_\hbar}^\bullet(\Gr_G)$ is defined in \Cref{secondmodule}.
\end{prop}
\begin{proof}
    It follows from \Cref{twistinglinear} and the projection formula.
\end{proof}

\begin{df}
We write
\begin{equation*}
\bbS_{G,\bN}^{\hbar = 0} \colon \sA_{G,\bN} \longrightarrow \operatorname{End}(H_T^\bullet(X))[[q,\tau]].
\end{equation*}
for the reduction modulo $\hbar$ of $\bbS_{G,\bN}$. The \textit{Seidel homomorphism with matter $\bN$} is the map
\begin{equation*}
    \Psi_{G,\bN,X} \colon  \sA_{G,\bN}\longrightarrow QH_G^\bullet(X),\qquad \Psi_{G,\bN,X}(\Gamma)=\bbS^{\hbar=0}_{G,\bN}(\Gamma)(1).
\end{equation*}
\end{df}

\subsection{Independence of resolutions}
In this subsection, we prove that $\bbS_{G,\bN}$ is independent of the choices of the resolutions $\widetilde{C}_{\leq\lambda}$. This result also follows from \Cref{TGcompatible}.

\begin{prop}
    The operator $\bbS_{G,\bN}$ is independent of the choices of resolutions $\widetilde{C}_{\leq\lambda}$.
\end{prop}
\begin{proof}
    Suppose $\widetilde{C}_{\leq\lambda}'\to C_{\leq\lambda}$ is another resolution of $C_{\leq\lambda}$ on which $\BS_{\lambda}$ extends as in \Cref{existresol}. Replacing $\widetilde{C}_{\leq\lambda}'$ by a resolution of the fibre product $\widetilde{C}_{\leq\lambda}'\times_{C_{\leq\lambda}}\widetilde{C}_{\leq\lambda}$ if necessary, we may assume $\widetilde{C}_{\leq\lambda}'\to C_{\leq\lambda}$ factors as 
    \[\widetilde{C}_{\leq\lambda}'\stackrel{r}{\longrightarrow}\widetilde{C}_{\leq\lambda}\longrightarrow C_{\leq\lambda}.\] 
    Let us denote
    \[\widetilde\alpha_{\leq\lambda}=\tw_G\left([\widetilde C_{\leq\lambda}]\otimes\alpha \right),\quad \widetilde\alpha_{\leq\lambda}'=\tw_G'\left([\widetilde C_{\leq\lambda}']\otimes\alpha\right).\]
    Let $\sE_{\leq\lambda}(X)$ and $\sE_{\leq\lambda}'(X)$ denote the Seidel spaces corresponding to $\widetilde C_{\leq\lambda}$ and $\widetilde C_{\leq\lambda}'$, respectively. We write $\hat\tau \in H^\bullet_{T\times\Cxh}(\sE_{\leq\lambda}(X))$ and $\hat\tau'\in H_{T\times\Cxh}^\bullet(\sE_{\leq\lambda}'(X))$ for the assignments \eqref{tauhatmap} associated to $\sE_{\leq\lambda}(X)$ and $\sE'_{\leq\lambda}(X)$, respectively. By \Cref{twistingfunctorial} and the projection formula, these classes satisfy
    \begin{equation}\label{alphapushforward}\widetilde\alpha_{\leq\lambda}=r_*\widetilde\alpha_{\leq\lambda}',\quad \hat\tau = r_*\hat\tau'.\end{equation}
    Here, we use the same symbol $r_*$ to denote the pushforward maps $H_{T_\oh\rtimes\Cx_{\hbar}}^\bullet(\widetilde{G}_\ok^{\prime\leq\lambda}\times_{G_\oh} X)\to H_{T\times\Cxh}^\bullet(\widetilde{G}_\ok^{\leq\lambda}\times_{G_\oh}X)$ and $H_{T\times\Cxh}^\bullet(\sE'_{\leq\lambda}(X))\to H^\bullet_{T\times\Cxh}(\sE_{\leq\lambda}(X))$, whenever there is no confusion.

    Write $\widetilde\bbS$ and $\widetilde\bbS'$ for the respective section-counting maps \eqref{Slambda}. It remains to verify that
    \[\widetilde\bbS(e(\widetilde{\BS}_{\leq\lambda})\cup\widetilde\alpha_{\leq\lambda}')=\widetilde\bbS'(e(\widetilde{\BS}_{\leq\lambda})\cup \widetilde\alpha_{\leq\lambda}).\]

    Note that the moduli spaces are related by the fibre diagram
    \begin{equation*}
        \begin{tikzcd}
{\overline M_{0,n+2}(\sE_{\leq\lambda}'(X),\beta)} \arrow[d, "\pr_{\widetilde C_{\leq\lambda}'}"'] \arrow[r] & {\overline M_{0,n+2}(\sE_{\leq\lambda}(X,r_*\beta)} \arrow[d, "\pr _{\widetilde C_{\leq\lambda}}"] \\
\widetilde C_{\leq\lambda}' \arrow[r, "r"]                                                                 & \widetilde C_{\leq\lambda} .
\end{tikzcd}
    \end{equation*}
    Therefore, by functoriality (see \cite[Proposition 5.10]{Intrinsic}), the virtual fundamental classes are related by Gysin pullback
    \[[\,\overline M_{0,n+2}(\sE'_{\leq\lambda}(X),\beta)]^\vir= r^![\,\overline M_{0,n+2}(\sE_{\leq\lambda}(X),r_*\beta)]^\vir,\]
    and hence the same is true for the zero loci (see \cite[Section 6.4]{Fulton})
    \[[\cZ'_{\leq\lambda}(X,\beta)_n]^{\vir}=r^![\cZ_{\leq\lambda}(X,r_*\beta)_n]^{\vir}.\]
    Denote by $\ev_0',\ev_\infty',\pr_X'$ the evaluation and projection maps for the spaces associated with $\widetilde{C}_{\leq\lambda}'$. Using the above, we then compute
    \begin{align*}
        \widetilde\bbS'(e(\widetilde{\BS}_{\leq\lambda})\cup\widetilde\alpha_{\leq\lambda}') 
        &= \sum_{\beta}\sum_{n=1}^\infty \frac{ q^{\overline\beta}}{n!} \, \mathrm{PD} \circ \pr'_{X*}\ev'_{\infty*} \left(\ev^{\prime*}_0(\widetilde\alpha'_{\leq\lambda})\prod_{\ell=1}^n \ev_\ell^{\prime*}(\hat{\tau}')\cap [\cZ'_{\leq\lambda}(X,\beta)_n]^{\vir}\right)\\
        &= \sum_\beta \sum_{n=1}^\infty\frac{ q^{\overline\beta}}{n!} \, \PD\circ \pr_{X*}\ev_{\infty*}\, r_*\left(\ev_0^{\prime*}(\widetilde{\alpha}'_{\leq\lambda})\prod_{\ell=1}^n \ev_\ell^{\prime*}(\hat{\tau}') \cap  r^![\cZ_{\leq\lambda}(X,r_*\beta)_n]^\vir \right) \\
        &= \sum_{\beta} \sum_{n=1}^\infty\frac{ q^{\overline\beta}}{n!}  \, \mathrm{PD} \circ \pr_{X*}\ev_{\infty*} \left(\ev^{*}_0( \widetilde\alpha_{\leq\lambda})\prod_{\ell=1}^n \ev_\ell^{*}(\hat{\tau})\cap [\cZ_{\leq\lambda}(X,\beta)_n]^{\vir}\right)\\
        &= \widetilde\bbS(e(\widetilde{\BS}_{\leq\lambda})\cup \widetilde\alpha_{\leq\lambda}).
    \end{align*}
Here, we have used the identification $\overline{r_*\beta}=\overline\beta$ and the last equality is by projection formula and \eqref{alphapushforward}.
\end{proof}

\subsection{Flavour symmetries and deformations}\label{section:flavoursymmetries}
Suppose we have a short exact sequence
\begin{equation*}
    \xi:\quad 1 \longrightarrow G \longrightarrow \hG \longrightarrow T_\xi \longrightarrow 1
\end{equation*}
where $T_\xi$ is a torus, and the $G$-action on $\bN$ extends to a $\hG$-action. We call $T_\xi$ a group of \emph{flavour symmetries} of the gauge theory $(G,\bN)$. We denote by $\hT$ the unique maximal torus of $\hG$ containing $T$.

Note that the stratified vector bundle $\BS_{G,\bN}$ on $\Gr_G$ is 
$\hG$-equivariant. 
Indeed, $\BS_{G,\bN}$ can be identified with the restriction of 
$\BS_{\hG,\bN}$ to 
$\Gr_{G}\subset \Gr_{\hG}$. 
We define
\begin{align*}
    \sA_{G,\bN}^\xi 
    &= e(\BS_{G,\bN})\cap H^{\hG_\oh}_\bullet(\Gr_G),\\
    \sA_{G,\bN}^{\xi,\hbar} 
    &= e(\BS_{G,\bN})\cap H^{\hG\times\Cxh}_\bullet(\Gr_G).
\end{align*}
We call $\sA_{G,\bN}^\xi$ and $\sA_{G,\bN}^{\xi,\hbar}$ 
the \emph{deformations of the Coulomb branch algebra} 
with respect to the flavour symmetry $\xi$.

If $f:X\to\bN$ is $\hG$-equivariant, then all moduli spaces involved in defining 
$\Psi$ and $\bbS$ admit $\hG$-actions, such that all relevant maps are 
$\hG$-equivariant. Hence we can extend the definitions of $\bbS_{G,\bN,X}$ and $\Psi_{G,\bN,X}$, and define \emph{shift operator under the flavour symmetry $\xi$}:
\begin{equation*}
\bbS_{G, \bN,X}^\xi \colon \sA_{G,\bN}^{\xi,\hbar} \longrightarrow \Hom(H_{\hT}^\bullet(X)[\hbar])[[q,\tau]],
\end{equation*}
and the \emph{Seidel homomorphism under the flavour symmetry $\xi$}:
\begin{equation*}
\Psi_{G,\bN,X}^\xi \colon \sA_{G,\bN}^\xi\longrightarrow QH_{\hG}^\bullet(X).
\end{equation*}
Recall that we use $G$-equivariant Novikov and bulk variables (rather than $\hG$-equivariant ones).
\begin{rem}
The connected components of $\Gr_{\hG}$ is labelled by $\pi_1(\hG)$. Then homomorphism $\pi_1(\hG)\to \pi_1(T_\xi)$ gives a $\pi_1(T_\xi)$-grading, or equivalently a $\widecheck{T}_\xi$-action on $\sA_{\hG,\bN}$. 
It is known that $\sA_{G,\bN}^\xi$ is naturally isomorphic to 
$(\sA_{\hG,\bN})^{\widecheck{T}_\xi}$. In other words, there is a commutative diagram
\begin{equation*}
    \begin{tikzcd}
        \sA_{G,\bN}^\xi \ar[r,"\Psi_{G,\bN,X}^\xi"] \ar[rd,"\Psi_{\hG,\bN,X}",swap] 
        & QH^\bullet_{\hG}(X) \ar[d] \\
        & QH^\bullet_{\hG}(X).
    \end{tikzcd}
\end{equation*}
Indeed, one can show that $\Psi_{G,\bN,X}$ is obtained from 
$\Psi_{\hG,\bN,X}$ by \emph{quantum Hamiltonian reduction} 
in the sense of \cite[§3(vii)(d)]{BFN}. Similar remarks apply to the shift operators on quantized Coulomb branches.
\end{rem}

\section{Shift operators III: properties}\label{propertyshiftoperators}
In this section we establish several fundamental properties of the shift operators. 
In particular, we give the proofs of \Cref{IntroThm1} and \Cref{IntroThm2} stated in the introduction. 
Throughout we assume that the pair $(X, f \colon X \to \bN)$ satisfies the conditions described in \Cref{sectionshiftoperators}. 
All of the propositions remain valid for shift operators with respect to a flavour symmetry $\xi$ provided that $f$ is $\hG$-equivariant; however, since the arguments are parallel, we present proofs only in the case without flavour symmetries.

\subsection{Change of representations}\label{section:changeofrep}

In this subsection, we study the compatibility of shift operators under changes of the representations. 

\begin{prop}[Change of representations]\label{shiftopfunctoriality}
    Suppose $\bV$ is another $G$-representation, and assume one of the following:
    \begin{enumerate}[label=(\alph*)]
        \item $\bV$ contains $\bN$ as a subrepresentation; or
        \item $\bN$ is a quotient representation of $\bV$, and the morphism $f:X\to \bN$ factors through $\bV \to \bN$.
    \end{enumerate}
    Then $\bbS_{G,\bV}$ coincides with the restriction of $\bbS_{G,\bN}$ along
    \begin{equation*}
        e(\BS_{\bV})\cap H_\bullet^{T\times\Cxh}(\Gr_G)
        \;\subset\;
        e(\BS_{\bN})\cap H_\bullet^{T\times\Cxh}(\Gr_G).
    \end{equation*}
    A similar statement holds in the presence of a flavour symmetry $\xi$.
\end{prop}

\begin{proof}
We prove case (a); the other case proceeds similarly.

Write $\bV = \bN \oplus \bN'$, and choose a resolution $\widetilde{C}_{\leq \lambda}$ such that both $\BS_\bN$ and $\BS_{\bN'}$ extend over it. Since the composition $X \to \bV \to \bN'$ is zero, we have $\can_{\leq \lambda}^{\bN'}(X) = 0$. Moreover, since $\can_{\leq \lambda}^{\bV} (X)= (\can_{\leq \lambda}^{\bN}(X), \can_{\leq \lambda}^{\bN'}(X))$, it follows that
\begin{equation*}
\cZ_{\leq \lambda}^{\bV}(X,\beta)_n = \cZ_{\leq \lambda}^{\bN}(X,\beta)_n
\end{equation*}
for any $\beta \in \Eff(\sE_{\leq \lambda}(X))^{\mathrm{sec}}$.

We also have the identity
\[
[\cZ_{\leq \lambda}^{\bV}(X,\beta)_n]^{\mathrm{vir}} = e(\widetilde{\BS}_{\bN',\leq \lambda}) \cap [\cZ_{\leq \lambda}^{\bN}(X,\beta)_n]^{\mathrm{vir}},
\]
which follows from \eqref{virZ} and excess intersection formula associated to the diagram:
\begin{equation*}
\begin{tikzcd}
\cZ_{\leq\lambda}(X,\beta)_n \arrow[r] \arrow[d] & 
\cM_{\leq \lambda}(X, \beta)_n \arrow[d, "\can^\bN_{\leq \lambda}(X)"] \\
\widetilde{C}_{\leq \lambda} \arrow[r] \arrow[d, equal] & 
\widetilde{\BS}_{\bN, \leq \lambda} \arrow[d] \\
\widetilde{C}_{\leq \lambda} \arrow[r] & 
\widetilde{\BS}_{\bV, \leq \lambda} = \widetilde{\BS}_{\bN, \leq \lambda} \oplus \widetilde{\BS}_{\bN', \leq \lambda}
\end{tikzcd}
\end{equation*}

Let $\ev_0'$ and $\ev_\infty'$ denote the evaluation maps on $\cZ^{\bV}_{\leq \lambda}(X, \beta)_n$, and writing $\widetilde\alpha_{\leq\lambda}=\tw_G([\widetilde C_{\leq\lambda}]\otimes \alpha)$. Then we verify
\begin{align*}
&\bbS_{G, \bN}(e(\BS_{\bV}) \cap [C_{\leq \lambda}])(\alpha)\\
&= \widetilde{\bbS}_{\bN, \leq \lambda} \left( e(\widetilde{\BS}_{\bN, \leq \lambda}) \cup e(\widetilde{\BS}_{\bN', \leq \lambda}) \cup \widetilde{\alpha}_{\leq \lambda} \right) \\
&= \sum_\beta\sum_{n=0}^\infty\frac{q^{\overline{\beta}}}{n!}\, \PD \circ \pr_{X*} \ev_{\infty*} \left( 
\ev_0^*\left( \widetilde{\alpha}_{\leq \lambda} \cup e(\widetilde{\BS}_{\bN', \leq \lambda}) \right) 
\prod_{\ell=1}^n\ev_\ell^{*}(\hat{\tau})\cap [\cZ_{\leq \lambda}^{\bN}(X, \beta)_n]^{\mathrm{vir}} \right) \\
&= \sum_\beta \sum_{n=0}^\infty\frac{q^{\overline{\beta}}}{n!}\, \PD \circ \pr_{X*} \ev'_{\infty*} \left( 
\ev_0^{\prime *}(\widetilde{\alpha}_{\leq \lambda}) 
\prod_{\ell=1}^n\ev_\ell^{\prime*}(\hat{\tau})\cap [\cZ_{\leq \lambda}^{\bV}(X, \beta)_n]^{\mathrm{vir}} \right) \\
&= \bbS_{G, \bV}(e(\BS_{\bV}) \cap [C_{\leq \lambda}])(\alpha). \qedhere
\end{align*}
\end{proof}
\begin{cor}\label{CommonDomain}
Suppose $\bN'$ is another representation of $G$, and $g \colon X \to \bN'$ is a $G$-equivariant proper morphism. Then
\[
\bbS_{G,\bN}(\Gamma) = \bbS_{G,\bN'}(\Gamma)
\]
for any $\Gamma$ in the intersection of $(e(\BS_\bN) \cap H^{T_\oh \rtimes \Cx_\hbar}_\bullet(\Gr_G))$ and $(e(\BS_{\bN'}) \cap H^{T_\oh \rtimes \Cx_\hbar}_\bullet(\Gr_G))$.

A similar statement holds in the presence of a flavour symmetry $\xi$.
\end{cor}

\begin{proof}
Let $\bV = \bN \oplus \bN'$, and let $P \in H_{T\times\Cxh}^\bullet(\pt)$ be such that $P \Gamma \in e(\BS_{\bV}) \cap H^{T_\oh \rtimes \Cx_\hbar}_\bullet(\Gr_G)$. By \Cref{shiftopfunctoriality}, we have
\[
\bbS_{G,\bN}(P \Gamma)(\alpha) = \bbS_{G,\bV}(P \Gamma)(\alpha) = \bbS_{G,\bN'}(P \Gamma )(\alpha).
\]
The result then follows from the lemma below by taking $R = H_{T\times\Cxh}^\bullet(\pt)$.
\end{proof}

\begin{lem}
Let $R$ be an integral domain, and let $\zeta, \zeta' \colon A \to B$ be homomorphisms of torsion-free $R$-modules. Suppose that $\zeta$ and $\zeta'$ agree on a submodule $A' \subset A$ such that 
\[
\Frac(R) \otimes_R A = \Frac(R) \otimes_R A'.
\]
Then $\zeta = \zeta'$.
\end{lem}

\begin{proof}
Let $a \in A$. Since $\Frac(R) \otimes_R A = \Frac(R) \otimes_R A'$, there exists $r \in R \setminus \{0\}$ such that $r a \in A'$. Then
\[
r \zeta(a) = \zeta(r a) = \zeta'(r a) = r \zeta'(a),
\]
and since $B$ is torsion-free over $R$, it follows that $\zeta(a) = \zeta'(a)$.
\end{proof}

Recall that the localization of 
$e(\BS_\bN)\cap H_\bullet^{T\times\Cxh}(\Gr_G)$ 
by equivariant parameters is independent of the representation $\bN$ 
(see \eqref{localizeindepofN}). 
Thus we may localize $\bbS_{G,\bN,X}$ and $\Psi_{G,\bN,X}$ and obtain
\begin{equation*}
\begin{aligned}
    \bbS_{G,X,\loc}&\colon 
    \bigl(e(\BS_\bN)\cap H^{T\times\Cxh}_\bullet(\Gr_G)\bigr)_\loc
    \to \operatorname{End}^\bullet(H_{T}^\bullet(X)[\hbar])_\loc[[q,\tau]],\\
    \Psi_{G,X,\loc} &\colon  
    \bigl(e(\BS_\bN)\cap H_\bullet^{T\times\Cxh}(\Gr_G)\bigl)_\loc
    \longrightarrow QH_T^\bullet(X)_\loc[[q,\tau]].
\end{aligned}
\end{equation*}
Similarly, one can define $\bbS_{G,X,\loc}^\xi$ in the presence of a flavour symmetry $\xi$. The following result is an immediate consequence of \Cref{CommonDomain}. We sometimes omit the subscript $\loc$ when no confusion is likely to arise.
\begin{cor}[Independence of $\bN$ and $f$]\label{independence of f}
    The localized shift operator $\bbS_{G,X,\loc}$ (resp.  $\bbS^\xi_{G,X,\loc}$) and localized Seidel homomorphism $\Psi_{G,X,\loc}$ (resp.  $\Psi^\xi_{G,X,\loc}$) are independent of the choice of 
    $f\colon X\to\bN$. 
\end{cor}

\subsection{Abelianization of the shift operators}
For $\lambda \in \Lambda$ and $\gamma \in H_{T \times \Cxh}^\bullet(X^\lambda)_\loc$ (see \Cref{abeliantwlinear}), we can choose $P \in H^\bullet_{T \times \Cxh}(\pt)$ so that $P\gamma \in H_\bullet^{T \times \Cxh}(X^\lambda)$. We set
\begin{equation*}
   \widetilde\bbS_{\lambda}(\gamma) = \frac{1}{e(\BS_{t^\lambda})\, P} \sum_{\beta \in \Eff(\sE_{t^{\lambda}}(X))^{\sec}} \sum_{n=0}^\infty \frac{q^{\overline\beta}}{n!} \PD \circ \ev_{\infty *} \left( \ev_0^*(P\gamma) \prod_{\ell=1}^n \ev_\ell^*(\hat\tau) \cap  [\cZ_{t^\lambda}(X,\beta)_n]^\vir \right),
\end{equation*}
where $[\cZ_{t^\lambda}(X,\beta)_n]^\vir$ is defined similar to $[\cZ_{\leq \lambda}(X,\beta)_n]^\vir$. It is easy to check that $\widetilde\bbS_{\lambda}$ is independent of the choice of $P$, and induces the maps
\begin{align}
    \widetilde\bbS_{\lambda} &: 
    H_{T\times\Cxh}^\bullet(X^\lambda)_\loc 
    \longrightarrow H_{T\times\Cxh}^\bullet(X)_\loc[[q_G,\tau_G]], \nonumber \\
    \bbS_{\lambda} \coloneq \widetilde\bbS_{\lambda} \circ \Phi_\lambda &: 
    H_{T\times\Cxh}^\bullet(X)_\loc 
    \longrightarrow H_{T\times\Cxh}^\bullet(X)_\loc[[q_G,\tau_G]]
    \label{definitionbbSlambda}
\end{align}
where $\Phi_\lambda$ is defined in \Cref{abeliantwlinear}. Note that we use the $G$-equivariant Novikov and bulk parameters instead of the $T$-equivariant ones. One can define $\bbS^\xi_\lambda$ similarly in the presence of a flavour symmetry $\xi$. 

We also write $\bbS_{G,X,\lambda}$ if we want to specify the dependence on $G$ and $X$. The definition of $\bbS_\lambda$ depends on the choice of representation $\bN$. However, it follows from \Cref{intertwiningFundamentalsolution} in the next subsection that this dependence is superfluous.
\begin{rem}\label{IritaniRemark}
In the case where $X^T$ is compact. Consider the pairing $(-,-)_\lambda$ on $H_{T \times \Cxh}^\bullet(X)$ given by $
(\alpha, \alpha')_\lambda \coloneqq \Phi_\lambda(\alpha, \alpha')$,
where $(-,-)$ is the Poincar\'e pairing. Let $\bbS_\lambda^{\Iri}$ denote the ``adjoint'' of $\bbS_{\lambda}$ in the sense that $(\bbS_\lambda^{\Iri}(\alpha), \alpha') = (\alpha, \bbS_{\lambda} (\alpha'))_\lambda.$

By localization, the operator $\bbS_\lambda^{\Iri}$ agrees with the shift operators defined in~\cite[Definition 3.9]{Iritani}, provided that in \emph{loc.~cit.} one replaces all $\hat{d} - \sigma_{\min}$ with $\overline{d}$. The appearance of the adjoint is due to the difference in the conventions for the zero and infinity fibres in our definition of Seidel spaces versus that in Iritani’s (cf.~\Cref{CompareIritani2}).
\end{rem}

The main result of this subsection is the following.
\begin{thm}\label{TGcompatible}
Let $\Gamma \in \sA_{G,\bN}^\hbar$, suppose that $c_{\mu}\in H^\bullet_{T\times\Cxh}(\pt)$ such that 
$\Gamma=\sum c_\mu[t^\mu]\in H^{T\times\Cxh}_\bullet(\Gr_G)$. Then, for any 
$\alpha \in H_{G\times\Cxh}^\bullet(X)[[q_G,\tau_G]]$, we have
\begin{equation*}\label{TGcompatibleeq}
\bbS_{G,\bN}\big(\Gamma\big) (\alpha)
= \sum_{\mu \leq \lambda} c_{\mu} \, \bbS_{\mu}(\alpha).
\end{equation*}
A similar statement holds in the presence of a flavour symmetry $\xi$.
\end{thm}

Since both sides of \eqref{TGcompatibleeq} are linear over equivariant parameters, we only need to prove the case when $\Gamma = e(\BS) \cap [C_{\leq \lambda}]$ for some $\lambda \in \Lambda$. We will fix $\lambda \in \Lambda$ until the end of this subsection.

We need some preparation before the proof of \Cref{TGcompatible}. For any $\mu \leq \lambda$, let $F_\mu$ denote the $T \times \Cx_\hbar$-fixed locus of 
$\rho_\lambda^{-1}(t^\mu) \subset \widetilde{C}_{\leq \lambda}$. 
Let $e(N_\mu) \in H_{T \times \Cx_\hbar}^\bullet(F_\mu)$ be the Euler class of the normal bundle $N_\mu$ of $F_\mu$ in $\widetilde{C}_{\leq \lambda}$\footnote{We allow $N_\mu$ to have different ranks on different connected components.}.

The following formula follows from localization immediately.
\begin{lem}\label{localizeSlambda}
    We have 
        $e(\BS) \cap [C_{\leq \lambda}] 
        = \sum_{\mu \leq \lambda} c^\bN_{\lambda,\mu} [t^\mu]$, where 
    \begin{equation*}
        c^{\bN}_{\lambda,\mu} = \int_{F_\mu} 
        \frac{e(\widetilde\BS_{\leq\lambda}|_{F_\mu})}{e(N_\mu)}.
    \end{equation*}
\end{lem}

Let us denote
\begin{equation*}
\cZ_{F_\mu}(X,\beta)_n \coloneq F_\mu \times_{\widetilde C_{\leq\lambda}} \cZ_{\leq\lambda}(X,\beta)_n.
\end{equation*}
Equivalently, $\cZ_{F_\mu}(X,\beta)_n$ is cut out from 
$F_\mu \times_{\widetilde C_{\leq\lambda}} \cM_{\leq\lambda}(X,\beta)_n$ 
by the canonical section $\can_{\leq\lambda}(X)$ of $\widetilde\BS_{\leq\lambda}|_{F_\mu}$. 
Recall that, by construction, we have a quotient map of vector bundles
\begin{equation*}
\zeta : F_\mu \times \BS_{t^\mu} \longrightarrow \widetilde\BS_{\leq\lambda}\big|_{F_\mu}.
\end{equation*}
In particular, we have closed embeddings
\begin{equation*}
g_\mu : F_\mu \times \cZ_{t^\mu}(X,\beta)_n \longrightarrow \cZ_{F_\mu}(X,\beta)_n
\end{equation*}
and
\begin{equation*}
f_\mu: F_\mu\times\cZ_{t^\mu}(X,\beta)_n\longrightarrow \cZ_{\leq\lambda}(X,\beta)_n.
\end{equation*}
Note that we have
\begin{equation}\label{eqfmu}
    g_{\mu*}\big([F_\mu]\otimes [\cZ_{t^\mu}(X,\beta)_n]^\vir\big)
    = \frac{e(\BS_{t^\mu})}{e(\widetilde\BS_{\leq\lambda}|_{F_\mu})}\cap 
      [\cZ_{F_\mu}(X,\beta)_n]^\vir.
\end{equation}

\begin{lem}\label{cZvloc}
The following formula holds in 
$H_{T\times\Cxh}^\bullet(\cZ_{\leq\lambda}(X,\beta)_n)_\loc$:
\begin{equation*}
    [\cZ_{\leq\lambda}(X,\beta)_n]^\vir 
    = \sum_{\mu\leq\lambda} 
      f_{\mu*}\!\left(
        \frac {e(\widetilde\BS_{\leq\lambda}|_{F_\mu})}
             {e(N_\mu)\, e(\BS_{t^\mu})}
        [F_\mu]\otimes[\cZ_{t^\mu}(X,\beta)_n]^\vir
      \right).
\end{equation*}
\end{lem}

\begin{proof}
Throughout the proof, all equalities of homology classes are considered in a suitable localized homology, without being mentioned explicitly.  

The inclusion
\[
j_\mu : 
F_\mu \times \sE_{t^\mu}(X) 
\cong F_\mu \times_{\widetilde{C}_{\leq \lambda}} \sE_{\leq \lambda}(X) 
\hookrightarrow \sE_{\leq \lambda}(X),
\]
induces
\[
\overline{j}_\mu : 
F_\mu \times \overline{M}_{0,n+2}(\sE_{t^\mu}(X), \beta) 
\cong F_\mu \times_{\widetilde{C}_{\leq \lambda}} 
        \overline{M}_{0,n+2}(\sE_{\leq \lambda}(X), \beta) 
\hookrightarrow \overline{M}_{0,n+2}(\sE_{\leq \lambda}(X), \beta).
\]
Taking $T \times \Cxh$-fixed loci, we obtain
\[
\overline M_{0,n+2}(\sE_{\leq\lambda}(X),\beta)^{T \times \Cxh} 
= \bigsqcup_{\mu \leq \lambda} 
   F_\mu \times 
   \overline M_{0,n+2}(\sE_{t^\mu}(X),\beta)^{T \times \Cxh}.
\]
By the virtual localization formula, this gives
\begin{equation}\label{LocaliseMvir}
    [\overline M_{0,n+2}(\sE_{\leq\lambda}(X),\beta)]^\vir 
    = \sum_{\mu\leq\lambda}
      \overline j_{\mu*}\!\left(
        \frac{1}{e(N_\mu)} [F_\mu]\times 
        [\overline M_{0,n+2}(\sE_{t^\mu}(X),\beta)]^\vir
      \right).
\end{equation}

Since $F_\mu\times\cM_{t^\mu}(X,\beta)_n$ is isomorphic to the fibre product of 
$\overline M_{0,n+2}(F_\mu\times\sE_{t^\mu}(X),\beta)$ 
and $\cM_{\leq\lambda}(X,\beta)_n$ over 
$\overline M_{0,n+2}(\sE_{\leq\lambda}(X),\beta)$, 
\eqref{LocaliseMvir} implies
\[
[\cM_{\leq\lambda}(X,\beta)_n]^\vir 
= \sum_{\mu\leq\lambda}
  (j_{\cM,\mu})_*\!\left(
     \frac{1}{e(N_\mu)} [F_\mu]\times [\cM_{t^\mu}(X,\beta)_n]^\vir
  \right),
\]
where $j_{\cM,\mu}:F_\mu \times \cM_{t^\mu}(X,\beta)_n\to \cM_{\leq\lambda}(X,\beta)_n$ is the restriction of $\overline{j}_{\mu}$. 

Now consider the diagram
\begin{equation*}
    \begin{tikzcd}
{F_\mu\times\cZ_{t^\mu}(X,\beta)_n} 
  \arrow[rr, "f_\mu", bend left]
  \arrow[r, "g_\mu"]
  \arrow[rd] 
& {\cZ_{F_\mu}(X,\beta)_n}  
  \arrow[r,"j'_{\cM,\mu}"]\arrow[d]   
& {\cZ_{\leq\lambda}(X,\beta)_n}\arrow[d] \ar[r]
& \widetilde{C}_{\leq \lambda}\ar[d,"\mathfrak{z}"]\\ 
& {F_\mu\times\cM_{t^\mu}(X,\beta)_n}  
  \arrow[r, "j_{\cM,\mu}"]
& {\cM_{\leq\lambda}(X,\beta)_n}\ar[r]
& \widetilde{\BS}_{\leq\lambda}.
\end{tikzcd}
\end{equation*}
Both squares are Cartesian. By functoriality of Gysin pullbacks and the excess intersection formula \cite[Section~6]{Fulton}, we obtain
\begin{align*}
    [\cZ_{\leq\lambda}(X,\beta)_n]^\vir 
    &= \sum_{\mu\leq\lambda}
       \mathfrak{z}^! (j_{\cM,\mu})_*\!\left(
          \frac{1}{e(N_\mu)}[F_\mu]\times [\cM_{t^\mu}(X,\beta)_n]^\vir
       \right) \\[4pt]
    &= \sum_{\mu\leq\lambda} 
       (j'_{\cM,\mu})_*\mathfrak{z}^!\!\left(
          \frac{1}{e(N_\mu)}[F_\mu]\times[\cM_{t^\mu}(X,\beta)_n]^\vir
       \right) \\[4pt]
    &= \sum_{\mu\leq\lambda} 
       f_{\mu*}\!\left(
          \frac{e(\widetilde\BS_{\leq\lambda}|_{F_\mu})}
               {e(N_\mu)\,e(\BS_{t^\mu})}
          [F_\mu]\times[\cZ_{t^\mu}(X,\beta)_n]^\vir
       \right),
\end{align*}
where the last step follows from \eqref{eqfmu}.
\end{proof}

Recall that we denote $X^\mu= t^\mu G_\oh\times_{G_\oh} X\subset G_\ok\times_{G_\oh}X$. The restriction of the $X$-bundle $\cX_{\leq\lambda,0}\to \widetilde C_{\leq\lambda}$ to $F_\mu$ is naturally identified with $F_\mu\times X^\mu$ (cf. \Cref{zerofibre}). Let $\pr_{X^\mu}$ be the projection map
\[
\pr_{X^\mu} \colon F_\mu \times X^\mu \to X^\mu.
\]

\begin{lem}\label{Localprop}
For any $e(\widetilde\BS_{\leq\lambda})\cup \gamma \in e(\widetilde{\BS}_{\leq\lambda})\cup H_{T \times \Cx_\hbar}^\bullet(\cX_{\leq\lambda,0})$, we have
\begin{equation}\label{Localpropformula}
    \widetilde\bbS_{\bN,\leq \lambda}(e(\widetilde{\BS}_{\leq\lambda})\cup \gamma) = \sum_{\mu\leq\lambda} \widetilde\bbS_{\mu}(\gamma_\mu),
\end{equation}
where
\begin{equation}\label{gammamu}
\gamma_\mu=(\pr_{X^\mu})_*\left(\frac{e(\widetilde\BS_{\leq\lambda}|_{F_\mu})}{e(N_\mu)}\gamma\big|_{F_\mu\times X^\mu}\right).
\end{equation}
\end{lem}
\begin{proof}
Since $X$ is equivariantly formal, $H_T^\bullet(X)$ is torsion-free over $H_T^\bullet(\pt)$. 
Thus it suffices to check that \Cref{Localpropformula} holds after applying the tensor product 
$\otimes_{H_T^\bullet(\pt)} \Frac(H_T^\bullet(\pt))$. Now we compute
\begin{align*}
    &\pr_{X*}\ev_{\infty*}\left(\ev_0^*(\gamma)\prod_{\ell=1}^n \ev_\ell^*(\hat\tau)\cap [\cZ_{\leq\lambda}(X,\beta)_n]^\vir\right) \\ 
    &= \sum_{\mu\leq\lambda}\pr_{X*}\ev_{\infty*}\left(\frac{e(\widetilde\BS_{\leq\lambda}|_{F_\mu})}{e(N_\mu) e(\BS_{t^\mu})}\ev_0^*(\gamma\big|_{F_\mu\times X^\mu})\prod_{\ell=1}^n \ev_\ell^*(\hat\tau)\cap [F_k]\times [\cZ_{t^\mu}(X,\beta)_n]^\vir\right) \\
    &=\sum_{\mu\leq\lambda}\frac{1}{e(\BS_{t^\mu})}\ev_{\infty*}\left(\ev_0^*\pr_{X_\mu*}\left(\frac{e(\widetilde\BS_{\leq\lambda}|_{F_\mu})}{e(N_\mu)}\gamma\big|_{F_\mu\times X^\mu}\right)\prod_{\ell=1}^n \ev_{\ell}^*(\hat\tau)\cap [\cZ_{t^\mu}(X,\beta)_n]^\vir\right).
\end{align*}
Here the evaluation maps in the first line send $\cZ_{\leq\lambda}$ to $\sE_{\leq\lambda}$; in the second line send $F_\mu\times \cZ_{t^\mu}\subset \cZ_{\leq\lambda}$ to $F_\mu\times \sE_{t^\mu}\subset \sE_{\leq\lambda}$; and in the last line send $\cZ_{t^\mu}$ to $\sE_{t^\mu}$.

In the above, the first equality follows from \Cref{cZvloc} and the second equality follows from projection formula. Multiplying with the equivariant Novikov variables and summing over $n\geq 0$ and section classes yields the desired result. 
\end{proof}
\begin{proof}[Proof of \Cref{TGcompatible}]
Let $\gamma=\tw_G([\widetilde C_{\leq\lambda}]\otimes \alpha)$. By \Cref{twistingcompat} and \Cref{abeliantwlinear}, the restriction of $\gamma$ to $F_\mu\times X^\mu\subset \widetilde G_\ok^{\leq\lambda}\times_{G_\oh} X$ is equal to $[F_\mu]\otimes \Phi_\mu(\alpha)$. Using the projection formula, we obtain
\begin{equation}\label{twistinggammamu}
\gamma_\mu=\left(\int_{F_\mu} \frac{e(\widetilde\BS_{\leq\lambda}|_{F_\mu})}{e(N_\mu)}\right) \cup \Phi_\mu( \alpha) = c^\bN_{\lambda,\mu}\Phi_\mu(\alpha),
\end{equation}
where the class $\gamma_\mu$ defined as in \eqref{gammamu}. Finally, by \Cref{Localprop}, \Cref{localizeSlambda}, and \Cref{twistinggammamu} above, we have
\begin{align*}
    \bbS_{G,\bN}(e(\BS_{\leq\lambda})\cap [ C_{\leq\lambda}]\otimes \alpha)&\coloneq \widetilde\bbS_{\bN,\leq\lambda}(e(\widetilde\BS_{\leq\lambda})\cup \gamma) \\
    &=\sum_{\mu\leq\lambda}\widetilde\bbS_\mu(c^\bN_{\lambda,\mu}\Phi_\mu(\alpha)) \\
    &=\sum_{\mu\leq\lambda}c^\bN_{\lambda,\mu}\bbS_\mu(\alpha) \qedhere
\end{align*}
\end{proof}

\begin{cor}\label{Wequiv}
    The map $\bbS_{G,\bN,X}$ (resp. $\bbS^\xi_{G,\bN,X}$) is $W$-equivariant.
\end{cor}

\begin{proof}
By \Cref{TGcompatible}, it suffices to show that 
\begin{equation*}
w \cdot \bbS_\lambda(\alpha) = \bbS_{w\lambda}(w \cdot \alpha)
\end{equation*}
for any $w \in W$, $\lambda \in \Lambda$, and $\alpha \in H_{T\times\Cxh}^\bullet(X)$.

Indeed, the map $\tw_T$ is $W$-equivariant. So it remains to check that
\begin{equation*}
w \cdot \widetilde\bbS_\lambda(\alpha) = \widetilde\bbS_{w\lambda}(w \cdot \alpha).
\end{equation*}
Observe that the $G_{\C[t^{-1}]} \rtimes \Cxh$-action on $\sE(X)$ induces an $N(T)$-action on
\begin{equation*}
\bigsqcup_{\mu \in \Lambda} \sE_{t^\mu}(X),
\end{equation*}
which is compatible with the natural $N(T)$-actions on both $T_\ok \times_{T_\oh} X$ and on $\Gr_T$. Furthermore, the bundle $\BS_{T,\bN} \to \Gr_T$ is $N(T)$-equivariant. Thus, for each $w \in N(T)$, we have an induced isomorphism
\begin{equation*}
w : \cZ_{t^\mu}(X,\beta)_n \longrightarrow \cZ_{w\cdot t^\mu}(X,w_*\beta)_n
\end{equation*}
such that
\begin{equation*}
[\cZ_{w\cdot t^\mu}(X,w_*\beta)_n]^\vir = w_* [\cZ_{t^\mu}(X,\beta)_n]^\vir.
\end{equation*}

By the projection formula, and the fact that the evaluation maps are $N(T)$-equivariant, this yields the desired result.
\end{proof}

In view of \Cref{Wequiv}, it makes sense to take the $W$-invariant part of $\bbS_{G,\bN,X}$ and $\Psi_{G,\bN,X}$. 
\begin{df}
The $W$-invariant part of $\bbS_{G,\bN,X}$ and $\Psi_{G,\bN,X}$, still denoted by the same symbol and referred to as \emph{the shift operators with matter} $\bN$ and \emph{the Seidel homomorphism with matter} $\bN$ respectively, are the maps
\[\bbS_{G,\bN,X}: \sA^\hbar_{G,\bN}\longrightarrow \operatorname{End}^\bullet(H^\bullet_G(X)[\hbar])[[q_G,\tau]].\]
\[\Psi_{G,\bN,X}: \sA_{G,\bN}\longrightarrow QH^\bullet_G(X)[[q_G,\tau]].\]
Similarly for $\bbS_{G,\bN,X}^\xi$ and $\Psi^\xi_{G,\bN,X}$.
\end{df}

\subsection{Module property}\label{subsection:moduleproperty}
In this subsection, we prove the following theorem.

\begin{thm}\label{moduleprop}
For $\lambda, \mu \in \Lambda$, $\alpha \in H_{T \times \Cx_\hbar}^\bullet(X)_\loc$, 
and $P \in H_{T \times \Cx_\hbar}^\bullet(\pt)$, we have
\begin{equation*}
    \bbS_0=\operatorname{id},\qquad
    \bbS_{\lambda+\mu}(\alpha) = \bbS_{\lambda}\big(\bbS_{\mu}(\alpha)\big), 
    \qquad
    \bbS_{\lambda}(P \alpha) = \Phi_\lambda(P)\,\bbS_{\lambda}(\alpha).
\end{equation*}
A similar statement holds in the presence of a flavour symmetry $\xi$.
\end{thm}

\Cref{IntroThm2} now follows from \Cref{TGcompatible} and \Cref{moduleprop}.

\begin{thm}[=\Cref{IntroThm2}]\label{moduleproperty1}
The operator 
\[\bbS_{G,\bN,X}: \sA^\hbar_{G,\bN}\longrightarrow \operatorname{End}^\bullet(H^\bullet_G(X)[\hbar])[[q_G,\tau]]\]
is a graded $H^\bullet_G(\pt)[\hbar]$-algebra homomorphism. A similar statement holds in the presence of a flavour symmetry $\xi$.
\end{thm}

\begin{proof}
We already know that $\bbS_{G,\bN,X}$ is graded and $$H^\bullet_G(\pt)[\hbar]$$-linear, it remains to show it is an algebra homomorphism. Let $\Gamma = \sum_{\lambda} a_\lambda [t^\lambda]$ and 
$\Gamma' = \sum_{\mu} b_\mu [t^\mu]$, where $a_\lambda, b_\mu \in H_{T \times \Cxh}^\bullet(\pt)$. 
By \Cref{TGcompatible} and \Cref{moduleprop}, we have
\begin{align*}
\bbS_{G,\bN}\big(\Gamma \otimes \bbS_{G,\bN}(\Gamma' \otimes \alpha)\big) 
&= \sum_{\lambda} a_\lambda \bbS_{\lambda}\!\left(\sum_\mu b_\mu \bbS_{\mu}(\alpha)\right) \\
&= \sum_{\lambda,\mu} a_\lambda \,\Phi_\lambda(b_\mu) \cup \bbS_{\lambda+\mu}(\alpha) \\
&= \bbS_{G,\bN}\big((\Gamma \ast \Gamma') \otimes \alpha\big).
\end{align*}
Here we used the formula
\[
[t^\lambda] \ast b_\mu = \Phi_\lambda(b_\mu) \ast [t^\lambda],
\]
which expresses the twisted linearity of the convolution product \eqref{convolution2} 
(cf. \cite[Section~4(ii)]{BFN}). This completes the proof.
\end{proof}
The idea of the proof of \Cref{moduleprop} is similar to that of~\cite[Theorem 3.14]{Iritani}. 
However, since we do not assume $X^T$ is proper, we must take additional care when applying localization techniques. 
We begin with some preparations.

\subsubsection*{Shift operators on the Givental space}
Let $F$ be a smooth algebraic variety equipped with the trivial $T$-action, and let $E$ be a $T$-equivariant vector bundle on $F$. 
For each $\lambda \in \Lambda$, the space $\sE_{t^\lambda}(E)$ is a vector bundle over $\sE_{t^\lambda}(F) = \PP^1 \times F$. 
Let $\Phi_\lambda$ denote the automorphism of $K_{T \times \Cxh}(F)$ induced by the automorphism of 
$T \times \Cxh$ sending $(z, z_\hbar)$ to $(z\lambda(z_\hbar)^{-1}, z_\hbar)$. 
We define $\Delta_\lambda(E)$ to be the following $K$-theory class on $F$:
\begin{equation*}
    \Delta_\lambda(E) \coloneq [E]
    - p_*\big[\sE_{t^\lambda}(E)\big]
    \in K_{T \times \Cxh}(F),
\end{equation*}
where $p : \PP^1 \times F \to F$ is the projection to the second factor. 

It is clear that $\Delta_\lambda$ is additive on short exact sequences of vector bundles, 
and thus induces a map 
\begin{equation*}
\Delta_\lambda : K_{T \times \Cxh}(F) \to K_{T \times \Cxh}(F).
\end{equation*}
Indeed, for any $A \in K_{T \times \Cxh}(F)$, we have
\begin{equation}\label{Deltalambda}
\Delta_\lambda(A) 
= \frac{A - \Phi_\lambda(A)}{1 - \zeta},
\end{equation}
where $\zeta = [\C_\hbar] \in K_{T \times \Cxh}(\pt)$. 

To see this, we may assume $A$ is represented by a vector bundle $E$. 
\Cref{Deltalambda} then follows from the localization formula, together with the observations
\begin{equation*}
[E] = p_*([p^*E]), \qquad 
[\sE_{t^\lambda}(E)|_{\{0\} \times F}] \cong \Phi_\lambda([E]), \qquad 
[\sE_{t^\lambda}(E)|_{\{\infty\} \times F}] \cong [E].
\end{equation*}

\begin{lem}\label{Krelation}
For any $A \in K_{T \times \Cxh}(F)$, we have
\begin{equation*}
\Delta_{\lambda+\mu}(A) = \Delta_{\lambda}(A) + \Phi_\lambda\big(\Delta_{\mu}(A)\big).
\end{equation*}
\end{lem}

\begin{proof}
We may assume $A$ is represented by a vector bundle $E$. 
We can decompose $E$ as $\bigoplus_\alpha E_\alpha$, where $T$ acts on $E_\alpha$ via the character $\alpha : T \to \Cx$. 
Since $\Delta_\lambda$ is additive, we may further assume $E = E_\alpha$ for some character $\alpha$. 
In this case,
\begin{equation*}
\Delta_\lambda(E_\alpha) 
= \frac{E_\alpha - \zeta^{\langle \alpha, \lambda \rangle}\otimes E_\alpha}{1 - \zeta} 
= \left[\langle \alpha, \lambda \rangle\right]_\zeta \otimes E_\alpha,
\end{equation*}
where 
\begin{equation*}
[n]_\zeta \coloneq \frac{1 - \zeta^n}{1 - \zeta}.
\end{equation*}
The lemma now follows from the well-known identity $[m+n]_\zeta = [m]_\zeta + \zeta^m [n]_\zeta$.
\end{proof}

\begin{df}[The shift operator on the Givental space]\label{Giventalshift}
Let $\{F_k : k \in I\}$ be the connected components of $X^T$. 
For each $k \in I$ and $\lambda \in \Lambda$, let 
$i_k : F_k \to X$ and $i_{\lambda,k} : F_k \to X^\lambda$ denote the inclusions. 
We denote by 
$\beta_{\lambda,k} \in \Eff(\sE_{t^\lambda}(X))^{\sec}$ 
the class represented by $\sE_{t^\lambda}(x)$ for some (hence any) $x \in F_k$.

Consider the map
\begin{equation*}
\widetilde{\mathsf{S}}_{\lambda} \colon 
H_{T \times \Cxh}^\bullet(X^\lambda)_\loc
\longrightarrow 
H_{T \times \Cxh}^\bullet(X)_\loc[[q_G,\tau]]
\end{equation*}
defined by
\begin{equation*}
    \widetilde{\mathsf{S}}_{\lambda}(\alpha)\big|_{F_k}
    = q^{\overline{\beta}_{\lambda,k}} \,
      \alpha\big|_{F_k} \cup e\big(\Delta_\lambda(N_{F_k/X})\big),
    \qquad k \in I.
\end{equation*}
Equivalently,
\begin{equation*}
\widetilde{\mathsf{S}}_{\lambda}(\alpha) =
\sum_{k \in I} 
   q^{\overline{\beta}_{\lambda,k}} \,
   i_{k*}\!\left(
      \frac{i_{\lambda,k}^*(\alpha)}{e(N_{F_k/X})}
      \cup e\big(\Delta_\lambda(N_{F_k/X})\big)
   \right).
\end{equation*}

\noindent
The \emph{shift operator on the Givental space} is then defined as
\begin{equation*}
\mathsf{S}_{\lambda} \coloneq \widetilde{\mathsf{S}}_\lambda \circ \Phi_\lambda \colon 
H_{T \times \Cxh}^\bullet(X)_\loc
\longrightarrow 
H_{T \times \Cxh}^\bullet(X)_\loc[[q_G,\tau]].
\end{equation*}
\end{df}

\begin{rem}
    (cf.~\cite[Definition~3.13]{Iritani})
    Let $ N_{F_k/X} = \bigoplus_{\alpha} N_{k,\alpha} $ be the $T$-eigenbundle decomposition of the normal bundle of $ F_k $ in $ X $. 
    Suppose $ r_{k,\alpha,j} $, $ j = 1, \ldots, \mathrm{rk}(N_{k,\alpha}) $, are the Chern roots of $ N_{k,\alpha} $. 
    Then explicitly,
    \begin{equation*}
        e(\Delta_\lambda(N_{F_k/X}))
        = \prod_{\alpha,j}
            \frac{
                \prod_{c=0}^{\infty} 
                (r_{k,\alpha,j} + \alpha + c\hbar)
            }{
                \prod_{c=\alpha(\lambda)}^{\infty} 
                (r_{k,\alpha,j} + \alpha + c\hbar)
            }.
    \end{equation*}
\end{rem}

\begin{prop}\label{sfmodule}
For $\lambda, \mu \in \Lambda$ and $\alpha \in H_{T \times \Cx_\hbar}^\bullet(X)_\loc$, we have
\begin{equation*}
    \mathsf{S}_{\lambda+\mu}(\alpha) = \mathsf{S}_{\lambda}\big(\mathsf{S}_{\mu}(\alpha)\big).
\end{equation*}
\end{prop}

\begin{prop}\label{intertwiningFundamentalsolution}
We have 
\[
\bbS_{\lambda} = \mathbb{M}_X \circ \mathsf S_{\lambda} \circ \mathbb{M}_X^{-1}.
\]
\end{prop}

\begin{proof}[Proof of \Cref{moduleprop}]
The last equality in \Cref{moduleprop} is immediate from the definition \eqref{definitionbbSlambda}. 
Thus it remains to prove the first two equalities. 
By \Cref{intertwiningFundamentalsolution}, it is enough to show that $\mathsf S_0=\operatorname{id}$ and
\begin{align}
    \overline{\beta}_{\lambda+\mu,k} &= \overline{\beta}_{\lambda,k} + \overline{\beta}_{\mu,k}\label{equalitybeta}\\
    \mathsf S_{\lambda+\mu}(\alpha) &= \mathsf S_{\lambda}\big(\mathsf S_{\mu}(\alpha)\big).\label{equalitysfS}
\end{align}
For \eqref{equalitybeta}, it suffices to check the equality in $H_{2}^{T}(x;\Z)$ for some $x \in F_k$, which is clear. The equation \eqref{equalitysfS} is precisely \Cref{sfmodule}. 
\end{proof}

The rest of this subsection is devoted to the proof of \Cref{sfmodule} and \Cref{intertwiningFundamentalsolution}.

\begin{proof}[Proof of \Cref{sfmodule}]
We may assume $\alpha = i_{\lambda,k*}(\alpha')$ for some $k \in I$ and 
$\alpha' \in H_{T \times \Cxh}^\bullet(F_k)_\loc$. 
This reduces the claim to checking the identity
\[
e\big(\Delta_{\lambda+\mu}(N_{F_k/X})\big)
= e\big(\Delta_\lambda(N_{F_k/X})\big) \cdot 
  \Phi_{\lambda}\big(e(\Delta_\mu(N_{F_k/X}))\big),
\]
which follows by taking Euler classes of the equality in \Cref{Krelation} with $A = [N_{F_k/X}]$.
\end{proof}

\begin{proof}[Proof of \Cref{intertwiningFundamentalsolution}]
(cf.~\cite[Proposition~8.2.1]{MO} and \cite[Theorem~3.14]{Iritani}). 
Let $\mathbb M_{X^\lambda}'$ be defined analogously to \Cref{def:fundamental_solution}, 
but with equivariance induced by the $T \times \Cxh$-action on $X^\lambda$. 
We have $\mathbb M'_{X^\lambda} = \Phi_\lambda \circ \mathbb M_X \circ \Phi_\lambda^{-1}$, 
as well as the unitarity relation 
$\mathbb M'_{X^\lambda}(-\hbar)^* = \mathbb M'_{X^\lambda}(\hbar)^{-1}$ 
(see \cite[Section~1]{Giventalelliptic}). 
It therefore suffices to show
\begin{equation*}
\widetilde \bbS_{\lambda}
= \mathbb M_X(\hbar) \circ \widetilde{\mathsf S}_\lambda \circ \mathbb M'_{X^\lambda}(-\hbar)^*.
\end{equation*}

\noindent
A $T \times \Cxh$-fixed stable map $\sigma$ has domain curve of the form 
$\Sigma_0 \cup \Sigma_{\sec} \cup \Sigma_\infty$, where
\begin{itemize}
    \item $\sigma|_{\Sigma_0}$ is a $T$-fixed stable map to the zero fibre $X^\lambda$,
    \item $\sigma|_{\Sigma_\infty}$ is a $T$-fixed stable map to the infinity fibre $X$,
    \item $\sigma|_{\Sigma_{\sec}}$ is a constant section of $\sE_{t^\lambda}(X)$ at a fixed point $x \in X^T$, with $\Sigma_{\sec} \cong \PP^1$.
\end{itemize}

Therefore, the $T \times \Cxh$-fixed locus of $\overline{M}_{0,n+2}(\sE_{t^\lambda}(X),\beta)$ is given by (cf.~\cite{Iritani})
\begin{equation*}
\bigsqcup_k \ \bigsqcup_{I_0 \sqcup I_\infty = \{0,1,2,\dots,n,\infty\}} \
\bigsqcup_{\beta_0+\beta_{\lambda,k}+\beta_\infty=\beta} 
\overline{M}_{0,I_0\cup p}(X^\lambda,\beta_0)^{T}
\times_{F_k}
\overline{M}_{0,I_\infty\cup q}(X,\beta_\infty)^{T},
\end{equation*}
where $p$ and $q$ are the intersection points $\Sigma_0 \cap \Sigma_{\sec}$ and 
$\Sigma_{\infty} \cap \Sigma_{\sec}$, respectively. 

Fix $I_0, I_\infty, \beta_0, \beta_\infty$. 
The virtual normal bundle $\cN_k^\vir$ of 
$\overline M_{0,I_0\cup p}(X^\lambda,\beta_0)\times_{F_k} \overline M_{0,I_\infty\cup q}(X,\beta_\infty)$ 
is given by
\begin{equation*}
    \cN_k^\vir 
    = \cN_0^\vir + \cN_\infty^\vir + \cN_{\sec,k}
    - N_{F_k/X^\lambda} - N_{F_k/X} 
    + L_p^{-1} \otimes \zeta^{-1} + L_{q}^{-1} \otimes \zeta.
\end{equation*}
Here:
\begin{itemize}
    \item $\cN^\vir_0$ is the virtual normal bundle of $\overline{M}_0^T$ in $\overline{M}_{0,I_0\cup p}(X^\lambda,\beta_0)$,
    \item $\cN^\vir_\infty$ is the virtual normal bundle of $\overline{M}_\infty^T$ in $\overline{M}_{0,I_\infty\cup q}(X,\beta_\infty)$,
    \item $\cN_{\sec,k}$ is the moving part of the bundle with fibre $\chi(\sigma|_{\Sigma_{\sec}}^*T\sE_{t^\lambda}(X))$,
    \item $L_p$ and $L_q$ are the universal cotangent lines at the marked points $p$ and $q$.
\end{itemize}

Note that $\chi(\Sigma_{\sec},T\Sigma_{\sec}) = \zeta + 1 + \zeta^{-1}$. 
The short exact sequence
\begin{equation*}
0 \to T\Sigma_{\sec} \to \sigma|_{\Sigma_{\sec}}^*T\sE_{t^\lambda}(X) \to \sE_{t^\lambda}(N_{F_k/X}) \to 0
\end{equation*}
then gives
\begin{equation*}
    \Delta_{\lambda}(N_{F_k/X})
    = \zeta+ \zeta^{-1} + N_{F_k/X} - \cN_{\sec,k}.
\end{equation*}
Hence we may write
\begin{equation*}
\widetilde{\mathsf{S}}_{\lambda}(\alpha) =
\hbar(-\hbar)\sum_{k} 
   q^{\overline{\beta}_{\lambda,k}} \,
   i_{k*}\left(\frac{i_{\lambda,k}^*(\alpha)}{e(\cN_{\sec,k})}\right).
\end{equation*}

We next consider $[\cM_{t^\lambda}(X,\beta)_n]^\vir$ and $[\cZ_{t^\lambda}(X,\beta)_n]^\vir$. 
Recall that $\cM_{t^\lambda}(X,\beta)$ is defined as the fibre over $(0,\infty)$ of the map 
\[
(\pr_{\PP^1}\circ\mathrm{Ev}_0,\;\pr_{\PP^1}\circ\mathrm{Ev}_\infty):
\overline M_{0,n+2}(\sE_{t^\lambda}(X),\beta)\to \PP^1\times\PP^1.
\]

Let $\cM_k$ be the union of 
\[
\overline{M}_{0,I_0\cup p}(X^\lambda,\beta_0)^{T}
\times_{F_k}
\overline{M}_{0,I_\infty\cup q}(X,\beta_\infty)^{T}
\]
for which $0\in I_0$ and $\infty\in I_\infty$. 
Then $\cM_{t^\lambda}(X,\beta)_n^{T\times\Cxh}=\bigsqcup_k \cM_k$, 
and we have the Cartesian diagram
\begin{equation*}
    \begin{tikzcd}[column sep=large, row sep=large]
        \cM_k \arrow[r] \arrow[d] 
        & \overline{M}_0^{T}\times_{F_k}\overline{M}_\infty^{T} 
          \arrow[d,"{(\pr_{\PP^1}\circ\EV_0,\;\pr_{\PP^1}\circ \EV_\infty)}"] \\
        \{(0,\infty)\} \arrow[r,"j"] & \PP^1\times\PP^1
    \end{tikzcd}
\end{equation*}
where 
\[
\overline{M}_0^{T}=\bigsqcup_{I_0,\beta_0}\overline{M}_{0,I_0\cup p}(X^\lambda,\beta_0)^{T}, 
\qquad 
\overline{M}_\infty^{T}=\bigsqcup_{I_\infty,\beta_\infty}\overline{M}_{0,I_\infty\cup q}(X,\beta_\infty)^{T}.
\]
Since $j_{\cM_k}$ is an inclusion of components, the refined Gysin pullback $j^!$ acts by cap product with $\hbar(-\hbar)$.

Next observe that the tautological section 
\[
\can_{t^{\lambda}}(X)_n:\cM_{t^\lambda}(X,\beta)_n\to \BS_{t^\lambda}
\]
is $T\times\Cxh$-equivariant, and $\BS_{t^\lambda}^{T\times\Cxh}=0$. 
It follows that the restriction of $\can_{t^{\lambda}}$ to $\cM_k$ is zero, and hence
\[
\cZ_{t^\lambda}(X,\beta)_n^{T\times\Cxh}
=\cM_{t^\lambda}(X,\beta)_n^{T\times\Cxh}.
\]

An application of the virtual localization formula (cf.~\cite{vloc}) for $\overline M_{0,n+2}(\sE_{t^\lambda}(X),\beta)$, 
together with base change, yields
\begin{equation}\label{vloczerolocus}
\begin{aligned}
    [\cZ_{t^\lambda}(X,\beta)_n]^\vir 
    = e(\BS_{t^\lambda}) \cap\sum_k \iota_{*}\left(
       \frac{j^![\overline{M}_0^{T}\times_{F_k}\overline{M}_\infty^{T}]^\vir}
            {e(\cN_{k}^\vir)}\right)
\end{aligned}
\end{equation}
Here $[\overline M_0^T\times_{F_k}\overline M_\infty^T]^\vir$ and 
$[\overline{M}_{0,I_1\cup p}(X^\lambda,\beta_0)^{T}\times_{F_k}
\overline{M}_{0,I_2\cup q}(X,\beta_\infty)^{T}]^\vir$ 
denote the virtual fundamental classes of the fixed loci of 
$\overline{M}_{0,n+2}(\sE_{t^\lambda}(X),\beta)$ in the sense of \cite{vloc}.

We need to compare $[\overline M_0^T\times_{F_k}\overline M_\infty^T]^\vir$ with the virtual fundamental class of $M_0^T$ and $M_\infty^T$. Consider the Cartesian diagram
\begin{equation*}
    \begin{tikzcd}
        \overline M_0^T\times_{F_k}\overline M_\infty^T \arrow[r,"\eta"] \arrow[d,"\ev_{p=q}"] & \overline M_0^T\times\overline M_\infty^T \arrow[d,"\ev_p\times\ev_q"]\\ 
        F_k \arrow[r,"\delta"] & F_k\times F_k .
    \end{tikzcd}
\end{equation*}
\begin{lem}
We have
\[
[\overline M_0^T\times_{F_k}\overline M_\infty^T]^\vir
= \delta^!\big([\overline M_0^T]^\vir \otimes [\overline M_\infty^T]^\vir\big).
\]
\end{lem}

\begin{proof}
Let $\mathcal C$ be the pullback of the universal curve over 
$\overline M_{0,n+2}(\sE_{t^\lambda}(X),\beta)$ along the inclusion of the fixed component. 
We write
\begin{equation*}
    \begin{tikzcd}
        \mathcal C \arrow[r,"f"] \arrow[d,"\pi"] 
        & \sE_{t^\lambda}(X) \\ 
        \overline M_0^T\times_{F_k}\overline M_\infty^T
    \end{tikzcd}.
\end{equation*}
Then the obstruction theory on 
$\overline M^T_0\times_{F_k}\overline M^T_\infty$, 
viewed as a union of fixed components of 
$\overline{M}_{0,n+2}(\sE_{t^\lambda}(X),\beta)$, 
is given by $F\to \mathbb L_{\overline M^T_0\times_{F_k}\overline M^T_\infty}$, 
where $F^\vee=\mathrm{Cone}(A\to B)$ with
\begin{equation*}
    A = (R\pi_* R\mathcal{H}om(\Omega^1_{\mathcal C}(\mathbf x),\oh_{\mathcal C}))^{\mathrm{fix}},
    \qquad 
    B = (R\pi_*f^*T_{\sE_{t^\lambda}(X)})^{\mathrm{fix}}.
\end{equation*}
Here $\Omega^1_{\mathcal{C}}$ is the sheaf of relative differentials of $\mathcal{C}$ over 
$\overline M^T_0\times_{F_k}\overline M^T_\infty$, 
and $\mathbf x=\sum_{i=1}^{n+2}\mathbf x_i$ is the divisor of marked points.  

We also have a similar diagram for $\overline M_0^T\times \overline M^T_\infty$:
\begin{equation*}
    \begin{tikzcd}
        \mathcal C'_0\times \mathcal C'_\infty \arrow[r,"f'_0\times f'_\infty"] 
        \arrow[d,"\pi'_0\times\pi'_\infty"] 
        & X^\lambda\times X \\ 
        \overline M_0^T\times\overline M_\infty^T
    \end{tikzcd}
\end{equation*}
and the obstruction theory on $\overline M^T_0\times\overline M^T_\infty$ 
is given by $E\to \mathbb L_{\overline M^T_0\times\overline M^T_\infty}$, 
where $E^\vee=\mathrm{Cone}(A'\to B')$ with
\begin{align*}
    A' &= 
    (R\pi'_{0*}R\mathcal{H}om(\Omega^1_{\mathcal{C}'_0}(\mathbf x_0+\mathbf p),\oh_{\mathcal C_0'}))^{\mathrm{fix}}
    \boxplus 
    (R\pi'_{\infty*}R\mathcal{H}om(\Omega^1_{\mathcal{C}'_\infty}(\mathbf x_\infty+\mathbf q),\oh_{\mathcal C'_\infty}))^{\mathrm{fix}},\\
    B' &= 
    (R\pi'_{0*}f_0'^*T_{X^\lambda})^{\mathrm{fix}}
    \boxplus 
    (R\pi'_{\infty*}f_\infty'^* T_X)^{\mathrm{fix}}.
\end{align*}
Here $\Omega^1_{\mathcal{C}'_0}$ and $\Omega^1_{\mathcal{C}'_\infty}$ denote the sheaves of relative differentials of $\pi_0$ and $\pi_\infty$ respectively, and $\mathbf x_0,\mathbf x_\infty,\mathbf p,\mathbf q$ are the divisors associated to the marked points.

We claim that there is a distinguished triangle
\[
\eta^*E \to F \to \mathbb L_{F_k/F_k\times F_k} \to \eta^*E[1]
\]
in the derived category of $\overline M^T_0\times_{F_k}\overline M^T_\infty$ compatible with the morphisms to cotangent complexes. 
By \cite[Proposition~7.5]{Intrinsic}, this implies the lemma.

First note that the universal curve $\mathcal C$ decomposes as 
$\mathcal C_0\cup \mathcal C_{\sec}\cup \mathcal C_{\infty}$, 
where $\mathcal C_0$ and $\mathcal C_\infty$ are the pullbacks of 
$\mathcal C'_0$ and $\mathcal C'_\infty$ along the natural projections 
$\overline M_0^T\times_{F_k}\overline M^T_\infty\to \overline M^T_0$ 
and $\overline M_0^T\times_{F_k}\overline M^T_\infty\to \overline M^T_\infty$. 
We write $\pi_i, f_i$ ($i=0,\infty$) for the pullbacks of $\pi_i', f_i'$. 
Let $\iota_0,\iota_{\sec},\iota_\infty$ be the inclusions of 
$\mathcal C_0,\mathcal C_{\sec},\mathcal C_\infty$ into $\mathcal C$. 
Let $\sigma_{\mathbf p},\sigma_{\mathbf q}$ be the sections 
$\overline M^T_0\times_{F_k}\overline M^T_\infty \to \mathcal C$ 
defining the nodal points $p$ and $q$.

We have the following exact triangles in $D^b(\mathcal C)$:
\begin{gather*}
\iota_{0*}R\mathcal{H}om\!\big(\Omega^1_{\mathcal C_0}(\mathbf{x}_0+\mathbf{p}), \oh_{\mathcal C_0}\big)
\oplus \iota_{\sec*}\oh_{\mathcal C_{\sec}}
\oplus \iota_{\infty*}R\mathcal{H}om\!\big(\Omega^1_{\mathcal C_\infty}(\mathbf{x}_\infty +\mathbf{q}),\oh_{\mathcal C_{\infty}}\big) \\
\to R\mathcal{H}om\!\big(\Omega^1_{\mathcal C}(\mathbf x), \oh_{\mathcal C}\big) 
\to \sigma_{\mathbf p*}L_p^{-1}\otimes\C_{-\hbar} 
   \oplus \sigma_{\mathbf q*}L_q^{-1}\otimes \C_{\hbar}\stackrel{+1}{\to},\\
f^*T_{\sE_{t^\lambda}(X)}\to 
\iota_{0*}f^*_0T_{\sE_{t^\lambda}(X)}
\oplus \iota_{\sec*}f_{\sec}^*T_{\sE_{t^\lambda}(X)}
\oplus \iota_{\infty*}f_{\infty}^*T_{\sE_{t^\lambda}(X)}
\to \sigma_{\mathbf p*}\sigma_{\mathbf p}^*f^*T_{\sE_{t^\lambda}(X)}
   \oplus \sigma_{\mathbf q*}\sigma_{\mathbf q}^*f^*T_{\sE_{t^\lambda}(X)}\stackrel{+1}{\to}.
\end{gather*}

Applying $R\pi_*$ and taking the fixed part, we obtain
\begin{align*} 
A &\simeq 
   R\pi_{0*}R\mathcal{H}om\!\big(\Omega^1_{\mathcal C_0}(\mathbf x_0+\mathbf p),\oh_{\mathcal C_0}\big)^{\mathrm{fix}}
   \oplus R\pi_{\sec*}\oh_{\mathcal C_{\sec}}
   \oplus R\pi_{\infty*}R\mathcal{H}om\!\big(\Omega^1_{\mathcal C_{\infty}}(\mathbf x_\infty+\mathbf q),\oh_{\mathcal C_{\infty}}\big)^{\mathrm{fix}}, \\[4pt]
B &\simeq 
   \mathrm{Cone}\!\left( 
      (R\pi_{0*}f_0^*T_{\sE_{t^\lambda}(X)})^{\mathrm{fix}}
      \oplus R\pi_{\sec*}\oh_{\mathcal{C}_{\sec}} 
      \oplus (R\pi_{\infty*} f_\infty^* T_{\sE_{t^\lambda}(X)})^{\mathrm{fix}}
      \to T_{F_k} 
   \right)[-1].
\end{align*}

Hence
\[
\mathrm{Cone}(A\to B)\simeq \mathrm{Cone}\big(\mathrm{Cone}(\eta^*A'\to \eta^*B')\to T_{F_k}\big)[-1].
\]
Therefore, we obtain a natural morphism
\[
F^\vee\to \eta^*E^\vee
\]
whose cone is quasi-isomorphic to $T_{F_k}$. 
Finally, since $\mathbb L_{F_k/F_k\times F_k}\cong T^*_{F_k}[1]$, the result follows.
\end{proof}

For the rest of the argument, we restrict attention to those components in 
$\overline M^T_0\times_{F_k}\overline M^T_\infty$ such that $\sigma$ is mapped 
to the zero and infinity fibres by $\ev_0$ and $\ev_\infty$, respectively. 
We will continue to denote these components by the same notation. 
Consider the correspondence diagram:
\begin{equation*}
\begin{tikzcd}[column sep={between origins, 1.5cm}, row sep=1.8em]
  & & & 
  \overline M^T_0 \times_{F_k} \overline M^T_\infty 
      \arrow[ld, "\ev^{\prime\prime}_q"'] 
      \arrow[rd, "\ev^{\prime\prime}_p"] 
      \arrow[lllddd, bend right=40, "\widetilde\ev_0"']
      \arrow[rrrddd, bend left=40, "\widetilde\ev_\infty"]
  & & & \\
  & & \overline M^T_0 \arrow[ld, equal] \arrow[rd, "\ev'_p"] 
  & & \overline M^T_\infty \arrow[ld, "\ev'_q"'] \arrow[rd,equal] 
  & & \\
  & \overline M_0^T \arrow[ld, "\ev_0"'] \arrow[rd, "\ev_p"] 
  & & F_k \arrow[ld, "{i_{\lambda,k}}"'] \arrow[rd, "{i_{k}}"] 
  & & \overline M_\infty^T \arrow[ld, "\ev_q"'] \arrow[rd, "\ev_\infty"] 
  & \\
X^\lambda & & X^\lambda & & X & & X
\end{tikzcd}
\end{equation*}

\begin{lem}
Let $\Gamma_0 \in H_\bullet^{T\times\Cxh}(\overline M_0^T)$, 
$\Gamma_\infty \in H_\bullet^{T\times\Cxh}(\overline M_\infty^T)$, 
and $\gamma \in H^\bullet_{T\times\Cxh}(F_k)$. 
Set 
\begin{equation*}
    \Gamma = \big(\gamma \cup e(N_{F_k/X^\lambda}) \cup e(N_{F_k/X})\big) 
    \cap \delta^!\big( \Gamma_0 \otimes \Gamma_\infty\big)
    \in H_\bullet^{T\times\Cxh}(\overline M^T_0\times_{F_k}\overline M^T_\infty).
\end{equation*}
Then for any $\alpha \in H^\bullet_{T\times\Cxh}(X^\lambda)$, we have
\begin{equation*}
    \ev_{\infty*}\Big(
        \ev_q^* i_{k*} \big( i_{\lambda,k}^* \PD \ev_{p*}(\ev_0^*(\alpha) \cap \Gamma_0)\cup\gamma \big)
        \cap \Gamma_\infty
    \Big)
    = \widetilde\ev_{\infty*}\big(\widetilde\ev_0^*(\alpha) \cap \Gamma\big).
\end{equation*}
\end{lem}

\begin{proof}
The proof is by diagram chasing, and we include the details for completeness. 
The lemma follows from the three identities
\begin{align}
   i_{\lambda,k}^* \PD\ev_{p*}(\Gamma'_0)
   &= \PD\big(e(N_{F_k/X^\lambda})\cap\ev'_{p*}(\Gamma_0')\big), 
   \label{conveq1}\\[4pt]
   \ev_q^* i_{k*}(\gamma')\cap \Gamma_\infty
   &= \big(e(N_{F_k/X})\cup\ev'^*_q(\gamma')\big)\cap\Gamma_\infty,
   \label{conveq2}\\[4pt]
   \big(\ev'^*_q\PD(\gamma\cap\ev'_{p*}(\Gamma_0''))\big)\cap \Gamma_\infty
   &= \ev''_{p*}\big(\gamma\cap\delta^!(\Gamma_0''\otimes\Gamma_\infty)\big),
   \label{conveq3}
\end{align}
valid for $\Gamma'_0,\Gamma''_0\in H_{\bullet}^{T\times\Cxh}(\overline{M}^T_0)$ 
and $\gamma'\in H^\bullet_{T\times\Cxh}(F_k)$. 
These correspond to the left square, the right square, and the upper square of the diagram above.

\Cref{conveq1} follows from
\begin{equation*}
    \PD^{-1}i_{\lambda,k}^* \PD\ev_{p*}(\Gamma'_0)
    = i_{\lambda,k}^! \ev_{p*}(\Gamma'_0)
    = i_{\lambda,k}^! i_{\lambda,k*}\ev'_{p*}(\Gamma'_0)
    = e(N_{F_k/X^\lambda})\cap \ev'_{p*}(\Gamma'_0).
\end{equation*}

\Cref{conveq2} follows from
\begin{equation*}
    \ev_q^* i_{k*}(\gamma')
    = \ev'^*_q i_{k}^* i_{k*}(\gamma')
    = e(N_{F_k/X})\cup\ev'^*_q(\gamma').
\end{equation*}

For \Cref{conveq3}, we may assume $\gamma=1$. 
Consider the diagram
\begin{equation*}
    \begin{tikzcd}
        \overline M^T_0\times_{F_k} \overline M^T_\infty 
        \arrow[r,"\eta"] \arrow[d,"\ev''_p"] 
        & \overline M^T_0\times \overline M^T_\infty 
        \arrow[d,"\ev'_p\times \id"] \\
        \overline M^T_\infty 
        \arrow[r,"\delta'"] \arrow[d,"\ev'_q"] 
        & F_k\times \overline M^T_\infty 
        \arrow[d,"\id\times \ev'_q"] \\
        F_k \arrow[r,"\delta"] 
        & F_k\times F_k
    \end{tikzcd}
\end{equation*}
Both squares are Cartesian, and $\delta,\delta'$ are regular embeddings of the same codimension. 
Thus \eqref{conveq3} follows from
\begin{equation*}
    \big(\ev'^*_q\PD(\ev'_{p*}(\Gamma_0''))\big)\cap \Gamma_\infty
    = \delta^!(\ev'_{p*}(\Gamma_0'')\otimes \Gamma_\infty)
    = \ev''_{p*}\delta^!(\Gamma_0''\otimes \Gamma_\infty).
\end{equation*}
\end{proof}
Now let $\tau\in H_{T\times\Cxh}^\bullet(X)$ and set 
$\tau'=\Phi_\lambda(\tau)\in H_{T\times\Cxh}^\bullet(X^\lambda)$. 
Note that $\hat\tau|_X=\tau$ and $\hat\tau|_{X^\lambda}=\tau'$. 
In the notation of the previous lemma, define for each $a,b\in \Z_{\geq 0}$:
\begin{align*}
    \Gamma_0 = \Gamma_0^a 
    &= \frac{\prod_{\ell=1}^{a}\ev_{\ell}^*(\tau')}
            {(-\hbar-\psi_p)\,e(\cN_0^\vir)} 
       \cap[\overline M_0^T]^\vir,\\
    \Gamma_\infty = \Gamma_\infty^b 
    &= \frac{\prod_{\ell=1}^{b}\ev_{\ell}^*(\tau)}
            {(\hbar-\psi_q)\,e(\cN_\infty^\vir)} 
       \cap[\overline M_\infty^T]^\vir,\\
    \gamma &= \frac{\hbar(-\hbar)}{e(\cN_{\sec,k})}.
\end{align*}
Hence,
\begin{align*}
    \Gamma = \Gamma^{a,b}
    &= \frac{\hbar(-\hbar)\,e(N_{F_k/X^\lambda})\,e(N_{F_k/X})
             \prod_{\ell=1}^{a}\ev^*_{\ell}(\tau') 
             \prod_{\ell=1}^{b}\ev^*_{\ell}(\tau)}
            {e(\cN_{\sec,k})(-\hbar-\psi_p)(\hbar-\psi_q)
             e(\cN^\vir_0)\,e(\cN^\vir_\infty)} 
       \cap \delta^!\big([\overline M_0^T]^\vir \otimes [\overline M_\infty^T]^\vir\big)\\[4pt]
    &=  \frac{\hbar(-\hbar)\prod_{\ell=1}^{a}\ev^*_{\ell}(\tau') 
             \prod_{\ell=1}^{b}\ev^*_{\ell}(\tau)}{e(\cN^\vir_k)}
       [\overline M^T_0\times_{F_k}\overline M^T_\infty]^\vir. \\[4pt]
       &=  \frac{\prod_{\ell=1}^{a}\ev^*_{\ell}(\tau') 
             \prod_{\ell=1}^{b}\ev^*_{\ell}(\tau)}{e(\cN^\vir_k)}
       j^![\overline M^T_0\times_{F_k}\overline M^T_\infty]^\vir.
\end{align*}

Therefore, by \Cref{vloczerolocus}, for $\alpha\in H_\bullet^{T\times\Cxh}(X^\lambda)$, we obtain
\begin{align*}
    \widetilde\bbS_{\lambda}(\alpha)
    &= \sum\frac{q^{\overline\beta}}{a!\,b!} 
       \PD\,\widetilde\ev_{\infty*}\!\left(
          \widetilde\ev_0^*(\alpha)\cap\Gamma^{a,b}
       \right) \\[6pt]
    &= \PD \sum_{b,\beta_\infty}\frac{q^{\overline\beta_\infty}}{b!}\,
       \ev_{\infty*}\!\Bigg(
          \ev_q^*\sum_{k}q^{\overline{\beta}_{\lambda,k}}\,i_{k*}\Bigg(
             i_{\lambda,k}^* \PD \ev_{p*}\!\left(
                \sum_{a,\beta_0}\frac{q^{\overline\beta_0}}{a!}\,
                \ev_0^*(\alpha)\cap \Gamma_0^a
             \right) 
             \cdot \frac{\hbar(-\hbar)}{e(\cN_{\sec})}
          \Bigg) \cap \Gamma_\infty^b
       \Bigg) \\[6pt]
    &= \mathbb M_X(\hbar)\circ \widetilde{\mathsf S}_\lambda 
       \circ \mathbb M_{X^\lambda}(-\hbar)^*(\alpha),
\end{align*}
where the summation on the first line ranges over all $a,b\geq 0$, all fixed components $\overline M^T_0\times_{F_k}\overline M^T_{\infty}$, all $\beta$ and all $F_k$. This concludes the proof.
\end{proof}

\subsubsection*{More properties}

We now collect other properties of $\bbS_{G,\bN,X}$ that follow from the localization result of \Cref{TGcompatible}.

\begin{prop}\label{CommuteConnection}
    For any $\Gamma \in \sA^\hbar_{G,\bN}$, the shift operator $\bbS_{G,\bN}(\Gamma,-)$ commutes with the quantum connections (see \Cref{def:quantum_connection}), i.e.,
    \begin{equation*}
        [\nabla_{\tau^{i,j}},\bbS_{G,\bN}(\Gamma)] = 0, \quad
        [\nabla_{\hbar\partial_\hbar},\bbS_{G,\bN}(\Gamma)] = 0, \quad
        [\nabla_{Dq\partial_q}, \bbS_{G,\bN}(\Gamma)] = 0.
    \end{equation*}
    A similar statement holds in the presence of a flavour symmetry $\xi$.
\end{prop}
\begin{proof}
By \Cref{TGcompatible}, it suffices to show that $\bbS_\lambda$ commutes with quantum connections for each $\lambda \in \Lambda$. By \eqref{qconnintertwine} and \Cref{intertwiningFundamentalsolution}, the claim reduces to showing that each $S_{\lambda}$ commutes with $\partial_{\tau^{i,j}}$, $\hbar \partial_\hbar - \hbar^{-1}(c_1^T(X) \cup) + \mu_X$, and $Dq \partial_q + \hbar^{-1}(D \cup)$, which follows easily from a direct computation using \Cref{Giventalshift} (cf.\ \cite[Corollary 2.11]{Iritaniblowup}).
\end{proof}

\begin{prop}
There is a commutative diagram:
\begin{equation*}
\begin{tikzcd}
\sA_{T,\bN}^\hbar \otimes_{\C[\hbar]} H_{G \times \Cxh}^\bullet(X)\arrow[d] \arrow[r, "{\bbS_{T,\bN}}"] 
  & H_{T \times \Cxh}^\bullet(X)[[q,\tau]] \arrow[d] \\
e(\BS_\bN) \cap H_\bullet^{\BI \rtimes \Cxh}(\Gr_G) \otimes_{\C[\hbar]} H_{G \times \Cxh}^\bullet(X) \arrow[r, "{\bbS_{G,\bN}}"] 
  & H_{T \times \Cxh}^\bullet(X)[[q,\tau]]
\end{tikzcd}
\end{equation*}
A similar diagram holds in the presence of a flavour symmetry $\xi$.
\end{prop}

\begin{proof}
This follows directly from \Cref{TGcompatible}, applied both for $G$ and for $T$. Recall we use $G$-equivariant Novikov and Bulk variables for both $\bbS_{T,\bN}$ and $\bbS_{G,\bN}$.
\end{proof}

By specializing $\hbar=0$, we also obtain properties for Seidel representation and Seidel homomorphisms. 

\begin{cor}\label{compatqprod}
The map $\bbS_{G,\bN}^{\hbar = 0}$ is compatible with the quantum product, in the sense that
\begin{equation*}
\bbS^{\hbar = 0}_{G,\bN}(\Gamma)( \alpha_1 \star_{\tau} \alpha_2) = \bbS^{\hbar = 0}_{G,\bN}(\Gamma)(\alpha_1) \star_{\tau} \alpha_2.
\end{equation*}
In particular,
\begin{equation*}
\bbS^{\hbar = 0}_{G,\bN}(\Gamma)(\alpha) = \Psi_{G,\bN}(\Gamma) \star_{\tau} \alpha.
\end{equation*}
A similar statement holds in the presence of a flavour symmetry $\xi$.
\end{cor}

\begin{proof}
This follows from \Cref{CommuteConnection}. Setting $\hbar=0$ in the equality $[\hbar\nabla_{\tau^{i,j}}, \bbS_{G,\bN}(\Gamma)]=0$ yields the desired result.
\end{proof}

\begin{cor}
The map $\Psi_{G,\bN}$ (resp. $\Psi_{G,\bN}^\xi$) is a ring homomorphism.
\end{cor}
\begin{proof}
This follows directly from \Cref{moduleproperty1} and \Cref{compatqprod}:
\begin{equation*}
\bbS^{\hbar=0}_{G,\bN}(\Gamma \ast \Gamma')(1)
=
\bbS^{\hbar=0}_{G,\bN}(\Gamma)\bigl(\bbS^{\hbar=0}_{G,\bN}(\Gamma')(1)\bigr)
=
\Psi_{G,\bN,X}(\Gamma)\star_\tau \Psi_{G,\bN,X}(\Gamma').
\qedhere
\end{equation*}
\end{proof}

For $a, b \in \sA_{G,\bN}$, choose lifts $\widetilde{a}, \widetilde{b} \in \sA_{G,\bN}^\hbar$. There is a Poisson algebra structure on $\sA_{G,\bN}$ given by the bracket
\begin{equation*}
\{a, b\} = \frac{1}{\hbar}(\widetilde{a} \widetilde{b} - \widetilde{b} \widetilde{a}) \mod \hbar.
\end{equation*}
It is known that this Poisson bracket induces a symplectic structure on the smooth locus of $\spec \sA_{G,\bN}$ (see \cite[Proposition 6.15]{BFN}).

\begin{df}
Let $R$ be a commutative ring, and let $\sA_R$ be a finite type, commutative, flat $R$-algebra equipped with an $R$-linear Poisson bracket. Suppose the Poisson bracket is non-degenerate on $\sA = \sA_R/\mathfrak{m}\sA_R$ for a maximal ideal $\mathfrak{m} \subset R$. Then a closed subscheme $Z \subset \spec \sA_R$ is said to be a family of \emph{Lagrangians} over $R$ if 
\begin{enumerate}
    \item $Z$ is flat over $R$, 
    \item the radical of its defining ideal is a Lie subalgebra of $\sA_R$, and 
    \item $\dim Z \;=\; \dim R + \tfrac{1}{2} \dim \sA_R.$
\end{enumerate}
\end{df}
Note that if the fibres of $\spec \sA_R \to \spec R$ are smooth, then the above definition agrees with the usual notion of a coisotropic family of Lagrangian subvarieties in a family of smooth symplectic manifolds (see \cite[Proposition~1.5.1]{ChrissGinzburg}).

Let
\begin{equation*}
    \Psi_{G,\bN,X}^{q}:\sA_{G,\bN}[[q,\tau]]\to QH_{G}^\bullet(X)[[q,\tau]]
\end{equation*}
be defined by applying $\Psi_{G,\bN,X}$ on each coefficients of a power series expansion in $q_G$.
\begin{prop}\label{prop:Lagrangian}
The morphism of schemes
\begin{equation*}
    \spec\Psi_{G,\bN,X}^{q}:\spec QH_{G}^\bullet(X)
    \;\longrightarrow\; \spec \sA_{G,\bN}[[q,\tau]]
\end{equation*}
is finite, and the image $V(\Psi_{G,\bN,X}^{q})$ is a family of Lagrangians over $\C[[q_G,\tau]]$. Moreover, if all infinite sums involved in the quantum product and shift operators converge upon evaluation at a homomorphism
\begin{equation*}
    q_0 \colon \C[q,\tau] \to \C,
\end{equation*}
and if we denote by
\begin{equation*}
   \Psi_{G,\bN,X}^{q_0}: \Psi_{G,\bN,X}^{q_0} \colon \sA_{G,\bN} \to QH_{G}^\bullet(X)
\end{equation*}
the resulting specialization, then $V(\ker \Psi^{q_0}_{G,\bN,X})$ is Lagrangian (i.e., a family of Lagrangians over $\C$). Similarly, in the presence of a flavour symmetry $\xi$, we have a finite morphism
\begin{equation*}
    \spec QH_{\hG}^\bullet(X)
    \;\longrightarrow\; \spec \sA_{G,\bN}^\xi[[q_G,\tau]]
\end{equation*}
whose image is a family of Lagrangians over $\C[\lt_\xi][[q,\tau]]$.
\end{prop}
\begin{proof}
This follows from \cite[Theorem I]{Gabber} (see also \cite[page 56]{ChrissGinzburg}).
\end{proof}
\begin{rem}
    It can be shown that the image of $\spec(\Psi_{G,\bN,X}^{q})$ intersects the smooth locus of the morphism $\spec \sA_{G,N}^{q}[q,\tau]\to \spec \C[[q,\tau]]$ nontrivially.
\end{rem}

\begin{df}\label{Non-equivariant}
    The linear maps
\begin{equation*}
        \bbS^{\mathrm{noneq}}_{G,\bN,X} \colon\sA_{G,\bN}^\hbar\longrightarrow\Hom(H^\bullet_{G}(X)[\hbar],QH^\bullet(X)[\hbar])[[q_G,\tau]], 
\end{equation*}
and 
\begin{equation*}
        \Psi^{\mathrm{noneq}}_{G,\bN,X} \colon\sA_{G,\bN}\longrightarrow QH^\bullet(X) .
\end{equation*}
    are obtained rom $\bbS_{G,\bN,X}$ and $\Psi_{G,\bN,X}$ by specializing the equivariant parameters $H_G^\bullet(\pt)$ to zero. Moreover, $\Psi^{\mathrm{noneq}}_{G,\bN,X}$ defines a ring homomorphism. A similar statement holds in the presence of a flavour symmetry $\xi$.
\end{df}
We refer to $\bbS^{\mathrm{noneq}}_{G,\bN,X}$ and $\Psi_{G,\bN,X}^{\mathrm{noneq}}$ as the non-equivariant limits of the shift operator $\bbS_{G,\bN,X}$. See \Cref{Non-equivariantExample} for a sample calculation.

\subsection{Coproduct structure of Coulomb branches}\label{coproduct section}
The diagonal map
\begin{equation*}
    \Delta:\Gr_G\longhookrightarrow \Gr_G\times\Gr_G
\end{equation*}
induces a pushforward homomorphism
\begin{equation}\label{Delta}
    \Delta_*: \sA_G^\hbar\to \sA_G^\hbar\otimes_{H_{G\times\Cxh}^\bullet(\pt)}\sA_G^\hbar,
\end{equation}
which defines a coproduct on $\sA^\hbar_G.$ It is easy to show that $\Delta_*$ is coassociative. If we set $\hbar=0$, then $\spec \sA_G$ can be identified with the regular centralizer of the Langlands dual group $\widecheck{G}$ (\cite{BFM}), which admits a symplectic groupoid structure where multiplication map is induced by $\Delta_*$.

Moreover, if $\bN_1$ and $\bN_2$ are two representations of $G$, it is clear that $\Delta^{-1} (\BS_{G, \bN_1}\boxtimes \BS_{G, \bN_2}) \cong \BS_{G, \bN_1 \oplus \bN_2}$. Therefore, \eqref{Delta} restricts to give a homomorphism (denoted by the same symbol)
\begin{equation*}
    \Delta_{*}: \;\sA^{\hbar}_{G, \bN_1 \oplus \bN_2}\;\longrightarrow\;
    \sA^\hbar_{G, \bN_1} \otimes_{H_{G \times \Cxh}^\bullet(\pt)} \sA^\hbar_{G, \bN_2}.
\end{equation*}
In particular, for each representation $\bN$, one can show that $\sA^\hbar_{G,\bN}$ admits the structure of an $\sA^\hbar_G$-comodule. This induces a symplectic groupoid action of $\spec \sA_G$ on $\spec \sA_{G, \bN}$.

Recall that for a product variety, its quantum cohomology satisfies a Künneth formula \cite{Behrend}. The equivariant version also holds: if $X_1$ and $X_2$ are smooth semiprojective $G$-varieties, there is a homomorphism
\begin{equation}
\sigma:\ 
H^\bullet_{G\times\Cxh}(X_1)
\otimes_{H^\bullet_{G\times\Cxh}(\pt)}
H^\bullet_{G\times\Cxh}(X_2)\xrightarrow{\simeq}H^\bullet_{\Delta(G)\times\Cxh}(X_1\times X_2).
\end{equation}
In the following, $X_1\times X_2$ is equipped .with respect to the diagonal $G$-action. 

\begin{prop}\label{kunneth}
Let $f_i : X_i \to \bN_i$ be a $G$-equivariant proper morphism for $i = 1, 2$, satisfying the assumptions of \Cref{sectionshiftoperators}. Then we have
\begin{equation*}
\bbS_{G,\bN_1\oplus\bN_2}(\Gamma)\bigl(\sigma(\gamma_1\otimes\gamma_2)\bigr)
=
\sigma\bigl((\bbS_{G,\bN_1}\otimes \bbS_{G,\bN_2})(\Delta_*(\Gamma))(\gamma_1\otimes\gamma_2)\bigr),
\end{equation*}
for any $\Gamma\in \sA^\hbar_{G,\bN_1\oplus\bN_2}$, any $\gamma_i \in H_{G\times\Cxh}^\bullet(X_i)$, for $i=1,2$.
\end{prop}

\begin{proof}
Using \Cref{TGcompatible} and observing $\Delta_*([t^\lambda])=[t^\lambda]\otimes [t^\lambda]$, we only need to prove 
\begin{equation*}
    \bbS_{X_1\times X_2,\lambda}(\sigma(\gamma_1\otimes\gamma_2))=\sigma(\bbS_{X_1,\lambda}(\gamma_1)\otimes_{H^\bullet_{G\times\Cxh}(\pt)} \bbS_{X_2,\lambda}(\gamma_2)),
\end{equation*}
where $\bbS_{X_i,\lambda}$ denotes the abelian shift operators for $X_i$ (see \Cref{definitionbbSlambda}). This follows easily from \Cref{Giventalshift,intertwiningFundamentalsolution}. 
\end{proof}

\section{Applications/calculations}\label{Applications}
\subsection{Examples}
\begin{exmp}\label{abeliancase}
Let $G = T = (\Cx)^k$ and $X = \bN = \bigoplus_{j=1}^n \C_{\eta_j}$. Then
\begin{equation*}
    \bbS_{T}\big(e(\BS_{t^\lambda})\cap [t^\lambda]\big)(1) 
    = q^{\lambda}\prod_{j} 
      \prod_{c=0}^{\langle\eta_j,\lambda\rangle-1}(\eta_j+c\hbar),
\end{equation*}
and
\begin{equation}\label{abelianformula}
    \bbS_{T}([t^\lambda])(1) 
    = q^{\lambda}\,
      \frac{\prod_{j}\prod_{c=0}^{\langle\eta_j,\lambda\rangle-1}
                 (\eta_j+c\hbar)}
           {\prod_{j}\prod_{c=0}^{-\langle\eta_j,\lambda\rangle-1}
                 (\eta_j+c\hbar)}=q^{\lambda}\,
      \frac{\prod_{j}\prod_{c=0}^{\langle\eta_j,\lambda\rangle-1}
                 (\eta_j+c\hbar)}
           {e(\BS_{t^\lambda})}.
\end{equation}
Moreover, the same formulas hold in the presense of a flavour symmetry $\xi$, where each $\eta_j$ is understood as a character of $\hG$.
\end{exmp}
\begin{proof}
Let $[t^\lambda]\in \Gr_T$, by \Cref{CompareIritani2}, the associated Seidel space is given by
\[ \sE_{t^\lambda}(\bN)\cong \bigoplus_{j=1}^n \oh_{\PP^1}(-\langle\eta_j,\lambda\rangle).\]
It is straightforward to verify that the only contributions come from $2$-points invariants. The associated moduli space of sections, with marked points on the fibres over $0$ and $\infty$, is
\[
\cM_{t^\lambda}(\bN)_0= H^0\left(\PP^1,\sE_{t^\lambda}(\bN)\right)=\bigoplus_{j}\bigoplus_{c=\langle\eta_j,\lambda\rangle}^{0} \C_{\eta_j+c\hbar},
\]
and its virtual fundamental class is
\[
[\cM_{t^\lambda}(\bN)_0]^\vir = e(H^1(\PP^1,\sE_{t^\lambda}(\bN))\cap  [\cM_{t^\lambda}(\bN)_0]=\prod_{j}\prod_{c=1}^{\langle\eta_j,\lambda\rangle-1}(\eta_j+c\hbar)\cap [\cM_{t^\lambda}(\bN)_0].
\]
By the proof of \Cref{propermoduli} in the special case $X=\bN$, we have
\[\cZ_{t^\lambda}(\bN)_0\cong \bigoplus_{j:\langle\eta_j,\lambda\rangle\leq0} \C_{\eta_j}.\]
The inclusion $\cZ_{t^\lambda}(\bN)_0\hookrightarrow\cM_{t^\lambda}(\bN)_0$ is a regular embedding. Consequently, 
\[
[\cZ_{t^\lambda}(\bN)_0]^\vir=\prod_{j}\prod_{c=1}^{\langle\eta_j,\lambda\rangle-1}(\eta_j+c\hbar)\cap [\cZ_{t^\lambda}(\bN)_0].
\]
Finally, the evaluation map $\ev_\infty:\cZ_{t^\lambda}(\bN)\to \bN$ is simply the inclusion. Thus, we obtain
\[
\bbS_{T,\bN}(e(\BS_{t^\lambda})\cap [t^\lambda])(1)=q^{\lambda}\prod_{j}\prod_{c=0}^{\langle\eta_j,\lambda\rangle-1}(\eta_j+c\hbar) \in H_T^\bullet(\bN)[q_T],
\]
as desired. By \eqref{Nweights}, we obtain \eqref{abelianformula}. 
\end{proof}

The following example demonstrates the existence of non-equivariant limits as in \Cref{Non-equivariant}. As we will see, the introduction of $G$-equivariant Novikov variables is essential for the existence of non-equivariant limits.

\begin{exmp}\label{Non-equivariantExample}
Let $G = \operatorname{GL}_2$ act on $\PP^1$ by fractional linear transformations, and consider the induced action on $X = T^*\PP^1$. Let $T = (\C^\times)^2 \subset \operatorname{GL}(2,\C)$ denote the subgroup of diagonal matrices. The cohomology of $T^*\PP^1$ admits the following presentation
\[
H^\bullet_{(\C^\times)^2 \times \C^\times_\hbar}(T^*\PP^1) = \frac{\C[x, a_1, a_2, \hbar]}{\langle(x + a_1)(x + a_2)\rangle},
\]
where $x$ is the pullback of $c_1(\mathcal O_{\PP^1}(1))$ along $T^*\PP^1\to \PP^1$, and $a_i$ are the equivariant parameters of $T$.

Let us denote $q_{T,\lambda}^0$ (resp. $q_{T,\lambda}^\infty$) to be the $T$-equivariant Novikov variables associated to the constant section at $0$ (resp. $\infty$) of the Seidel space $\sE_{t^\lambda}(X)$. One can compute that
\begin{align*}
\bbS_{{(1,0)}}(1) &= \frac{a_1 - a_2 + \hbar}{(a_1 - a_2)^2} \left(q^0_{T,(1,0)}(x+a_2) - q^\infty_{T,(1,0)}(x+a_1)\right),  \\
\bbS_{{(0,1)}}(1) &= \frac{a_2 - a_1 + \hbar}{(a_1 - a_2)^2} \left(q_{T,(0,1)}^\infty (x+a_1) - q^0_{T,(0,1)}(x+a_2)\right).
\end{align*}

Notice that there exists a $G$-equivariant proper morphism $T^*\PP^1 \to \mathfrak{gl}_2$, where $\mathfrak{gl}_2$ is the adjoint representation of $\mathrm{GL}_2$. For simplicity, let us denote the fixed point classes $z_1=[t^{(1,0)}]$ and $z_2=[t^{(0,1)}]$. In this case, a class $p(a_1, a_2)z_1 \in \sA^\hbar_{(\Cx)^2, \mathfrak{gl}_2}$ if and only if $p$ is divisible by $(a_1 - a_2)^2$. On the other hand, we find
\begin{align*}
\bbS^{\mathrm{noneq}}_{(\Cx)^2, \mathfrak{gl}_2, T^*\PP^1}((a_1 - a_2)^2 z_1)(1) &= (q^0_{T,(1,0)} - q^\infty_{T,(1,0)})\hbar x, \\
\bbS^{\mathrm{noneq}}_{(\Cx)^2, \mathfrak{gl}_2, T^*\PP^1}((a_1 - a_2)^2 z_2)(1) &= (q^\infty_{T,(0,1)} - q^0_{T,(0,1)})\hbar x.
\end{align*}

It is more interesting to consider the non-abelian shift operators. Consider the class
\[
e(\BS) \cap [C_{\leq (1,0)}] = (a_1 - a_2)\left(z_1-z_2\right) \in \sA^\hbar_{\operatorname{GL_2},\mathfrak{gl}_2}.
\]
Let $\eta:\C[q_T]\to\C[q_G]$ be the base change map, then the relevant $G$-equivariant Novikov variables are identified as follows
\[q_G \coloneq \eta(q_{T,(1,0)}^0)=\eta(q_{T,(0,1)}^\infty),\quad q_{G}'\coloneq \eta(q^\infty_{T,(0,1)})=\eta(q^0_{T,(0,1)}).\]

Then, we find
\[
\bbS_{\operatorname{GL_2},\mathfrak{gl}_2,T^*\PP^1}(e(\BS)\cap [C_{\leq(1,0)}])(1)= (q_G-q_G')(2x+a_1+a_2) - (q_G+q_G')\hbar.
\]
Note that $2x+a_1+a_2$ is the negative of equivariant Euler class of $T^*\PP^1$. So,
\[
\bbS^{\mathrm{noneq}}_{\operatorname{GL}_2, \mathfrak{gl}_2, T^*\PP^1}(e(\BS) \cap [C_{\leq (1,0)}])(1) = -(q_G - q_G')e(T^*\PP^1) -(q_G+q_G')\hbar.
\]
In particular, we have 
\[
\Psi^{\mathrm{noneq}}_{\operatorname{GL}_2, \mathfrak{gl}_2, T^*\PP^1}(e(\BS) \cap [C_{\leq (1,0)}] ) = -(q_G - q_G')e(T^*\PP^1).
\]
Note that the expression admits a non-equivariant limit only after identifying the $q_T$ variables via $\eta$. Without this identification, the cancellation needed for existence of the limit does not occur. \qed
\end{exmp}

\begin{exmp}
It is known that for adjoint matter $\mathfrak{g}$, the Coulomb branch $\sA_{G,\mathfrak{g}}$ is isomorphic to $\C[T^* \widecheck{T}]^W$ (see~\cite[Section~6(vi)]{BFN}). Continuing from \Cref{Non-equivariantExample}, we may compute explicitly the Seidel homomorphism
\[
\Psi_{\operatorname{GL}_2, \mathfrak{gl}_2, T^*\PP^1} \colon \C[T^*(\C^\times)^2]^{\Z_2} \longrightarrow QH_{\operatorname{GL}_2}^\bullet(T^*\PP^1)[[q_G]].
\]
It sends
\begin{alignat*}{3}
    z_1 z_2 \ &\mapsto\   q_G q_G', \qquad &
    z_1^{-1} z_2^{-1} \ &\mapsto\   q_G^{-1} q_G'^{-1}, \\
    z_1 + z_2 \ &\mapsto\   -q_G - q_G', \qquad &
    (a_1 - a_2)(z_1 - z_2) \ &\mapsto\   -(q_G - q_G') \, e^{\operatorname{GL}_2}(T^*\PP^1). \qedhere
\end{alignat*}
\qed
\end{exmp}

\begin{exmp}
    Let $G$ act on $G/B$ and consider the induced action on $T^*G/B$. There is a $G$-equivariant proper morphism $T^*G/B\to \mathfrak g$, where $\mathfrak g$ is the adjoint representation of $G$. Then we have
    \[H^{\mathrm{ord},G}_2(T^*G/B;\Z)\cong H_2^{\mathrm{ord},B}(\pt;\Z)\cong \Lambda.\]
    We denote by $q^\lambda$ the Novikov parameter corresponding to $\lambda\in\Lambda$.
    Suppose that $a \in H_T^\bullet(\pt;\Z)$ is the Euler class of some $B$-representation $V_a$. Let $\pi:T^*G/B\to G/B$. Let $e_a$ be the equivariant Euler class of the bundle
    \[\pi^* (G\times^B V_{a})\to T^*G/B.\]
    Then the Seidel homomorphism
    \begin{equation*}
    \Psi_{G,\mathfrak g,T^*G/B}: \C[T^*\widecheck T]^W\longrightarrow QH_G^\bullet(T^*G/B)[[q_G]]
    \end{equation*}
    admits the following description. For an element of the form $at^\lambda\in \C[T^*\widecheck T]$, where $a\in H_T^\bullet(\pt;\Z)$ is the Euler class of $V_a$, we have
    \[ \Psi_{G,\mathfrak g, T^*G/B}\left(\sum_{\sigma\in W}\sigma(a)z^{\sigma(\lambda)}\right)=\sum_{\sigma\in W}e_{\sigma(a)}q^{\sigma(\lambda)}, \]
    where the sum is taken over the Weyl group elements. In the subsequent work \cite{Cotangentofflag}, we compute the shift operators for $T^*G/B$ in the presence of the flavour symmetry of the conical action along cotangent fibre. \qed
\end{exmp}

\subsection{Peterson isomorphism for reductive groups}\label{Peterson}

In this subsection, we set $\tau=0$ and consider only the small quantum cohomology.

We compute $\bbS_{G,X}([C_{\leq \lambda}])(1)$, and in particular $\Psi_{G,X}([C_{\leq \lambda}])$ when $X$ is a partial flag variety. The main result is \Cref{PetersonIso}, which generalizes the Peterson isomorphism for simply connected groups~\cite{peterson,Lam-Shimozono,Chow}. It also offers compelling evidence that working with equivariant Novikov variables is the natural setting for studying quantum cohomology.

We begin by introducing some notation. Let $R$ be the set of roots of $G$, and let $R^+ \subset R$ be the set of positive roots. Let $P \supset B$ be a standard parabolic subgroup of $G$, and let $R_P \subset R$ be the set of roots of the Levi subgroup of $P$, so that $R_P^+ = R_P \cap R^+$. In particular, we have the decomposition
\[
\operatorname{Lie} P = \operatorname{Lie} B \oplus \bigoplus_{\alpha \in R_P^+} \mathfrak{g}_{-\alpha}.
\]
We denote the Weyl group of the Levi subgroup of $P$ by $W_P$.

Let $X = G/P$ and $v \in W$. Define
\[
    \sigma(v) = \overline{BvP} \subset G/P, \qquad
    \sigma^-(v) = \overline{B^-vP} \subset G/P.
\]

We write $\lambda = w_\lambda(\lambda^-)$, where $\lambda^-$ is antidominant and $w_\lambda$ is the longest-length element in the coset
$w_\lambda \operatorname{Stab}_W(\lambda^-)$.

\begin{df}
   We say that $\lambda \in \Lambda$ is \emph{$P$-allowed} if, for every $\alpha \in R_P^+$,
\[
\langle \alpha, \lambda^- \rangle = 
\begin{cases}
    0 & \text{if } w_\lambda(\alpha) < 0, \\
    -1 & \text{if } w_\lambda(\alpha) > 0.
\end{cases}
\]
\end{df}
Let $EP$ be a contractible space with a free $P$-action, and let $BP=EP/P$. The long exact sequence of homotopy groups associated to the fibration $P\to EP\to BP$ shows that $\pi_1(BP)=1$ and $\pi_2(BP)\cong \pi_1(P)$. In particular, the Hurewicz theorem in algebraic topology implies that
\[
H_2^{\mathrm{ord},P}(\pt;\Z)=H_2^{\mathrm{ord}}(BP;\Z)\cong\pi_1(P).
\]
We identify $\pi_1(T)$ with $\Lambda$, and for $\lambda\in\Lambda$, we denote the image of $\lambda$ under the natural map $\pi_1(T)\to \pi_1(P)$ by the same symbol.

\begin{thm}\label{PetersonIso}
\[
\bbS_{G,G/P}([C_{\leq \lambda}])(1) =
\begin{cases}
    q^{\lambda^-} \sigma(w_\lambda) & \text{if } \lambda \text{ is } P\text{-allowed}, \\
    0 & \text{otherwise}.
\end{cases}
\]
In particular, the homomorphism
\[
\Psi_{G,G/B}\colon \sA_G\to (H_G^\bullet(G/B)[q],\star)
\]
becomes an isomorphism if we localize $\sA_G$ by the classes $[C_{\leq w_0(\lambda^-)}]$ for all $\lambda^-\in \Lambda^-$.
\end{thm}
We write $e_\lambda$ for the Euler class of $T_{t^\lambda}C_\lambda$. We have
\[
    e_\lambda=\prod_{\alpha\in R^+} e_{\lambda,\alpha},
\]
where
\[
    e_{\lambda,\alpha}=
    \begin{cases}
    \displaystyle\prod_{c=0}^{\langle \alpha,\lambda\rangle-1}(\alpha+c\hbar), & \text{if }\langle \alpha,\lambda\rangle\geq 0,\\[0.8ex]
    \displaystyle\prod_{c=1}^{-\langle \alpha,\lambda\rangle-1}(-\alpha+c\hbar), & \text{if } \langle \alpha,\lambda\rangle<0.
    \end{cases}
\]

\begin{lem}\label{lem:PetersonIso_limit}
For any $\lambda\in \Lambda$ and $u\in W$, the limit
\[
\lim_{\hbar\to \infty} e_\lambda^{-1}\langle \bbS_{\lambda}(1),\sigma^-(u)\rangle
\]
exists and is a $\C[q]$-linear combination of homogeneous rational functions of nonpositive degree. Moreover, its degree-zero term is $q^{\lambda^-}$ if $\lambda$ is $P$-allowed and $u\in w_\lambda W_P$, and is zero otherwise.
\end{lem}
\begin{proof}
We use the formula $\bbS_\lambda(1)=\mathbb{M}S_\lambda\mathbb{M}^{-1}$ from \Cref{intertwiningFundamentalsolution}. Since $X=G/P$ is convex, both $\mathbb{M}$ and $\tw_\lambda(\mathbb{M})^{-1}$ are $\operatorname{id}+O(1/\hbar)$. Hence it suffices to prove the lemma with $\bbS_\lambda$ replaced by $S_\lambda$.

\Cref{Giventalshift} implies that
\[
e_\lambda^{-1}\langle S_{\lambda}(1),\sigma^-(u)\rangle
=
\sum_{v}q^{\overline{\beta_{\lambda,v}}}\cdot e_\lambda^{-1}
\frac{e(\Delta_{\lambda}(T_vX))\sigma^-(u)|_v}{e(T_vX)},
\]
where we recall that
\[
e(T_vX)=\prod_{\alpha\in v(R^+\setminus R^+_P)}(-\alpha),
\]
\[
e(\Delta_{\lambda}(T_vX))
=
\prod_{\alpha\in v(R^+\setminus R^+_P)}\Delta_{\lambda,v,\alpha},
\]
and
\[
\Delta_{\lambda,v,\alpha}
=
\begin{cases}
\displaystyle\prod_{c=0}^{-\langle \alpha,\lambda\rangle-1}(-\alpha+c\hbar),
& \text{if }\langle \alpha,\lambda\rangle\leq 0,\\[0.8ex]
\displaystyle\prod_{c=1}^{\langle \alpha,\lambda\rangle}(-\alpha+c\hbar)^{-1},
& \text{if } \langle \alpha,\lambda\rangle>0.
\end{cases}
\]
Note that the $\hbar\to \infty$ limit of
\[
A_{\lambda,u,v}\coloneqq
e_\lambda^{-1}
\frac{e(\Delta_{\lambda}(T_vX))\sigma^-(u)|_v}{e(T_vX)}.
\]
exists for all $u$ and $v$, and is nonzero only if $\langle \alpha,v^{-1}(\lambda)\rangle\leq 0$ for every $\alpha\in R^+\setminus R_P^+$, which is equivalent to $v\in w_\lambda W_P$.

It is immediate that
\[
A_{\lambda,w_\lambda,w_\lambda}
=
\prod_{\alpha\in R^+\cap w_\lambda(R_P)}e_{\lambda,\alpha}^{-1},
\]
which is $1$ if $\lambda$ is $P$-allowed, and is a rational function of negative degree otherwise.

Finally, the lemma follows from the fact that $\deg A_{\lambda,u,w_\lambda}<\deg A_{\lambda,w_\lambda,w_\lambda}$ for $u\notin w_\lambda W_P$, together with \Cref{Lemmacurveclass} below.
\end{proof}
\begin{lem}\label{Lemmacurveclass}
    $\overline{\beta_{\lambda,v}}=v^{-1}(\lambda)\in \pi_1(P)$.
\end{lem}
\begin{proof}[Proof of \Cref{Lemmacurveclass}]
Consider the diagram
    \begin{equation*}
        \begin{tikzcd}
            \pi_2(\sE_{\leq\lambda}\times (G/P))\ar[r]&\pi_2(\sE_{\leq\lambda}(G/P))\ar[r]\ar[d]&\pi_1(P)\ar[r,"1"]\ar[d,equal]&\pi_1(\sE_{\leq\lambda}\times (G/P))\\
            \pi_2(EP)\ar[r]&\pi_2(BP)\ar[r,"\cong"]&\pi_1(P)\ar[r]&\pi_1(EP),
        \end{tikzcd}
    \end{equation*}
    where the rows are the long exact sequence associated to the fibrations $P\to \sE_{\leq\lambda}\to \sE_{\leq\lambda}(G/P)$ and $P\to EP\to BP$ respectively. The vertical homomorphisms are induced by the classifying map $\sE_{\leq\lambda}(G/P)\to BP$ corresponding to the principal $P$-bundle $\sE_{\leq\lambda}\to \sE_{\leq\lambda}(G/P)$.
    By a diagram tracing, it remains to show that the boundary homomorphism $\pi_2(\sE_{\leq\lambda}(G/P))\to \pi_1(P)$ sends $\beta_{\lambda,v}$ to $v^{-1}(\lambda)$.

    Let $\oh_{\PP^1}(-1)^\times$ be the $\Cx$ bundle associated to the tautological bundle of $\PP^1$. By checking the transition functions, it is easy to see that the restriction of $\sE_{\leq\lambda}$ to $\beta_{\lambda,v}$ is isomorphic to $\oh_{\PP^1}(-1)^\times\times_{v^{-1}(\lambda)} P$.

    This implies that there is a commutative diagram
    \begin{equation*}
        \begin{tikzcd}
            \pi_2(\PP^1)\ar[r]\ar[d,equal]&\pi_1(\Cx)\ar[d]\\
            \pi_2(\PP^1)\ar[r]\ar[d]&\pi_1(P)\ar[d,equal]\\
            \pi_2(\sE_{\leq\lambda}(G/P))\ar[r]&\pi_1(P),
        \end{tikzcd}
    \end{equation*}
where the horizontal maps are boundary homomorphism associated to the fibrations $\Cx\to \oh_{\PP^1}^\times\to\PP^1$, $\Cx\to \oh_{\PP^1}(-1)^\times\times_{v^{-1}(\lambda)}P \to\PP^1$, and $P\to \sE_{\leq\lambda}\to \sE_{\leq\lambda}(G/P)$ respectively. Note that the left lower map sends $[\PP^1]$ to $\beta_{\lambda,v}$ and the right upper map sends the generator $1\in \Z\cong \pi_1(\Cx)$ to $v^{-1}(\lambda)_P$. Now the lemma follows from the convention that the upper boundary homomorphism sends $[\PP^1]$ to the generator $1\in \pi_1(\Cx)$.
\end{proof}

\begin{proof}[Proof of \Cref{PetersonIso}]
We will show that
\begin{equation}\label{eq:PetersonIso}
    (\bbS_{G,X}([C_{\leq \lambda}])(1),\sigma^-(u))=
    \begin{cases}
    q^{\lambda^{-}} & \text{if } \lambda \text{ is } P\text{-allowed and } u\in w_\lambda W_P,\\
    0 & \text{otherwise}.
\end{cases}
\end{equation}
By \Cref{IntroThm1}, the left-hand side of \eqref{eq:PetersonIso} is a polynomial in the equivariant parameters, so it suffices to take the $\hbar\to \infty$ limit of the left-hand side. If we write
\[
[C_{\leq \lambda}]=\sum_{\mu\leq \lambda}a_\mu^{-1} [t^\mu],
\qquad
a_\mu\in \C[\lt][\hbar],
\]
then the left-hand side of \eqref{eq:PetersonIso} can be written as
\[
\sum_{\mu\leq \lambda}a_\mu^{-1}(\bbS_{\mu}(1),\sigma^-(u)).
\]
Note that $a_\mu$ is strictly divisible by $e_\mu$ for $\mu\lneq \lambda$. Therefore \Cref{lem:PetersonIso_limit}, together with the fact that $(\bbS_{G,X}([C_{\leq \lambda}])(1),\sigma^-(u))$ is a polynomial in the equivariant parameters, implies that
\[
\begin{aligned}
\lim_{\hbar\to \infty}(\bbS_{G,X}([C_{\leq \lambda}])(1),\sigma^-(u))
=\lim_{\hbar\to \infty}a_\lambda^{-1}(\bbS_{\lambda}(1),\sigma^-(u)) =\lim_{\hbar\to \infty} e_\lambda^{-1}(\bbS_{\lambda}(1),\sigma^-(u)).
\end{aligned}
\]
This is equal to $q^{\lambda^-}$ if $\lambda$ is $P$-allowed and $u\in w_\lambda W_P$, and is equal to $0$ otherwise by \Cref{lem:PetersonIso_limit}. This finishes the proof.
\end{proof}

\subsection{Peterson isomorphism with matters}\label{SectionPeterson}
In the case when $\bN=0$ and $X=G/P$ is a partial flag variety, by a theorem of \cite{Lam-Shimozono} and also \cite{Chow}, the Novikov variables $q^\beta$ lie in the image of the homomorphism $H_\bullet^{T_\oh}(\Gr_G)\to QH^\bullet_{T}(G/P)[q_G]$. In particular, this implies that for different $G/P$ and $G/P'$, or for $G/P$ equipped with different symplectic forms, the corresponding Lagrangians $\mathbb L_{G,\bN}(G/P)$ and $\mathbb L_{G,\bN}(G/P')$ have empty intersections. 

The analogous statement for general $\bN$ is not true. The Novikov variables $q^{\overline\beta}$ may not lie in the image of $e(\BS_\bN)\cap H^{T_\oh}_\bullet(\Gr_G)\to QH_T^\bullet(G/P\times\bN)[q_G]$. Nevertheless, \Cref{MatterPeterson} below provides partial results, identifying a subset of quantum Schubert classes which do lie in the image of the Seidel homomorphism.

\begin{prop}\label{NNdual}
    Let $Y$ be a smooth projective $G$-variety, let $\bN$ be a $G$-representation and $\bN^\vee$ be its dual representation. For any $\lambda\in \Lambda$, we have
    \begin{equation}\label{NNdualeq}
        \Psi_{G,\bN,Y\times\bN}(e({\BS}_{\bN})\cap C_{\leq\lambda})= \Psi_{G,Y}(e({\BS}_{\bN^\vee})\cap C_{\leq\lambda}),
    \end{equation}
    under the canonical isomorphism $H_G^\bullet(Y)\cong H_{G}^\bullet(Y\times\bN)$.
\end{prop}
\begin{proof}
    By \Cref{TGcompatible}, it suffices to show the identity
    \[c_{\lambda,\mu}^\bN\bbS_{X=Y\times\bN,\mu}(1)=c_{\lambda,\mu}^{\bN^\vee}\bbS_{X=Y,\mu}(1),\]
    where $c_{\lambda,\mu}^\bN$ and $c_{\lambda,\mu}^{\bN^\vee}$ are the coefficients defined in \Cref{localizeSlambda}. By \Cref{kunneth}, we reduce to the case \( Y = \pt \), which follows from the calculation in \Cref{abeliancase}.
\end{proof}

Now we set $\tau=0$ and only consider small quantum cohomology.
\begin{cor}\label{Matterpeterson2}
    The Seidel homomorphism $\Psi_{G,\bN,G/B\times\bN}$ induces a birational morphism 
    \[\spec (H_G^\bullet(G/B\times\bN)[q],\star)\to \spec \sA_{G,\bN}.
    \]
\end{cor}
\begin{proof}
    This follows from \eqref{NNdualeq} and the $\bN=0$ case (\Cref{PetersonIso}).
\end{proof}

The following proposition is a generalization of \Cref{Peterson}. The assumption is satisfied if $G=G'\times \Cxd$, where the $\Cxd$-factor acts on $\bN$ by scaling.

\begin{cor}\label{MatterPeterson}
Let $\rho:\C^\times \subset G$ be the inclusion of a central subgroup, and let $\bN$ be a representation whose $\C^\times$-weights are all positive (resp.\ negative). Let $X = G/P \times \bN$, and fix $\lambda \in \Lambda$. Then for sufficiently large $n > 0$, we have
\begin{align*}
\Psi_{G ,\, \bN,\, X}\bigl(e(\widetilde{\BS}_{\leq \lambda - n\rho}) \cap [\widetilde{C}_{\lambda - n\rho}]\bigr) 
&= q^{(\lambda-n\rho)_P^-} \sigma(w_{\lambda}), \\
\text{resp.} \quad 
\Psi_{G,\, \bN,\, X}\bigl(e(\widetilde{\BS}_{\leq \lambda + n\rho}) \cap [\widetilde{C}_{\lambda + n\rho}]\bigr)
&= q^{(\lambda+n\rho)_P^-} \sigma(w_{\lambda}),
\end{align*}
under the pullback isomorphism $H_G^\bullet(G/P)\cong H_G^\bullet(G/P \times \bN)$.
\end{cor}

\begin{proof}
    This follows from \Cref{PetersonIso} and \Cref{NNdual} applied to $Y=G/P$, along with the observation that $\BS_{\bN^\vee,\leq \lambda-n\rho}=0$ (resp. $\BS_{\bN^\vee,\leq \lambda+n\rho}=0$) for $n$ sufficiently large. 
\end{proof}

\begin{rem}
    For each specialization $\varsigma \colon \mathbb{C}[q_G] \to \mathbb{C}$, we obtain a Lagrangian subvariety
\[
\mathbb{L}_{\varsigma} = \mathrm{Supp}\left(QH_G^\bullet(G/B \times \bN) \otimes_{\mathbb{C}[q_G]} \mathbb{C}\right).
\]
For any two distinct specializations $\varsigma_1$ and $\varsigma_2$, it is easy to see that the corresponding Lagrangians $\mathbb{L}_{\varsigma_1}$ and $\mathbb{L}_{\varsigma_2}$ are disjoint. Therefore, the previous corollary may be viewed as a generalization of Teleman's result in the case $\bN = 0$ \cite[Theorem 6.8]{Tel}.
\end{rem}

\subsection{A new characterization of the Coulomb branch}\label{sectionnewcharacterization}
The goal of this subsection is to prove \Cref{Introlargestsubalg}. In this subsection, we only need to consider the small quantum cohomology.

Let $\xi$ be a flavour symmetry of $\bN$ (see \Cref{section:flavoursymmetries}), i.e., a short exact sequence
\begin{equation}\label{eq:flavour_symmetry}
    \xi: 1\to G\to \hG\to T_\xi\to 1
\end{equation}
such that $T_\xi$ is a torus and $\bN$ lifts to an $\hG$-representation. We write $\hT$ for the maximal torus of $\hG$ containing $T$. We say that $\xi'$ contains $\xi$ if $T_{\xi'}\supset T_{\xi}$. Recall that we write
\[
\oh_{T^*\check T}^{\hbar,\xi}
=
\C[\hat\lt]\otimes_{\C[\lt]}\oh_{T^*\check T}^{\hbar}.
\]
As a $\C[\hat\lt]$-module, it is equal to $\C[\hat \lt\times \check T]$, but it has a twisted product with
\[
at^\lambda-t^\lambda a=a(\lambda)\hbar t^\lambda,
\qquad
a\in \hat\lt^*,\ \lambda\in\Lambda.
\]
\begin{df}\label{df:lift}
An element $\tilde a\in (\oh_{T^*\check T}^{\hbar,\xi})_\loc$ is called a \emph{lift} of $a \in (\oh_{T^*\check T}^\hbar)_\loc$ if $\tilde a$ can be represented by $a'/b$, where $a'\in \oh_{T^*\check T}^{\hbar,\xi}$ and $b\in H^\bullet_{\hT\times\Cxh}$ is homogeneous and not divisible by $\xi$, such that
\[
\tilde a|_{\xi=0}=a.
\]
\end{df}
\begin{thm}[=\Cref{Introlargestsubalg}]\label{Thmlargestsubalg}
For any $a \in (\oh_{T^*\check T}^\hbar)_\loc$, the following statements are equivalent:
\begin{enumerate}
    \item[(a)] $a\in \sA^\hbar_{G,\bN}$.
    \item[(b)] For any flavour symmetry $\xi$, there exists a lift $\tilde a \in (\oh_{T^*\check T}^{\hbar,\xi})_\loc$ of $a$ such that for any smooth semiprojective variety $X$ equipped with a $G\times \Cx$-equivariant proper morphism $f: X \to \bN$, the deformed shift operator
\[
\bbS_{G,\bN, X}^\xi(\widetilde a)\in \operatorname{End}^\bullet(H_{\hG\times \Cx}^\bullet(X)[\hbar])_\loc[[q,\tau]]
\]
is regular.
    \item[(c)] For any flavour symmetry $\xi$, there exists a lift $\tilde a \in (\oh_{T^*\check T}^{\hbar,\xi})_\loc$ of $a$ such that the deformed shift operator $\bbS_{G,\bN, X}^\xi(\widetilde a)$ is regular for $X=G/B\times \bN$ and $X=\bN$.
\end{enumerate}
\end{thm}

By \Cref{IntroThm2}, (a) implies (b). Clearly, (b) implies (c). It remains to prove that (c) implies (a). We first prove \Cref{Thmlargestsubalg} in the special case $\bN=0$.
\begin{prop}\label{largestsubalgcompact} $\sA^\hbar_G=H^{G_\oh\rtimes \Cx_\hbar}_\bullet(\Gr_G)$ is equal to the subset of $(\oh_{T^*\check T}^\hbar)_\loc$ consisting of elements $x$ for which $\bbS_{G,G/B}(x)(1)\in H_G^{\bullet}(G/B)[\hbar][[q]]$.
\end{prop}
\begin{proof}
Let $\sB$ be the subset of $(\oh_{T^*\check T}^\hbar)_\loc$ consisting of elements $x$ for which $\bbS_{G,G/B}(x)(1)\in H_G^{\bullet}(G/B)[\hbar][[q]]$. We have $\sA^\hbar_G\subset \sB$ by \Cref{IntroThm2}. To show the other inclusion, let $x\in \sB$. We may write 
\begin{equation*}
x=\sum_{\lambda\in\Lambda} x_\lambda [C_{\leq\lambda}],\quad x_\lambda\in \C(\lt)(\hbar).
\end{equation*}
Setting $\tau=0$. By \Cref{PetersonIso}, we have
\begin{equation*}
\bbS_{G,G/B}(x)(1)=\sum_{\lambda\in\Lambda} x_\lambda q^{\lambda^-}\sigma(w_\lambda)\in H_G^\bullet(G/B)[\hbar][q].
\end{equation*}
As $\{q^{\lambda}\sigma(w)\}_{\lambda\in\Lambda,w\in W}$ is a $\C[\lt][\hbar]$-linear basis of $H_{T}^\bullet(G/B)[\hbar][q]$, it follows that each $x_\lambda\in \C[\lt][\hbar]$. To show $x\in \sA_G^\hbar$, it remains to show that $x$ is invariant under the Weyl group action. Since $\bbS_{G,G/B}(x)(1)$ is $W$-invariant, this follows from the fact that $\bbS_{G,G/B}(-)(1)$ is injective and $W$-equivariant.
\end{proof}

Recall $\sA^\hbar_{G,\bN}$ carries a natural $\sA^\hbar_G$-comodule structure via (\Cref{coproduct section}):
\begin{equation*}
    \Delta_{*}:\sA^\hbar_{G,\bN}\to\sA^\hbar_{G}\otimes_{\C[\lt][\hbar]}\sA^\hbar_{G,\bN}.
\end{equation*}
A $\C[\lt]^W[\hbar]$-submodule $\sA'\subset \sA^\hbar_G$ is called a (left) $\sA^\hbar_G$-\emph{subcomodule} if it is stable under $\Delta_*$, i.e.,
\[\sA'\subset \sA^\hbar_G\xrightarrow{\Delta_*}\sA^\hbar_G\otimes_{H_G^\bullet(\pt)} \sA_G\]
has image contained in $\sA^\hbar_G\otimes_{H_G^\bullet(\pt)} \sA'$. The same notion applies to subcomodules of $\sA^{\hbar,\xi}_G$ for any flavour symmetry $\xi$, as well as their quantized counterparts.
\begin{lem}\label{largestsubcomod}
    Let $\hG=G\times\Cxd$, where $T_\xi=\Cxd$ acts on $\bN$ by the scaling action. Let $X=\bN$. Suppose $\sA'\subset \sA^{\hbar,\xi}_{G}$ is an $\sA^{\hbar,\xi}_{G}$-subcomodule such that
    \begin{equation*}
       \bbS_{G,X=\bN}(\sA')(1) \subset H_{\hG}^\bullet(\bN)[\hbar][q].
    \end{equation*}
    Then $\sA'\subset\sA^{\hbar,\xi}_{G,\bN}$.
\end{lem}

\begin{proof}
It suffices to prove the lemma after setting $q=1,\tau=0$. Replacing $\sA'$ by $\sA^{\hbar,\xi}_{G,\bN}+\sA'$, we may assume $\sA^{\hbar,\xi}_{G,\bN} \subset \sA'$. We write $\bbS$ for $\bbS_{G,X=\bN}$.

Consider the composition
\begin{equation*}
\sA^{\hbar,\xi}_{G} \xrightarrow{\Delta_*} \sA^{\hbar,\xi}_{G} \otimes_{H_{\hG}^\bullet(\pt)[\hbar]} \sA^{\hbar,\xi}_{G}
\xrightarrow{\id \otimes \bbS(-)(1)}
\sA^{\hbar,\xi}_{G} \otimes_{H^\bullet_{\hG}(\pt)[\hbar]} H^\bullet_{\hG}(\bN)_\loc[\hbar] \cong (\oh_{T^*\check T}^{\hbar,\xi})_\loc.
\end{equation*}
For $a\in \sA'$, we have 
\[
(\id \otimes \bbS)(\Delta_*(a))(1) \in \sA^{\hbar,\xi}_{G}.
\]
We may write 
\begin{equation*}
   a = \sum_{i=1}^k a_i [t^{\lambda_i}], 
   \qquad 
   a_i \in \C(\lt)(\hbar), \ a_i \neq 0.
\end{equation*}
Hence the support of $a$ (viewed as an element of $H^{\hT\times \Cx_\hbar}_{\bullet}(\Gr_{G})$) is
\begin{equation}\label{support of a'}
   \bigcup_{i=1}^k C_{\leq \lambda_i}.
\end{equation}
Assume that $\lambda_1$ is maximal among the $\lambda_i$. Then we have a short exact sequence
\begin{equation*}
0 \to H^{\hT\times \Cx_\hbar}_\bullet(C_{<\lambda_1})
   + \sum_{i=2}^k H^{\hT\times \Cx_\hbar}_\bullet(C_{\leq \lambda_i})
 \to \sum_{i=1}^k H^{\hT\times \Cx_\hbar}_\bullet(C_{\leq \lambda_i})
 \to H^{\hT\times \Cx_\hbar}_\bullet(C_{\lambda_1}) \to 0.
\end{equation*}
Let $e \in H_{\hT\times\Cxh}^\bullet(\pt)$ be the Euler class of the tangent space of cell $C_{\lambda_1}$ at $[t^{\lambda_1}]$. Then
\begin{equation*}
   H_\bullet^{\hT\times \Cxh}(C_{\lambda_1})
   = H_{\hT\times\Cxh}^\bullet(\pt)\cdot \frac{[t^{\lambda_1}]}{e}.
\end{equation*}
On the other hand,
\begin{equation*}
   (\id \otimes \bbS)(\Delta_*(a))(1)
   = \sum_{i=1}^k a_i \bbS_{\lambda_i}(1)[t^{\lambda_i}],
\end{equation*}
whose projection to $H^{\hT\times \Cxh}_\bullet(C_{\lambda_1})$ equals 
\begin{equation}\label{projecta'}
   a_1 \bbS_{\lambda_1}(1)[t^{\lambda_1}]
   =
   a_1 
   e(\BS_{t^{\lambda_1}})^{-1}\prod_j\prod_{c=0}^{\langle\eta_j,\lambda_1\rangle-1}
   (\eta_j+c\hbar)
   [t^{\lambda_1}].
\end{equation}
Here we use \Cref{abeliancase}, writing $\bN = \bigoplus_{j=1}^n \C_{\eta_j}$ where $\eta_j$ are the $\hT$-weights, and
\begin{equation*}
    e(\BS_{t^{\lambda_1}})
    =
    \prod_j\prod_{c=0}^{-\langle\eta_j,\lambda_1\rangle-1}
    (\eta_j+c\hbar).
\end{equation*}
According to the decomposition $\hT = T \times \Cxd$, we can write $\eta_j = \chi_{\mathrm{dil}} + \chi_j$, where $\chi_{\mathrm{dil}}$ is the weight of the standard representation of $\Cxd$, and $\chi_j$ are the $T$-weights. It is then easy to see that the numerator and denominator in \eqref{projecta'} are relatively prime. Hence
\begin{equation*}
   b \coloneq \frac{e a_1}{e(\BS_{t^{\lambda_1}})} \in H^\bullet_{T'\times\Cxh}(\pt).
\end{equation*}

We now replace $a$ by
\begin{equation*}
   a' = a - b e(\BS) \cap [C_{\leq \lambda_1}]
       = a -  a_1 [t^{\lambda_1}] + \text{(terms with } [t^\lambda], \ \lambda < \lambda_1).
\end{equation*}
The support of $a'$ is strictly smaller than that of $a$. Repeating this process for $a'$, we eventually terminate because \eqref{support of a'} is Noetherian. Therefore $a \in \sA^{\hbar,\xi}_{G',\bN}$.
\end{proof}
\begin{cor}\label{corwithdil}
Suppose $T_\xi=\Cxd$, then $\sA^{\hbar,\xi}_{G,\bN}$ is the subset of $(\oh_{T^*\check T}^{\hbar,\xi})_\loc$ consisting of $x$ for which $\bbS^\xi_{G,X}(x)(1)\in H_{\hG}^\bullet(X)[\hbar][[q]]$ for $X=G/B\times\bN$ and $X=\bN$.
\end{cor}
\begin{proof}
    Let $\sB$ denotes the subset of $(\oh_{T^*\check T}^{\hbar,\xi})_\loc$ consisting of $x$ for which $\bbS^\xi_{G,X}(x)(1)\in H_{\hG}^\bullet(X)[\hbar][[q]]$ for $X=G/B\times\bN$ and $X=\bN$. 

    By \Cref{largestsubalgcompact}, we have $\sB\subset \sA^{\hbar,\xi}_{G}$. By \Cref{largestsubcomod}, it suffices to prove that $\sB$ is stable under the comodule action $\Delta_*$.  We first observe that by \Cref{PetersonIso}, there exists a $\C[\lt][\hbar,\xi]$-linear homomorphism
    \begin{equation*}
        r:H_{T}^\bullet(G/B)[\hbar,\xi][q]\to H^{\hT\times \Cxh}_\bullet(\Gr_{G})
    \end{equation*}
    by sending each $q^\mu\sigma(w)$ to $[C_{\leq \lambda}]$ if $\mu=\lambda^-$ and $w=w_\lambda$ for some $\lambda\in \Lambda$, and to zero otherwise. This restricts to a $\C[\lt]^W[\hbar,\xi]$-linear homomorphism
    \begin{equation*}
        \sigma:H_{\hG}^\bullet(G/B)[\hbar][q]\to \sA^{\hbar,\xi}_{G}
    \end{equation*}
    with $\sigma\circ \bbS^\xi_{G,G/B}(-)(1)=\operatorname{id}$.
    
    Let $x\in \sB$, and choose a $\C[\lt]^W[\hbar,\xi]$-linear basis $\{b_i\}$ of $\sA^{\hbar,\xi}_{G}$. We may write
    \begin{equation*}
        \Delta_*(x)=\sum_i b_i\otimes x_i\in \sA^{\hbar,\xi}_{G}\otimes_{\C[\lt]^W[\hbar,\xi]} \sA^{\hbar,\xi}_{G},
    \end{equation*}
    where only finite many terms in the sum is nonzero. It remains to show that $x_i\in \sB$ for each $i$.
    
    Consider $X=G/B\times Y$, where $Y=G/B\times \bN$ or $Y=\bN$. Applying \Cref{kunneth}, we obtain 
    \begin{equation*}
    \bbS^\xi_{G,X}(x)(1)
    =\sum_{i} \bbS^\xi_{G,G/B}(b_i)(1)\otimes \bbS^\xi_{G,Y}(x_i)(1)\in H_{\hG}^\bullet(G/B)[\hbar]\otimes_{\C[\lt]^W[\hbar,\xi]} H_{\hG}^\bullet(Y)[\hbar][[q]]
    \end{equation*}
    Further applying $\sigma\otimes
    \operatorname{id}$, we obtain
    \begin{equation*}
    \sum_{i} b_i\otimes \bbS^\xi_{G,Y}(x_i)(1)\in \sA^{\hbar,\xi}_{G}\otimes_{\C[\lt]^W[\hbar,\xi]} H_{\hG}^\bullet(Y)[\hbar][[q]].
    \end{equation*}
    Since $\{b_i\}$ is a basis of $\sA^{\hbar,\xi}_{G}$, we see that
    \begin{equation*}
        \bbS^\xi_{G,Y}(x_i)(1)\in H_{\hG}^\bullet(Y)[\hbar][[q]]
    \end{equation*}
    for any $\hG$-equivariant proper morphism $Y\to \bN$. This shows that $x_i\in \sB$ and finishes the proof.
\end{proof}

\begin{proof}[Proof of Theorem~\ref{Thmlargestsubalg}]
Let $\sB$ be the subspace of $(\oh_{T^*\check T}^\hbar)_\loc$ which consists of elements $x$ for which conclusion of (c) holds. By \Cref{IntroThm2}, we have $\sA^\hbar_{G,\bN}\subset \sB$. It remains to show the other inclusion $\sB\subset \sA^\hbar_{G,\bN}$.

Let $x\in \sB$. Consider the flavour symmetry $T_\xi=\Cxd$. By definition, there exists some lift $x'\in (\oh_{T^*\check T}^{\hbar,\xi})_\loc$ for which
\begin{equation*}
\bbS_{G,X}^\xi(x')(1)\in H_{\hG}^\bullet(X)[\hbar][[q]
\end{equation*}
for for $X=G/B\times \bN$ and $X=\bN$. By \Cref{corwithdil}, we have $x'\in \sA^{\hbar,\xi}_{G,\bN}$. Since
\begin{equation*}
   \sA_{G,\bN}^{\hbar,\xi}/(\xi)=\sA_{G,\bN}^{\hbar,\xi},
\end{equation*}
we have $x^\hbar\in\sA_{G,\bN}$. This finishes the proof.
\end{proof}

The proof of the following result is completely analogous to that of \Cref{Thmlargestsubalg}.
\begin{cor}\label{Cor:Largestsubalgebra}
Suppose we have a flavour symmetry given by \eqref{eq:flavour_symmetry}. Let $a \in (\oh_{T^*\check T}^{\hbar,\xi})_\loc$, then the following statements are equivalent.
\begin{enumerate}
    \item[(a)] $a\in \sA^{\hbar,\xi}_{G,\bN}$.
    \item[(b)] For any flavour symmetry $\xi'$ containing $\xi$, there exists a lift $\tilde a \in (\oh_{T^*\check T}^{\hbar,\xi'})_\loc$ of $a$ such that for any smooth semiprojective variety $X$, equipped with a $G\times \Cx$-equivariant proper morphism $f: X \to \bN$, the deformed shift operator $\bbS_{G,\bN, X}^{\xi'}(\widetilde a)$ is regular.
    \item[(c)] For any flavour symmetry $\xi'$ containing $\xi$, there exists a lift $\tilde a \in (\oh_{T^*\check T}^{\hbar,\xi'})_\loc$ of $a$ such that the deformed shift operator $\bbS_{G,\bN, X}^{\xi'}(\widetilde a)$ is regular for $X=G/B\times \bN$ and $X=\bN$.
\end{enumerate}
\end{cor}

During the proof of \Cref{Thmlargestsubalg}, we obtained the following corollary, which recovers \cite[Theorem 1]{2drole}.
\begin{cor}\label{Telemangluing}
    Consider the short exact sequence
    \begin{equation*}
        1 \to G \to G'=G \times \Cxd \to \Cxd \to 1.
    \end{equation*}
    Then the Coulomb branch algebra $\sA_{G,\bN}$ is equal to
    \begin{equation*}
        \Big\{\, a \in \sA_{G}^{\mathrm{dil}} \;\Big|\;
        \operatorname{id}\otimes\Psi_{G,X=\bN,\loc}^{\mathrm{dil},q_0}\big(\Delta_*(a)\big)
        \in H^\bullet_{G'}(\pt) \,\Big\}
        \;\otimes_{H^\bullet_{G'}(\pt)} H^\bullet_{G}(\pt),
    \end{equation*}
    where $q_0$ is the specialization $q_G\mapsto 1$ and $\tau\mapsto 0$.

    A similar statement holds for the quantized Coulomb branch $\sA_{G,\bN}^\hbar$.
\end{cor}

We now give another corollary of the proof of \Cref{Thmlargestsubalg}. We introduce the following definition, which appeared in~\cite{3dmirror} and~\cite{functoriality}.

\begin{df}
A $G$-representation $\bN$ is called \emph{gluable} if for all nonzero $T$-weights 
$\eta_1,\eta_2$, the weight $\eta_1$ is not a negative multiple of $\eta_2$.
\end{df}

In particular, the $G \times \Cx_{\mathrm{dil}}$-representation in which 
$\Cx_{\mathrm{dil}}$ acts on $\bN$ by scaling is gluable.

\begin{thm}
Let $\bN$ be a $G$-representation. Then the following diagram commutes:
\begin{equation*}
\begin{tikzcd}
\sA_{G,\bN} \arrow[r, "\Delta_*"] \arrow[d]
  & \sA_{G} \otimes_{H^\bullet_{G}(\pt)} \sA_{G,\bN}
    \arrow[rr, "{\id\otimes \Psi_{G,\bN}} "] 
    && \sA_G \otimes_{H^\bullet_{G}(\pt)} QH^\bullet_{G}(X)[[q_G,\tau]] \arrow[d] \\
\sA_{G} \arrow[r, "\Delta_*"]
  & \sA_{G} \otimes_{H^\bullet_{G}(\pt)} \sA_{G}
    \arrow[rr, "{\id\otimes \Psi_{G,\bN,\loc}}"]
    && \sA_G \otimes_{H^\bullet_{G}(\pt)} QH^\bullet_{G}(X)_\loc [[q_G,\tau]] .
\end{tikzcd}
\end{equation*}
Moreover, if $\bN$ is gluable and $X = \bN$, then the outer square is Cartesian.

A similar statement holds for the quantized Coulomb branch $\sA_{G,\bN}^\hbar$, with $\Psi$ replaced by $\bbS(-\otimes 1)$.
\end{thm}

\begin{proof}
After observing that the numerator and denominator of \eqref{abelianformula} are relatively prime under the gluable assumption, the rest of the proof is similar to that of \Cref{largestsubcomod}.
\end{proof}
\subsection{Properties of the Coulomb branches and their categorifications}\label{subsection:properties_and_categorifications}
In this subsection, we demonstrate how properties of the quantized Coulomb branch algebra $\sA^\hbar_{G,\bN}$ follow from its characterization in \Cref{Thmlargestsubalg}.
\begin{prop}
The following statements are true.
\begin{enumerate}
    \item $\sA^\hbar_{G,\bN}$ is a subalgebra of $(\oh^\hbar_{T^*\check T})_\loc$.

    \item Suppose that $\bN_1$ and $\bN_2$ are $G$-representations, and let $\bN=\bN_1\oplus\bN_2$. Then there is a coproduct
    \[
    \Delta_*\colon
    \sA^\hbar_{G,\bN}
    \to
    \sA^\hbar_{G,\bN_1}
    \otimes_{\C[\lt]^W[\hbar]}
    \sA^\hbar_{G,\bN_2}.
    \]

    \item Suppose that we have a flavour symmetry given by \eqref{eq:flavour_symmetry}. Then there is a deformation
    \[
    \C[\xi]\to \sA^{\hbar,\xi}_{G,\bN}
    \]
    of Coulomb branches.

    \item In addition to the assumptions in (3), suppose that there is a short exact sequence
    \[
    1\to G\to G'\to G_\xi\to 1
    \]
    such that $G'\supset \hG$ and $T_\xi$ is a maximal torus of $G_\xi$. Suppose further that the $\hG$-action on $\bN$ extends to a $G'$-action. Then $W_{G_\xi}$ acts on the deformed Coulomb branch $\sA^{\hbar,\xi}_{G,\bN}$.

    \item The torus $T^!=H^2_G(\pt;\C)/H^2_G(\pt;\Z)$ acts on the Coulomb branch $\sA^\hbar_{G,\bN}$.
\end{enumerate}
\end{prop}
\begin{proof}
Each of these properties of Coulomb branches is well known. The point of the proposition is to relate these properties to the corresponding properties of shift operators. In what follows, we state the properties of shift operators that we need and then give a sketch of the proof.
\begin{enumerate}
    \item We use the fact that $\bbS$ is a homomorphism and that the composition of two regular operators is again regular.
    Let $a,b\in \sA^\hbar_{G,\bN}$. We want to show that $ab\in \sA^\hbar_{G,\bN}$. Let $\xi$ be a flavour symmetry. Suppose that $\tilde a$ and $\tilde b$ are lifts of $a$ and $b$ for which \Cref{Thmlargestsubalg}(b) holds. Then it is clear that $\tilde a\tilde b$ is a lift of $ab$ for which \Cref{Thmlargestsubalg}(b) holds.
    \item We use the K\"unneth formula for shift operators.
    Suppose $a\in \sA^\hbar_{G,\bN}$. Define
\[
\Delta_*\colon(\oh^\hbar_{T^*\check T})_\loc\to (\oh^\hbar_{T^*\check T})_\loc\otimes_{\C[\lt][\hbar]}(\oh^\hbar_{T^*\check T})_\loc
\]
by $\Delta_*(t^\lambda)=t^\lambda\otimes t^\lambda$. Let $Y_i$ be a smooth semiprojective variety that is $G$-equivariant proper over $\bN_i$ for $i=1,2$. Then $X:=X_1\times X_2$ is $G$-equivariant proper over $\bN=\bN_1\oplus\bN_2$. We have
\[
\bbS_{G,X}(a)=(\bbS_{G,X_1}\otimes \bbS_{G,X_2})(\Delta_*(a))\in H^\bullet_{G}(X)[\hbar]=H^\bullet_{G}(X_1)[\hbar]\otimes_{\C[\lt]^W[\hbar]} H^\bullet_{G}(X_2)[\hbar].
\]
The same argument applies in the presence of flavour symmetry. By choosing $T_\xi=\Cxd$ and $X_i=G/B\times \bN$, we conclude that $\bbS_{G,X}(a)\in \sA^\hbar_{G,\bN_1}\otimes_{\C[\lt]^W[\hbar]}\sA^\hbar_{G,\bN_2}$.
\item This is immediate, but the statement can also be interpreted as saying that the flavour symmetry deforms the shift operators.
\item We use that fact that the Weyl group action is regular. 
Suppose $a'\in \sA^{\hbar,\xi}_G$, then it is clear that $\bbS_{G,Y}^\xi(w(a'))=w\circ \bbS_{G,Y}^\xi(a')$ is regular for any $X$ that is $G'$-equivariantly proper over $\bN$, and in particular for $X=G/B\times \bN$ and $X=\bN$.
\item We identify $H^2_G(\pt;\C)$ with $\Hom(H^G_2(\pt;\Z),\C)=\Hom(\Lambda,\C)$. An element $u\in H^2_G(\pt;\C)$ induces a $\C[\lt][\hbar]$-automorphism of $(\oh^\hbar_{T^*\check T})_\loc$ given by
\[
t^\lambda\mapsto \exp(u(\lambda))t^\lambda.
\]
This descends to a $T^!$-action. It induces a deformation of shift operators, equivalently described by scaling the Novikov parameters according to
\[
q^\beta\mapsto \exp(\langle u,\beta \rangle)q^\beta.
\]
As a result, $\bbS_{G,Y}(a)$ is regular if and only if $\bbS_{G,Y}(z\cdot a)$ is regular for all $z\in T^!$. The result then follows.
\end{enumerate}
\end{proof}
We should interpret the above results as suggesting the existence of a 2-category associated with the Coulomb branch $\spec \sA_{G,\bN}$ (or with its quantization $\sA^\hbar_{G,\bN}$), which is yet to be defined (cf. \cite{KRS}). The objects of this 2-category should be given by Lagrangians such as $\spec QH_G(X)$ (or their quantizations $\bbS_{G,X}$). The properties of Coulomb branches should then be reflected in the properties of this 2-category.

\appendix

\section{Stability under the convolution product}\label{AppendixStability}

\begin{proof}[Proof of Proposition~\ref{stableprod}]
    We prove the case for $G$-equivariance. The $G\times \Cxh$-equivariance case can be proved in a similar way.
    Let $\lambda_1, \lambda_2 \in \Lambda^+$ be dominant coweights, and let $d > 0$ be a sufficiently large positive integer such that $\BR^d_{\mu}$ is defined for all $\mu$ satisfying $\mu \leq \lambda_1$, $\mu \leq \lambda_2$, or $\mu \leq \lambda_1 + \lambda_2$. It suffices to show that under the product map (\ref{convolution2}), the image of 
    \[
    \left(z^*H^{G}_\bullet(\BR^d_{\leq \lambda_1})\right) \otimes \left(z^*H^{G}_\bullet(\BR^d_{\leq \lambda_2})\right)
    \]
    is contained in $z^*H^{G}_\bullet(\BR^d_{\leq \lambda_1 + \lambda_2})$, where $z^*: H^{G}_\bullet(\BT^d) \to H^{G}_\bullet(\Gr_G)$ is the Gysin map.

    Consider the following diagram
    \[
    \begin{tikzcd}
    \BR^d_{\leq \lambda_1} \times \BR^d_{\leq \lambda_2} \ar[d, symbol=\subset] 
    & p^{-1}(\BR^d_{\leq \lambda_1} \times \BR^d_{\leq \lambda_2}) \ar[l, "p''", swap] \ar[d, symbol=\subset]
    & Z \ar[l, "j"] \ar[d, symbol=\subset] \\
    \BT^d_{\leq \lambda_1} \times \BT^d_{\leq \lambda_2} 
    & p^{-1}(\BT^d_{\leq \lambda_1} \times \BT^d_{\leq \lambda_2}) \ar[l, "p'", swap] 
    & (G_\ok)_{\leq \lambda_1} \times \BT^d_{\leq \lambda_2} \\
    C_{\leq \lambda_1} \times C_{\leq \lambda_2} \ar[u, "z_1"] 
    & (G_\ok)_{\leq \lambda_1} \times C_{\leq \lambda_2} \ar[l, "p", swap] \ar[u, "z_2"] \ar[r, equal]
    & (G_\ok)_{\leq \lambda_1} \times C_{\leq \lambda_2} \ar[u, "z_3"].
    \end{tikzcd}
    \]
    Here, 
    \[
    Z \vcentcolon= \{(g_1, [g_2, s]) \in (G_\ok)_{\leq \lambda_1} \times \BR^d_{\leq \lambda_2} : g_1 g_2 s \in \bN_\oh\},
    \]
    and $j$ is defined by 
    \[
    j(g_1, [g_2, s]) = (g_1, g_2 s, [g_2, s]).
    \]
    The $G_\oh \times G_\oh$-action on $p^{-1}(\BT^d_{\leq \lambda_1} \times \BT^d_{\leq \lambda_2})$ is given by:
    \[
    (g, g') \cdot (g_1, s_1, [g_2, s_2]) = (g g_1 (g')^{-1}, g' s_1, [g' g_2, s_2]),
    \]
    so $j$ is $G_\oh \times G_\oh$-equivariant. There is a section $\phi$ of (the pullback of) $\BT^d_{\leq \lambda_1}$ over $p^{-1}(\BR^d_{\leq \lambda_1} \times \BR^d_{\leq \lambda_2})$ defined by
    \[
    \phi(g_1, s_1, [g_2, s_2]) = [g_1, s_1 - g_2 s_2],
    \]
    whose vanishing locus is $Z$. Therefore,
    \begin{equation}\label{pullpart}
    \begin{split}
        p^*z_1^*H^{G \times G}_\bullet(\BR^d_{\leq \lambda_1} \times \BR^d_{\leq \lambda_2}) 
        =&z_2^*(p')^*H^{G \times G}_\bullet(\BR^d_{\leq \lambda_1} \times \BR^d_{\leq \lambda_2})\\
        \subset&z_2^*H^{G \times G}_\bullet(p^{-1}(\BR^d_{\leq \lambda_1} \times \BR^d_{\leq \lambda_2}))\\
        \subset& z_3^* H^{G \times G}_\bullet(Z).
    \end{split}
    \end{equation}
    Next, consider the diagram
    \[
    \begin{tikzcd}
    Z' \ar[d, symbol=\subset] \ar[r, "m''"] 
    & \BR^d_{\leq \lambda_1 + \lambda_2} \ar[d, symbol=\subset] \\
    (\Gr_\ok)_{\leq \lambda_1} \times_{G_\oh} \BT^d_{\leq \lambda_2} \ar[r, "m'"] 
    & \BT^d_{\leq \lambda_1 + \lambda_2} \\
    (G_\ok)_{\leq \lambda_1} \times_{G_\oh} C_{\leq \lambda_2} \ar[r, "m"] \ar[u, "z_4"] 
    & \overline{C}_{\lambda_1 + \lambda_2} \ar[u, "z_5"].
    \end{tikzcd}
    \]
    Here, $Z' = Z / G_\oh$ is the image of $Z$ under the canonical map 
    \[
    (\Gr_\ok)_{\leq \lambda_1} \times \BT^d_{\leq \lambda_2} \to (\Gr_\ok)_{\leq \lambda_1} \times_{G_\oh} \BT^d_{\leq \lambda_2}.
    \]
    Hence,
    \begin{equation}\label{pushpart}
    \begin{split}
        m_*(q^*)^{-1}z_3^* H^{G \times G}_\bullet(Z) 
        &= m_*z_4^* H^{G}_\bullet(Z') \\
        &= z_5^*(m')_* H^{G}_\bullet(Z') \\
        &\subset z_5^* H^{G}_\bullet(\BR^d_{\leq \lambda_1 + \lambda_2}).
    \end{split}
    \end{equation}

    Combining \eqref{pullpart} and \eqref{pushpart}, we obtain
    \[
    m_*(q^*)^{-1}p^*z_1^*H^{G\times G}_\bullet(\BR^d_{\leq \lambda_1} \times \BR^d_{\leq \lambda_2}) \subset z_5^* H^{G}_\bullet(\BR^d_{\leq \lambda_1 + \lambda_2}),
    \]
    as desired.
\end{proof}

\section{\texorpdfstring{Universal $G$-torsor}{Universal G-torsor}}
\label{Gtorsorappendix}
In this subsection, we will define in terms of functor of points an ind-scheme $\sE$ equipped with an action of $G_{\C[t^{-1}]}\rtimes\Cxh$ which is a $G$-torsor over $\Gr_G\times\PP^1$. It is best to work in the language of functors of points, since the algebraic loop group $G_\ok$ is highly non-reduced (see \cite[Remark 1.3.10]{Zhu}).

\subsubsection*{Universal $G$-torsors on the loop group}
 Here $G_{\C[t^{-1}]}$ is the fppf sheaf which sends any $\C$-algebra $R$ to $G(R[t^{-1}])$. The space $\sE$ is understood as the universal $G$-torsor.

Let $R$ be a $\C$-algebra, and $\gamma \in G(R\cct)$. By the theorem of \citeauthor{BL} (\cite{BL}), there exists a unique $G$-torsor $\hat{\sE}_{R,\gamma}$ on $\mathbb{P}^1_R$, equipped with trivializations
\begin{align*}
\varphi^\gamma_0 \colon \hat{\sE}_{R,\gamma}|_{\spec R[[t]]} &\xrightarrow{\sim} \spec R[[t]] \times G, \\
\varphi^\gamma_\infty \colon \hat{\sE}_{R,\gamma}|_{\spec R[t^{-1}]} &\xrightarrow{\sim} \spec R[t^{-1}] \times G.
\end{align*}
$\varphi^\gamma_\infty=\gamma\cdot\varphi^\gamma_0$ on $\hat{\sE}_{R,\gamma}|_{\spec R\cct}$. In simple terms, $\hat{\sE}_{R,\gamma}$ is obtained by gluing trivial $G$-torsors over $\spec R[[t]]$ and $\spec R[t^{-1}]$ using $\gamma$ as the transition function.

Moreover, if $f\colon R \to S$ is a $\C$-algebra homomorphism, and $\gamma' \in G(S\cct)$ is the image of $\gamma$ under the map $G_\ok(f)\colon G(R\cct) \to G(S\cct)$, then there is an isomorphism
\begin{equation}\label{cart}
\hat\sE_{S,\gamma'} \cong \hat\sE_{R,\gamma} \times_R S
\end{equation}
such that $\varphi_0^{\gamma'}=\varphi_0^\gamma\times_RS$ and $\varphi_\infty^{\gamma'}=\varphi_\infty^\gamma\times_RS$.

Conversely, if $\mathcal{P}$ is a $G$-torsor over $\PP^1_R$, with trivializations $\varphi_0$ and $\varphi_\infty$ over $\spec R[[t]]$ and $\spec R[t^{-1}]$ respectively, then there exists a unique $\gamma\in G_\ok(R)=G(R\cct)$ such that $\varphi_\infty=\varphi_0$ when restricted to $\spec R\cct$. The following lemma summarizes the above discussion.
\begin{lem}\label{Gtorsor}
The loop group $G_\ok$ represents the functor
\[
R \longmapsto \left\{
\begin{array}{l}
\text{Isomorphism classes of the pair } (\mathcal{P}, \varphi_0,\varphi_\infty): \\
\mathcal{P} \text{ is a } G\text{-torsor over } \PP^1_R,\ \\
\varphi_0 \text{ is a trivialization of }\mathcal{P} \text{ over } \spec R[[t]],\\ \varphi_\infty \text{ is a trivialization of }\mathcal{P} \text{ over } \spec R[t^{-1}]
\end{array}
\right\}.
\]
\end{lem}
By abstract nonsense, there exists a universal bundle $\hat \sE\to G_\ok\times \PP^1$ with trivializations $\varphi_0$ and $\varphi_\infty$ of $\hat \sE$ over $G_\ok\times \spec \C[[t]]$ and $G_\ok \times \spec \C[t^{-1}]$ respectively. We will now describe them explcitly.

We understand schemes as functors from the category $\C\Alg$ of $\C$-algebras to the category $\Set$ of sets, via the functor of points construction. If $f\colon R \to S$ is a $\C$-algebra homomorphism, we let $\hat\sE_{R,\gamma}(S)_f$ denote the preimage of $f \in \spec(R)(S)$ under the natural projection $\hat\sE_{R,\gamma}\to \spec R$. In particular, the isomorphism \eqref{cart} implies that
\begin{equation}\label{cart2}
    \hat\sE_{S, G_\ok(f)(\gamma)}(S)_{\operatorname{id}_S} = \hat\sE_{R,\gamma}(S)_f
\end{equation}
for any $\gamma \in G_\ok(R)$.

Now $\hat\sE\colon \C\Alg \to \mathrm{Sets}$ can be defined as follows. We set
\begin{equation*}
    \hat\sE(R) = \{(\gamma, x) \mid \gamma \in G(R\cct),\ x \in \hat\sE_{R,\gamma}(R)_{\operatorname{id}_R} \}
\end{equation*}
for each $\C$-algebra $R$; and
\begin{equation*}
    \hat\sE(f)(\gamma, x) = \big(G_\ok(f)(\gamma),\ \hat\sE_{R,\gamma}(f)(x)\big)
\end{equation*}
for each $\C$-algebra homomorphism $f\colon R \to S$. Note that $\hat\sE_{R,\gamma}(f)(x) \in \hat\sE_{S, G_\ok(f)(\gamma)}(S)_{\operatorname{id}_S}$ in view of (\ref{cart2}).

\subsubsection*{The \texorpdfstring{$(G_{\C[t^{-1}]} \times G_\oh) \rtimes \Cxh$}{(G_{C[t^{-1}]} x G_O) ⋉ C*}-actions on \texorpdfstring{$\hat\sE$}{Ê}.}
We let $\Cxh$ act on $G_\oh$, $G_\ok$, and $G_{\C[t^{-1}]}$ by loop rotations defined as follows. Let $R$ be a $\C$-algebra. Each element $z \in \Cxh(R) = R^\times$ induces an $R$-algebra automorphism $m_z^*$ of $R[[t]]$ by sending $t$ to $z^{-1}t$. By abuse of notation, we denote the composition
\[
G_\oh(R) = G(R[[t]]) \xrightarrow{G(m_{z}^*)} G(R[[t]]) = G_\oh(R),
\]
also by $m_z^*$.

The action of $\Cxh$ on $G_\oh$ is then given by
\begin{align*}
    \Cxh(R) \times G_\oh(R) &\to G_\oh(R) \\
    (z, g) &\mapsto m_z^{*-1}(g).
\end{align*}
It is notationally more instructive to write $g = g(t)$ and $m_z^{*-1}(g) = g(zt)$. The loop rotation actions on $G_\ok$ and $G_{\C[t^{-1}]}$ are defined similarly.

We write $(G_{\C[t^{-1}]} \times G_\oh) \rtimes \Cxh$ for the semidirect product in which $\Cxh$ acts on $G_{\C[t^{-1}]} \times G_\oh$ via loop rotation. It is clear that $(G_{\C[t^{-1}]} \times G_\oh) \rtimes \Cxh$ acts on $G_\ok$ by
\[
(g(t), h(t), z) \cdot \gamma(t) = g(t)\, \gamma(zt)\, h(t)^{-1}
\]
for any $\C$-algebra $R$, any $(g(t), h(t), z) \in \left(G(R[t^{-1}]) \times G(R[[t]])\right) \rtimes R^\times$, and any $\gamma(t)\in G(R\cct)$. In view of \Cref{Gtorsor}, one can understand the action of $G(R[t^{-1}])$ (resp.\ $G(R[[t]])$) as changing the trivialization $\varphi_\infty$ (resp.\ $\varphi_0$), and $z \in R^\times$ acts via the pullback along $m_z^{-1}\colon \PP^1_R \to \PP^1_R$.

We are going to show that this action lifts to an action of $(G_{\C[t^{-1}]} \times G_\oh) \rtimes \Cxh$ on $\hat\sE$. Let $R$ be a $\C$-algebra, $(g(t), h(t), z) \in (G(R[t^{-1}]) \times G(R[[t]])) \rtimes R^\times$, and let $\gamma(t) \in G_\ok(R)$,  $\gamma'=(g,h,z)\cdot\gamma$. Namely
\begin{equation*}
    \gamma'(t)=g(t)\gamma(zt)h(t)^{-1}.
\end{equation*}
Note that $\varphi_0' = m_z^*h\cdot \varphi_0^{\gamma}$ and $\varphi_\infty' = m_z^*g\cdot \varphi_\infty^{\gamma}$ are local trivializations of $\hat\sE_{R,\gamma}$ over $\spec R[[t]]$ and $\spec R[t^{-1}]$, respectively. One checks immediately that $\varphi_\infty' = m_z^*\gamma'\cdot\varphi_0'$ over $\spec R\cct$. By the theorem of \citeauthor{BL}, there exists a unique isomorphism
\[
\vartheta_{g,h,z}\colon \hat\sE_{R,\gamma} \xrightarrow{\sim} m_z^*\hat\sE_{R,\gamma'},
\]
such that $m_z^*\varphi_0^{\gamma'}\circ \vartheta_{g,h,z} =\varphi_0' $ and $m_z^*\varphi_\infty^{\gamma'}\circ\vartheta_{g,h,z} =\varphi_\infty' $.

Now let $(g'(t), h'(t), z') \in (G(R[t^{-1}]) \times G(R[[t]])) \rtimes R^\times$, then we have
\begin{equation}\label{productvarphi}
(g'(t), h'(t), z') \cdot (g(t), h(t), z) = (g'(t) g(z't),\ h'(t) h(z't),\ z'z).
\end{equation}
We denote the right-hand side of \eqref{productvarphi} by $(g'', h'', z'')$, and write $\gamma'' = (g'', h'', z'') \cdot \gamma \in G(R\cct)$. We claim that
\begin{equation}\label{cocylevarphi}
    m_z^*\vartheta_{g', h', z'} \circ \vartheta_{g, h, z} = \vartheta_{g'', h'', z''}.
\end{equation}

To see this, it suffices to check that both sides of \eqref{cocylevarphi} agree with the unique isomorphism $\vartheta\colon \hat\sE_{R,\gamma} \xrightarrow{\sim}  m_{z'z}^*\hat\sE_{R,\gamma''}$ satisfying
\begin{align*}
    m_{z'z}^*\varphi^{\gamma''}_0\circ\vartheta &= m_{z'z}^*h''\cdot  \varphi_0^\gamma, \\
    m_{z'z}^*\varphi^{\gamma''}_\infty\circ\vartheta &= m_{z'z}^*g''\cdot  \varphi_\infty^\gamma. \\
\end{align*}
This is immediate for the right-hand side of \eqref{cocylevarphi}. For the left-hand side, we compute
\begin{align*}
    m_{z'z}^*\varphi_0^{\gamma''}\circ m_{z}^*\vartheta_{g', h', z'} \circ \vartheta_{g, h, z}&=m_z^*(m_{z'}^*\varphi_0^{\gamma''}\circ\vartheta_{g', h', z'})\circ \vartheta_{g, h, z}\\
    &=m_z^*(m_{z'}^*h'\cdot \varphi_0^{\gamma'})\circ \vartheta_{g, h, z}\\
    &=m_{z'z}^*h'\cdot m_{z}^*\varphi_0^{\gamma'}\circ \vartheta_{g, h, z}\\
    &=m_{z'z}^*h'\cdot m_z^*h\cdot \varphi_0^\gamma\\
    &=m_{z'z}^*h''\cdot \varphi_0^\gamma.
\end{align*}
The other equality $m_{z'z}^*\varphi_\infty^{\gamma''}\circ m_z^*\vartheta_{g', h', z'} \circ\vartheta_{g, h, z} = m_{z'z}^*h''\cdot \varphi_\infty^\gamma$ can be checked in the same way.

Moreover, let $f\colon R \to S$ be a $\C$-algebra homomorphism. We write $\gamma_S = G_\ok(f)(\gamma)$, $g_S = G_{\C[t^{-1}]}(f)(g)$, etc. Then we have
\begin{equation}\label{naturalvarphi}
    \vartheta_{g_S, h_S, z_S} = \vartheta_{g, h, z} \times_R S,
\end{equation}
because both sides agree with the unique isomorphism $\vartheta:\hat\sE_{S,\gamma_S} \xrightarrow{\sim} m_{z_S}^*\hat  \sE_{S,\gamma'_S}$ with $m_{z_S}^*(\varphi^{\gamma'}_0)_S\circ \vartheta=m_{z_S}^*h_S\cdot(\varphi_0^\gamma)_S$ and $m_{z_S}^*(\varphi^{\gamma'}_\infty)_S\circ \vartheta=m_{z_S}^*h_S\cdot(\varphi_\infty^\gamma)_S$

Now we can define the action
\[\Phi_R:(G_\C[t^{-1}](R))\times G_\oh(R))\rtimes \Cxh(R)\times \hat\sE(R)\longrightarrow \hat\sE(R)\]
by
\begin{equation*}\Phi_R(g(t),h(t),z,\gamma(t),x)=(g(t)\gamma(zt)h^{-1}(t),\vartheta_{(g,h,z)}(x)).
\end{equation*}
By \eqref{cocylevarphi}, this defines an action of $(G_\C[t^{-1}](R))\times G_\oh(R))\rtimes \Cxh(R)$ on $\hat\sE(R)$. In view of \eqref{naturalvarphi}, we obtain an action of $(G_\C[t^{-1}]\times G_\oh)\rtimes \Cxh$ on $\hat\sE$.

Since the projection $\hat \sE\to G_\ok$ is clearly $(G_{\C[t^{-1}]}\times G_\oh) \rtimes \Cxh$-equivariant, and $\Gr_G=G_\ok/G_\oh$, 
\begin{equation*}
   \sE=\hat \sE/G_\oh 
\end{equation*}
is a $G$-torsor over $\Gr_G$. Moreover, there is a canonical trivialization $\varphi_\infty$ of $\sE$ over $\Gr_G\times \spec \C[t^{-1}]$. Moreover, there is a remaining action of $G_{\C[t^{-1}]}\rtimes\Cxh$ on $\sE$.

The trivializations $\varphi_0$ and $\varphi_\infty$ induce the isomorphisms
\begin{align*}
    \hat{\sE}|_{{G_\ok} \times \{0\}} \cong G_\ok \times G, \quad
    \hat{\sE}|_{{G_\ok} \times \{\infty\}} \cong G_\ok \times G.
\end{align*}
Since $G_\mathcal{O}$ acts by changing the trivialization $\varphi_0$, these identifications descend to isomorphisms
\begin{equation}\label{fibresofsE}
    \sE|_{\Gr_G \times \{0\}} \cong G_\mathcal{K} \times_{G_\mathcal{O}} G, \quad
    \sE|_{\Gr_G \times \{\infty\}} \cong \Gr_G \times G. 
\end{equation}

\printbibliography
\end{document}